\documentclass[11pt,a4paper]{article}
\usepackage[utf8]{inputenc}

\usepackage{amsmath} % AMS Math Package
\usepackage{amsthm} % Theorem Formatting
\usepackage{amssymb}	% Math symbols such as \mathbb
\usepackage{graphicx}
\usepackage{eucal}
\usepackage{stmaryrd}
\usepackage[shortlabels]{enumitem}
\usepackage{pdfpages}
\usepackage[textwidth=16cm,centering,headheight=24pt]{geometry} 
\usepackage{comment}
\newcommand{\abs}[1]{\left| #1 \right|} % for absolute value
\newcommand{\norm}[1]{\left\| #1 \right\|}
\newcommand{\pd}[2]{\frac{\partial #1}{\partial #2}} 
\newcommand{\R}{\mathbb R}
\newcommand{\C}{\mathbb C}
\newcommand{\N}{\mathbb N}
\newcommand{\M}{\mathbb{M}}

\newcommand{\Z}{\mathbb Z}
\newcommand{\K}{\mathbb K}
\renewcommand{\H}{\mathcal{H}}
\newcommand{\T}{\mathrm{T}}
\renewcommand{\S}{\mathbb{S}}
\newcommand{\F}{\mathbb{F}}

\newcommand{\eps}{\varepsilon}
\newcommand{\ip}[2]{\left\langle #1, \, #2\right\rangle}
\newcommand{\BBB}{\color{blue}}
\newcommand{\RRR}{\color{violet}}

\newcommand{\res}{\mathbin{\vrule height 1.6ex depth 0pt width
0.13ex\vrule height 0.13ex depth 0pt width 1.3ex}}

\DeclareMathOperator{\vol}{vol}
\DeclareMathOperator{\U}{U}

\DeclareMathOperator{\D}{D}
\let\d\relax
\DeclareMathOperator{\d}{d}

\DeclareMathOperator{\Harm}{Harm}
\DeclareMathOperator{\Sk}{Sk}
\DeclareMathOperator{\dist}{dist}

\newtheorem{theorem}{Theorem}
\newtheorem{goal}[theorem]{Goal}
\newtheorem{lemma}[theorem]{Lemma}
\newtheorem{prop}[theorem]{Proposition}
\newtheorem{corollary}[theorem]{Corollary}

\newtheorem{mainthm}{Theorem}           % For main theorems (letters)

\newtheorem{maincorollary}[mainthm]{Corollary}

\theoremstyle{definition}
\newtheorem{definition}{Definition}
\newtheorem{remark}[theorem]{Remark}

\theoremstyle{remark}
\newtheorem{step}{Step}

\title{The Yang-Mills-Higgs functional on complex line bundles:\\
$\Gamma$-convergence and the London equation}
\date{\today}
\author{Giacomo Canevari, Federico Luigi Dipasquale and Giandomenico Orlandi}
\newcommand{\Addresses}{{% additional braces for segregating \footnotesize
  \bigskip
  \footnotesize

  Giacomo~Canevari, \textsc{Dipartimento di Informatica, Universit\`{a} di Verona, Strada Le Grazie 15, 37134 Verona, Italy}\par\nopagebreak
  \textit{E-mail address}: \texttt{giacomo.canevari@univr.it}

  \medskip

  Federico Luigi Dipasquale (Corresponding author), \textsc{Dipartimento di Informatica, Universit\`{a} di Verona, Strada Le Grazie 15, 37134 Verona, Italy}\par\nopagebreak
  \textit{E-mail address}: \texttt{federicoluigi.dipasquale@univr.it}

  \medskip

  Giandomenico Orlandi, \textsc{\textsc{Dipartimento di Informatica, Universit\`{a} di Verona, Strada Le Grazie 15, 37134 Verona, Italy}}\par\nopagebreak
  \textit{E-mail address}: \texttt{giandomenico.orlandi@univr.it}

}}

\begin{document}

\maketitle

\begin{abstract}
We consider the Abelian Yang-Mills-Higgs functional, in the non-self dual scaling, on a complex line bundle over a closed Riemannian manifold of dimension $n\geq 3$. This functional is the natural generalisation of the Ginzburg-Landau model for superconductivity to the non-Euclidean setting. We prove a $\Gamma$-convergence result, in the strongly repulsive limit, on the functional rescaled by the logarithm of the coupling parameter. As a corollary, we prove that the energy of minimisers concentrates on an area-minimising surface of dimension $n-2$, while the curvature of minimisers converges to a solution of the London equation.
\end{abstract}

% \tableofcontents

\section*{Introduction}
\addcontentsline{toc}{section}{Introduction}

Let~$(M, \, g)$ be a smooth, compact, connected, oriented, Riemannian manifold
without boundary, of dimension~$n\geq 3$. Let~$E\to M$ be a Hermitian complex 
line bundle on~$M$, equipped with a (smooth) reference connection~$\D_0$.
For any~$1$-form $A\in W^{1,2}(M, \, \T^*M)$, we denote by~$\D_A := \D_0 - iA$
the associated connection and by~$F_A$ the curvature~$2$-form of~$\D_A$. 
Let~$\eps > 0$ be a small parameter.
For any section~$u\in W^{1,2}(M, \, E)$ of~$E$ and 
any $1$-form~$A\in W^{1,2}(M, \, \T^* M)$,
we consider the \emph{Ginzburg-Landau} or 
\emph{Abelian Yang-Mills-Higgs} functional
\begin{equation} \label{magneticGL}
 G_\eps(u, \, A) %= G_\eps(u, \, A; \, M)
 := \int_{M} \left(\frac{1}{2}\abs{\D_A u}^2
  + \frac{1}{4\eps^2} (1 - \abs{u}^2)^2 
  + \frac{1}{2}\abs{F_A}^2 \right) \vol_g
\end{equation}
In this paper, we prove a convergence result for minimisers in the limit as~$\eps\to 0$:
given a sequence of minimisers~$\{(u_\eps^{\min}, \, A_\eps^{\min})\}$,
the energy density of~$\{(u_\eps^{\min}, \, A_\eps^{\min})\}$ concentrates
%, up to subsequences, 
on an $(n-2)$-dimensional surface~$S_*$, which is area-minimising 
in a distinguished homology class, determined 
by the topology of the bundle
$E\to M$ (i.e., the Poincar\'e-dual to the first
Chern class~$c_1(E)\in H^2(M; \, \Z)$). 
On the other hand, the curvature of~$A_\eps^{\min}$ 
converges to a solution of the London equation, 
with a singular source term carried by~$S_*$. 
%===================================================================
% 2022/06/20: Modifica al fine di enfatizzare maggiormente Theorem A
%===================================================================
%This convergence result for minimisers is deduced as a corollary
%(see Corollary~\ref{cor:London} below) from a suitable $\Gamma$-convergence result
%for the functionals~$G_\eps$ 
%(which is our main result, Theorem~\ref{maingoal:G} below).
This convergence result for minimisers is stated in Corollary~\ref{cor:London} 
and it is deduced from our main result, Theorem~\ref{maingoal:G} below, which 
provides a full $\Gamma$-convergence theorem for the functionals $G_\eps$. Moreover, and in agreement with known results in the Euclidean setting (c.f., e.g., \cite{JerrardSoner-GL, ABO2}), Theorem~\ref{maingoal:G} shows that energy concentration on topological singular sets is not unique to minimisers but is a general feature for sequences $\{(u_\eps ,\,A_\eps)\}$ satisfying a (natural) logarithmic asymptotic bound on the energy.
% In this paper, we prove a $\Gamma$-convergence result for 
% the functionals~$G_\eps$:
% as $\eps \to 0$, the normalised functionals $\frac{G_\eps}{\abs{\log\eps}}$ 
% $\Gamma$-converge, in the topology of flat convergence of Jacobians, to the  
% area functional in codimension two (see Theorem~\ref{maingoal:G} below). 
% {\BBB \textbf{Menzionare London.}}

Functionals of the form~\eqref{magneticGL} were originally
proposed by V.~Ginzburg and L.~Landau in 1950~\cite{GinzburgLandau} as a %phenomenological 
model of %type II
superconductors subject to a magnetic field
(where~$u$ is an order parameter such that~$\abs{u}^2$ is 
proportional to the density of electronic Cooper pairs, 
while~$A$ is the vector potential for the magnetic field).
The theory accounts for most commonly observed effects 
(such as the quantisation of magnetic flux, 
the Meissner effect, and the emergence of Abrikosov vortex lattices,
see~e.g.~\cite{Tinkham}); moreover, it can be justified
as a suitable limit of Bardeen, Cooper and Schrieffer's
microscopic theory~\cite{BCS}.
Ginzburg-Landau functionals, or variants thereof,
arise in other areas of physics --- for instance,
%as models for 
superfluidity 
%and other physical phenomena exhibiting second-order phase transitions
(see e.g.~\cite{DuGunzburgerPeterson})
and %in gauge theories 
particle physics, for~\eqref{magneticGL}
is the Abelian version of Yang-Mills-Higgs 
action functional in gauge theory (see e.g.~\cite{JaffeTaubes}).

Being invariant under gauge transformations is, indeed,
one of the most prominent features of the functional~\eqref{magneticGL}.
In a Euclidean domain~$\Omega\subseteq\R^n$,
the functional~\eqref{magneticGL} reduces to
\begin{equation}\label{classic-GL-magn}
	\mathcal{GL}^{\text{magn}}_\eps(u, \, A) 
	:= \int_\Omega \left(\frac{1}{2} \abs{\nabla u - i A u}^2 
	+ \frac{1}{4\eps^2} (1-\abs{u}^2)^2 + \frac{1}{2} \abs{\d A}^2 \right) \d x,
\end{equation}
where $u \in W^{1,2}(\Omega,\,\C)$ is a complex-valued map and $A \in W^{1,2}(\Omega, \, \T^* \Omega)$ is a real-valued one-form on $\Omega$. 
% In the context of superconductivity, $u$ is an order parameter
% whose squared modulus is proportional to the density of pairs 
% of electrons in the superconducting state (i.e., of Cooper pairs), 
% while~$A$ represents the vector potential of the magnetic field.
% % $* \d A$ ($*$ being the Hodge operator). 
Assuming for simplicity that~$\Omega$ is simply connected,
the invariance of~\eqref{classic-GL-magn} under gauge-transformations
can be expressed as follows:
for any $(u, \, A) \in W^{1,2}(\Omega,\,\C) \times W^{1,2}(\Omega,\,\T^* \Omega)$
and any~$\theta\in W^{2,2}(\Omega, \, \R)$, there holds
\begin{equation} \label{gaugeinv-Eucl}
	\mathcal{GL}_\eps^{\text{magn}}(u, \, A) 
	= \mathcal{GL}_\eps^{\text{magn}}(e^{i\theta}u, \, A + \d\theta).
\end{equation}
The property~\eqref{gaugeinv-Eucl} suggests that Ginzburg-Landau functionals
are naturally set in the context of complex line bundles over a manifold.
Indeed, the gauge group~$\S^1\simeq\mathrm{U}(1)$ acts 
on a pair~$(u, \, A)$ %, as in the right-hand side of~\eqref{gaugeinv-Eucl},
precisely in the same way as transition functions of 
a bundle act (locally) on sections and connection forms.
In this general setting, gauge-invariance takes the following general form:
for any~$(u, \, A)\in W^{1,2}(M, \, E)\times W^{1,2}(M, \, \T^*M)$
and~$\Phi \in W^{2,2}(M, \, \S^1)$, consider the transformation
\begin{equation}\label{gauge-transf}
	(u, \, A ) \mapsto (\Phi u, A - i\,\Phi^{-1} \, \d\Phi),
\end{equation}
where~$\Phi u$ is defined by the fibre-wise action
of the structure group~$\mathrm{U}(1)\simeq\S^1$ on~$E$. Then 
there holds
\begin{equation} \label{gaugeinv}
 %G_\eps(u, \, A) = G_\eps(e^{i\varphi}u, \, A + \d\varphi)
 G_\eps(u, \, A) = G_\eps(\Phi u, \, A - i\,\Phi^{-1} \, \d\Phi)
\end{equation}
%where~$\Phi u$ is defined by the fibre-wise action
%of the structure group~$\mathrm{U}(1)\simeq\S^1$ on~$E$.
Gauge-invariance will play a crucial r\^{o}le in this paper. Physically, each observable quantity must be gauge-invariant and the energy is only one 
of them. For instance, each term in the energy density of $G_\eps$ is gauge-invariant, as well as $F_A$ (whose physical counterpart in superconductivity is, in fact, the magnetic field).

% This analogy suggests 
% to recast the study of Ginzburg-Landau functionals in the
% setting of complex line bundles over a Riemannian manifold 
% (possibly, with additional structure).
Minimisers of gauge-invariant functionals on manifolds, 
such as~\eqref{magneticGL}, include several examples
of physically relevant objects (such as Bogomolnyi monopoles, vortices, instantons, Hermite-Einstein connections on K\"alher manifolds).
Moreover, such objects play an important r\^ole in topology and complex geometry.
For example, Yang-Mills minimisers are used in the classification 
of stable holomorphic vector bundles on complex manifolds
(this is the content of Donaldson, Uhlenbeck and Yau's 
theorem~\cite{Donaldson, UhlenbeckYau}; for generalisations
to Yang-Mills-Higgs minimisers, 
see e.g.~\cite{Bradlow1990, Bradlow1991, Garcia-Prada}).
Functionals such as~\eqref{magneticGL}
may also be coupled to Einstein equations;
this approach is relevant to cosmology, as it provides
a model for gravity-driven spontaneous symmetry breaking~\cite{SpruckYang}.

We are interested in
% the limit as the (inverse) Ginzburg-Landau parameter~$\eps$ tends to zero, 
the asymptotic regime as~$\eps\to 0$, which is known as
the \emph{London limit} within the context of %type~II 
superconductivity, or the \emph{strongly repulsive limit}, 
within the context of particle physics.
This limit is characterised by the emergence of
\emph{topological singularities} ---
the energy of minimisers concentrates,
to leading order, on $(n-2)$-dimensional sets, whose 
global structure depends on the topology of the bundle~$E\to M$.
We are interested in providing a variational characterisation,
in the sense of~$\Gamma$-convergence,
of those singularities %that arise in the limit.
as being area-minimising in a homology class.
% Based on the analogy with the Euclidean setting~\cite{}, 
We expect that this variational characterisation
will provide indications on the dynamics
of singularities arising in the limit of the
corresponding time-dependent models. 
For instance, for a generalisation of~\eqref{magneticGL}
on a Lorentzian manifold, we expect that the energy 
concentrate on time-like relativistic strings or $M$-branes
(see e.g.~\cite{Jerrard2011, BellettiniNovagaOrlandi2010, BellettiniNovagaOrlandi2012}
for the analysis of related problems in Minkowski space-time).
Moreover, we expect the heat flow of~\eqref{magneticGL}
to be related to motion by mean curvature
(see~\cite{BethuelOrlandiSmets-Annals} for the asymptotics
of a non-gauge-invariant problem, with no dependence on~$A$,
in the Euclidean setting).

% The mathematical literature on Ginzburg-Landau functionals is huge.
% Although some results on~\eqref{magneticGL}, 
% \eqref{classic-GL-magn} were soon established
% (see, e.g., \cite{YisonYang1,YisonYang2} and references therein), 
The asymptotic analysis of~\eqref{magneticGL} %, \eqref{classic-GL-magn},
in the limit as~$\eps\to 0$ heavily relies on
analogous results for the simplified Ginzburg-Landau functional, i.e.
% the asymptotic analysis in the limit as~$\eps\to 0$
% was originally carried out on a simplified version
% of~\eqref{classic-GL-magn}, i.e.
\begin{equation}\label{eq:classic-GL}
	\mathcal{GL}_\eps(u) := \int_\Omega \left(\frac{1}{2} \abs{\nabla u}^2 + \frac{1}{4\eps^2}(1-\abs{u}^2)^2 \right) \d x
\end{equation}
The functional~\eqref{eq:classic-GL} was first considered
by Bethuel, Brezis and H\'elein~\cite{BBH}, in case~$\Omega$
is a two-dimensional Euclidean domain. 
$\Gamma$-convergence results for~\eqref{eq:classic-GL},
in case~$\Omega\subseteq\R^n$ with~$n\geq 2$,
were proved in~\cite{JerrardSoner-GL,ABO2}.
These results show that energy concentration on area-minimising surfaces
(of codimension two) is already a feature of the 
simplified functionals~\eqref{eq:classic-GL}
% {\BBB already for the functional~\eqref{eq:classic-GL},}
% the energy of minimisers concentrates on area-minimising surfaces
--- cf. Modica and Mortola's $\Gamma$-convergence result 
in the real-valued case, \cite{ModicaMortola}.
In fact, both the functionals~\eqref{classic-GL-magn}, \eqref{eq:classic-GL}
and the corresponding heat flows
have been extensively studied in Euclidean domains 
(see e.g.~\cite{SandierSerfaty-book, PacardRiviere-book}),
as well as on Riemannian manifolds~\cite{ChenSternberg, Stern2021, Colinet}.

% In the setting of complex line bundles, 
In case~$E$ is the tangent bundle on a closed Riemann surface 
and the connection~$\D_A$ is fixed, asymptotics for minimisers 
and gradient-flow solutions are available in~\cite{IgnatJerrard, AG-Reno}.
Minimisers of~\eqref{magneticGL} in the limit as~$\eps\to 0$
were studied in case the base manifold~$M$ is 
two-dimensional~\cite{Orlandi, Qing}.
%===================================================================
% 2022/06/20: Rimuovere la formula fuori testo, dà troppa enfasi a 
% (Parise)-Pigati-Stern. Porre la menzione del caso self-dual a valle,
% en-passant
%===================================================================
Other results address a \emph{self-dual} variant of the
functional~\eqref{magneticGL}, %that is
%\begin{equation} \label{magneticGL-selfdual}
% G_\eps^{\textrm{self-dual}}(u, \, A) %= G_\eps(u, \, A; \, M)
% := \int_{M} \left(\frac{1}{2}\abs{\D_A u}^2
%  + \frac{1}{4\eps^2} (1 - \abs{u}^2)^2 
%  + \frac{\eps^2}{2}\abs{F_A}^2 \right) \vol_g
%\end{equation}
in which the curvature term $\abs{F_A}^2$ is replaced by $\eps^2 \abs{F_A}^2$
--- see e.g.~\cite{HongJostStruwe} for the two-dimensional case
and~\cite{PigatiStern, ParisePigatiStern} for the higher-dimensional case.
In a way, %(see the discussion in~\cite{PigatiStern} for an accurate discussion and references to the literature), 
self-duality is a sort of additional symmetry in the Yang-Mills-Higgs functionals, %(and the related Euler-Lagrange equations), 
with applications in the theory of minimal surfaces, as explored e.g.~in the recent papers \cite{PigatiStern,Cheng2021,ParisePigatiStern}. 
%On the other hand, 
However, the non self-dual scaling~\eqref{magneticGL} appears to be more closely related to the
original physical motivation of the Ginzburg-Landau functional.
% as a model for superconductivity.
To better explain this point, let us introduce the 
\emph{gauge-invariant Jacobian}~$J(u, \, A)$ of a pair $(u, \, A)$, defined pointwise (when the right-hand-side of~\eqref{pointwise-Jac} below makes sense) as the $2$-form
\begin{equation} \label{pointwise-Jac}
 J(u, \, A)(X, \, Y) := \ip{i\D_{A,X} u}{\D_{A,Y} u}
 + \frac{1}{2} \left(1 - \abs{u}^2\right) F_A(X, \, Y)
\end{equation}
for any smooth test field~$X$, $Y$ on~$M$. Here~$\ip{\cdot}{\cdot}$ is the real-valued scalar product induced by the Hermitian form on~$E$.  
Plainly, $J(u,\, A)$ is well-defined for any $(u, \, A)\in W^{1,2}(M, \, E)\times W^{1,2}(M, \, \T^*M)$ such that~$G_\eps(u, \, A) < +\infty$. In superconductivity, there is a physical observable associated with gauge-invariant Jacobians: the \emph{supercurrent vorticity} (cf., e.g., \cite[Chapter~5]{Tinkham} and \cite[Chapter~1]{JaffeTaubes}).

As it is easy to check, for minimisers 
$(u_\eps^{\min}, \, A_\eps^{\min})$
of the non-self dual functional~\eqref{magneticGL},
the curvatures~$F_\eps := F_{A_\eps^{\min}}$ satisfy the (gauge-invariant) \emph{London equation}
\begin{equation} \label{London}
 -\Delta F_\eps + F_\eps = 2 J_\eps,
\end{equation}
a distinctive feature of superconductivity (see, e.g., \cite[Chapters~1 and~5]{Tinkham}). 
%The right-hand side of~\eqref{London} is the
%(gauge-invariant) Jacobian, $J_\eps := J(u_\eps^{\min}, \, A_\eps^{\min})$,
%which is associated with the \emph{supercurrent vorticity} {\BBB\textbf{[REF]}}.
%Given $(u, \, A)\in W^{1,2}(M, \, E)\times W^{1,2}(M, \, \T^*M)$
%such that~$G_\eps(u, \, A) < +\infty$,
%{\RRR For a pair $(u,\,A)$, the  Jacobian~$J(u, \, A)$ is defined pointwise 
%as the $2$-form
%\begin{equation} \label{pointwise-Jac}
% J(u, \, A)(X, \, Y) := \ip{i\D_{A,X} u}{\D_{A,Y} u}
% + \frac{1}{2} \left(1 - \abs{u}^2\right) F_A(X, \, Y)
%\end{equation}
%for any smooth test field~$X$, $Y$ on~$M$. Here~$\ip{\cdot}{\cdot}$ is the real-valued scalar product induced by the Hermitian form on~$E$.  
%Plainly, $J(u,\, A)$ is well-defined for pairs  $(u, \, A)\in W^{1,2}(M, \, E)\times W^{1,2}(M, \, \T^*M)$ such that~$G_\eps(u, \, A) < +\infty$.}
%{\RRR Notice that Eq.~\eqref{London} is gauge-invariant, as it involves only gauge-invariant quantities.}
% When~$E = M\times\C$ is the trivial bundle and~$\D_A = \d$ is the flat connection,
% $J(u, \, A)$ reduces to the usual Jacobian determinant of~$u$.
In Corollary~\ref{cor:London} below, we show that
 the curvatures~$F_\eps = F_{A^{\min}_\eps}$
converge, up to extraction of a subsequence, to a
solution~$F_*$ of the (limiting) London equation
\begin{equation} \label{lim_London}
 -\Delta F_* + F_* = 2\pi J_*.
\end{equation}
% which is a distinctive feature of superconductivity
% (see Corollary~\ref{cor:London} below).
The right-hand side~$J_*$ of~\eqref{lim_London} is the 
limit of the Jacobians, i.e.~$J_\eps\to \pi J_*$
(cf. Theorem~\ref{maingoal:G}). $J_*$ is a singular measure,
with values in~$2$-forms, carried by an area-minimising set of dimension~$n-2$.
%{\RRR %In superconductivity, $J_*$ is associated with the limiting supercurrent vorticity and $F_*$ with the limiting magnetic field, and the spatial distribution of the latter is uniquely determined by Equation~\eqref{lim_London} (\cite[Chapter~5]{Tinkham}).
%Along the lines of this driving analogy, 
%Moreover, since the convergence $F_\eps \to F_*$ is in a quite strong sense 
%(cf. Corollary~\ref{cor:London}), 
%we may still interpret $F_*$ as the curvature of some suitable connection over a (generalised) bundle over $J_*$.
%Thus, in this sense, we may say that $J_*$ retains some nontrivial \emph{intrinsic} geometry. This interpretation is particularly suggestive when $J_*$ is a smooth $(n-2)$-manifold, in which case $F_*$ is smooth, too (by the ellipticity of \eqref{lim_London}).%, which ensures, even at the level of generality of Corollary~\ref{cor:London}, that $F_*$ is at least of class $W^{1,p}$).
%}

By contrast, minimisers of the self-dual functionals  %of~\eqref{magneticGL-selfdual}
do \emph{not} converge to solutions of~\eqref{lim_London};
in the limit as~$\eps\to 0$, %{\RRR which is merely in the sense of currents, } there holds~$F_* = 2\pi J_*$ instead.
%{\RRR{. In such case, $J_*$ seems to come without any intrinsic (albeit generalised) notion of curvature}}.
Moreover, while minimisers of the self-dual functionals %~\eqref{magneticGL-selfdual}
have uniformly bounded energy as~$\eps\to 0$,
minimisers~$(u_\eps^{\min}, \, A_\eps^{\min})$
of~\eqref{magneticGL} on a non-trivial bundle must
satisfy~$G_\eps(u_\eps^{\min}, \, A_\eps^{\min})\to+\infty$
as~$\eps\to 0$ --- in fact, $G_\eps(u_\eps^{\min}, \, A_\eps^{\min})$
is of the order of~$\abs{\log\eps}$, as we will see below (cf. Remark~\ref{rk:energy-bound-min}). Such a large amount of energy allows for wild oscillations in the phases of the maps $u_\eps$, preventing compactness even in weak topologies in Sobolev spaces and making the proof of Theorem~\ref{maingoal:G} quite challenging, as in fact is for its Euclidean counterpart in \cite{JerrardSoner-GL, ABO2}. %However, as we shall see, we can take advantage of gauge-invariance to discard the ``topologically irrelevant'' oscillations. %To this purpose, it is essential that the London equation at the $\eps$-level, i.e., \eqref{London}, retains the \emph{same form} independently of $\eps$ (cf.~Section~\ref{sec:maingoalG}). This is \emph{not} the case in 
%\cite{ParisePigatiStern, PigatiStern}, where the equation degenerates in the limit. (DIRE MEGLIO E RIFINIRE).}

One might expect that critical points
of~\eqref{magneticGL} converge, in a suitable sense, 
to minimal surfaces, as is the case for~\eqref{BBH-GL} (see~\cite{BethuelBrezisOrlandi}) 
and for the self-dual Yang-Mills-Higgs energies %~\eqref{magneticGL-selfdual} 
(see~\cite{PigatiStern}). 
Here, we prove this result for sequences of minimisers of~\eqref{magneticGL}
(see Corollary~\ref{cor:London} below);
we plan to investigate non-minimising critical 
points in a forthcoming work~\cite{CanevariDipasqualeOrlandiII}.

\paragraph*{Main results.}
We consider sequences~$\{(u_\eps, A_\eps)\}\subset 
W^{1,2}(M,E) \times W^{1,2}(M,\T^*M)$
that satisfy a \emph{logarithmic energy bound}
\begin{equation}\label{eq:bound-Geps}
 \sup_{\eps>0} \frac{G_\eps(u_\eps, \, A_\eps)}{\abs{\log\eps}} < +\infty
\end{equation}
Our main result, Theorem~\ref{maingoal:G} below, 
is a $\Gamma$-convergence result for the rescaled 
functionals~$\frac{G_\eps}{\abs{\log\eps}}$. 
The topology of the $\Gamma$-convergence is defined in terms of
the Jacobian of~$(u_\eps, \, A_\eps)$.
% Given $(u, \, A)\in W^{1,2}(M, \, E)\times W^{1,2}(M, \, \T^*M)$
% such that~$G_\eps(u, \, A) < +\infty$,
% the (gauge-invariant) Jacobian~$J(u, \, A)$ is defined pointwise 
% as the $2$-form
% \begin{equation} \label{pointwise-Jac}
%  J(u, \, A)(X, \, Y) := \ip{i\D_{A,X} u}{\D_{A,Y} u}
%  + \frac{1}{2} \left(1 - \abs{u}^2\right) F_A(X, \, Y)
% \end{equation}
% for any smooth test field~$X$, $Y$ on~$M$. 
% Here~$\ip{\cdot}{\cdot}$ is the real-valued scalar product
% induced by the Hermitian form on~$E$. When~$E = M\times\C$
% is the trivial bundle and~$\D_A = \d$ is the flat connection,
% $J(u, \, A)$ reduces to the usual Jacobian determinant of~$u$.

We denote by~$\star$ the Hodge dual operator, regarded
as a map from~$k$-forms to~$(n-k)$-currents and from~$k$-currents
to~$(n-k)$-forms (upon composition with the natural isomorphism 
between vectors and covectors, induced by the metric
--- see~\eqref{musicalstar} below and
Appendix~\ref{sect:currents} for more details).
We consider a distinguished, non-empty class~$\mathcal{C}$
of integer-multiplicity, rectifiable~$(n-2)$-currents with no boundary.
The class~$\mathcal{C}$ is uniquely defined in terms of 
the topology of the bundle~$E\to M$; more precisely,
$\mathcal{C}\in H_{n-2}(M; \, \Z)$ is Poincar\'e-dual to 
the first Chern class~$c_1(E)\in H^2(M; \, \Z)$ of the bundle
(see~\eqref{C} below for more details).

\begin{mainthm} \label{maingoal:G}
 The following statements hold.
 \begin{enumerate}[label=(\roman*)]
  \item Let~$\{(u_\eps, \, A_\eps)\}_\eps$ be a sequence  
  in~$W^{1,2}(M, \, E)\times W^{1,2}(M, \, \T^*M)$ 
  that satisfies~\eqref{eq:bound-Geps}.
  Then, there exist a (non-relabelled) subsequence and
  a bounded measure~$J_*$, with values in~$2$-forms, such that 
  $J(u_\eps, \, A_\eps)\to \pi J_*$ %as~$\eps\to 0$
  in~$W^{-1,p}(M)$ for any~$p$ with~$1 \leq p < n/(n-1)$ and
  \begin{equation}\label{eq:lower-bound-Geps-intro}
   \pi \abs{J_*}\!(M) \leq
   \liminf_{\eps\to 0}\frac{G_\eps(u_\eps, \, A_\eps)}{\abs{\log\eps}}
  \end{equation}
  Moreover, $\star J_*$ is an integer-multiplicity, rectifiable~$(n-2)$-cycle 
  in the class~$\mathcal{C}$.
  
  \item Let~$S_*$ be an integer-multiplicity, rectifiable~$(n-2)$-cycle 
  in the class~$\mathcal{C}$ and let~$J_* := \star S_*$
  be the dual~$2$-form. 
  Then, there exists a sequence $\{(u_\eps, \, A_\eps)\}_\eps$ 
%   in~$W^{1,2}(M, \, E)\times W^{1,2}(M, \, \T^*M)$
  such that $J(u_\eps, \, A_\eps)\to \pi J_*$ %as~$\eps\to 0$
  in~$W^{-1,p}(M)$ for any~$p$ with~$1 \leq p < n/(n-1)$ and
  \[
   \limsup_{\eps\to 0}\frac{G_\eps(u_\eps, \, A_\eps)}{\abs{\log\eps}} 
   \leq \pi \abs{J_*}\!(M) 
  \]
 \end{enumerate}
\end{mainthm}

\begin{remark}
 Not every pair~$(u, \, A)\in W^{1,2}(M, \, E)\times W^{1,2}(M, \, \T^*M)$
 has finite energy. Conversely, not every Sobolev pair
 $(u, \, A)\in W^{1,1}(M, \, E)\times W^{1,1}(M, \, \T^*M)$ with~$G_\eps(u, \, A)<+\infty$ is such that $(u, \, A)\in W^{1,2}(M, \, E)\times W^{1,2}(M, \, \T^*M)$ (for instance, it may or may not be true that $\abs{\D_0u} = \abs{\D_A u + iAu}\in L^2(M)$). For simplicity, we have chosen to state our results in terms of $W^{1,2}$-pairs, but our results extend to all Sobolev pairs of finite energy, with no significant change in the proofs.
\end{remark}

% Questo Remark è citato appena prima del Lemma 3.4, pertanto se lo rimuoviamo dobbiamo aggiornare brevemente la frase di lancio del lemma
\begin{remark}\label{rk:energy-bound-min}
	Since the class $\mathcal{C}$ is always not empty, a straightforward consequence of Theorem~\ref{maingoal:G} is that the energy of sequences of \emph{minimisers} of $G_\eps$ is \emph{automatically} of order $\abs{\log \eps}$ as $\eps \to 0$. %This is due to the non-self dual scaling of $G_\eps$ and the topology of the bundles under consideration.
\end{remark}

As a corollary, we obtain a variational characterisation
for the limit of a sequence of minimisers. 
Given~$(u, \, A)\in W^{1,2}(M, \, E)\times W^{1,2}(M, \, \T^*M)$,
we denote the (rescaled) energy density of~$(u, \, A)$ as
\begin{equation} \label{mu_eps}
 \mu_\eps(u, \, A) := \frac{1}{\abs{\log\eps}}\left(\frac{1}{2}\abs{\D_A u}^2
  + \frac{1}{4\eps^2} (1 - \abs{u}^2)^2 
  + \frac{1}{2}\abs{F_A}^2 \right) 
\end{equation}

\begin{maincorollary} \label{cor:London}
 Let~$(u_\eps^{\min}, \, A_\eps^{\min})$ be a 
 minimiser of~$G_\eps$ in~$W^{1,2}(M, \, E)\times W^{1,2}(M, \, \T^*M)$.
 Then, there exist bounded measures~$J_*$, $F_*$, with values in~$2$-forms,
 and a (non-relabelled) subsequence such that 
 \[
   J(u_\eps^{\min}, \, A_\eps^{\min})\to \pi J_* 
   \quad \textrm{in } W^{-1,p}(M), \qquad
   F_{A_\eps^{\min}} \to F_* \quad \textrm{in } W^{1,p}(M)
 \]
 for any~$p < n/(n-1)$, and $\mu_\eps(u_\eps^{\min}, \, A_\eps^{\min})\rightharpoonup \pi \abs{J_*}$ weakly$^*$ as measures. 
 Moreover, the current~$\star J_*$ belongs to~$\mathcal{C}$
 and has minimal mass among all the currents in~$\mathcal{C}$, 
 while~$F_*$ satisfies the London equation
 \begin{equation} \label{Londonnonorfana}
   -\Delta F_* + F_* = 2\pi J_*
 \end{equation}
\end{maincorollary}

In other words, the energy of the 
minimisers~$(u_\eps^{\min}, \, A_\eps^{\min})$ concentrate,
to leading order, on the support of a $(n-2)$-dimensional
current, which is dual to the limit Jacobian~$J_*$
and  minimises of the area in its homology
class. Moreover, the limit curvature~$F_*$ is uniquely determined
from~$J_*$, via the London equation. Due to gauge-invariance~\eqref{gaugeinv},
we cannot expect compactness for the minimisers themselves,
$u_\eps^{\min}$, $A_\eps^{\min}$. However, it seems plausible 
that compactness (in suitable norms) should be restored
if we make a suitable choice of the gauge.
We plan to address this point in a forthcoming
paper~\cite{CanevariDipasqualeOrlandiII}. 
%{\BBB In this respect, we point out since now that, as it is easily 
%checked by a quick inspection of the proof, 
%Corollary~\ref{cor:London} still holds (except for the claim of minimality of 
%$\star J_*$ in the class $\mathcal{C}$)
%for any sequence of critical points 
%of \eqref{magneticGL} satisfying the logarithmic energy bound 
%\eqref{eq:bound-Geps}.}

As an intermediate step towards the proof of Theorem~\ref{maingoal:G},
we prove a $\Gamma$-convergence result for a simpler functional
that only depends on the variable~$u$. For any~$u\in W^{1,2}(M, \, E)$, 
we consider
\begin{equation} \label{BBH-GL}
 E_\eps(u) %= E_\eps(u; \, M)
 := \int_{M} \left(\frac{1}{2}\abs{\D_0 u}^2
  + \frac{1}{4\eps^2} (1 - \abs{u_\eps}^2)^2 \right) \vol_g
\end{equation}
The functional~$E_\eps$ is analogous to the non-gauge-invariant
version of the Ginzburg-Landau functional, \eqref{eq:classic-GL}.
We define~$J(u) := J(u, \, 0)$, i.e.~$J(u)$ is the Jacobian of~$u$
with respect to the reference connection~$\D_0$.

\begin{mainthm} \label{maingoal:E}
 The following statements hold.
 \begin{enumerate}[label=(\roman*)]
  \item Let~$\{u_\eps\}_\eps \subset W^{1,2}(M, \, E)$
  be a sequence such that
  \begin{equation}\label{eq:bound-Eeps-intro}
   \sup_{\eps>0} \frac{E_\eps(u_\eps)}{\abs{\log\eps}} < +\infty
  \end{equation}
  Then, there exists a (non-relabelled) subsequence and
  a bounded measure~$J_*$, with values in~$2$-forms, such that 
  $J(u_\eps)\to \pi J_*$ %as~$\eps\to 0$
  in~$W^{-1,p}(M)$ for any~$p$ with~$1 \leq p < n/(n-1)$ and
  \begin{equation}\label{eq:lower-bound-Eeps}
   \pi \abs{J_*}\!(M) \leq 
   \liminf_{\eps\to 0}\frac{E_\eps(u_\eps)}{\abs{\log\eps}}
  \end{equation}
  Moreover, $\star J_*$ is an integer-multiplicity, rectifiable~$(n-2)$-cycle 
  in the class~$\mathcal{C}$.
  
  \item Let~$S_*$ be an integer-multiplicity, rectifiable~$(n-2)$-cycle 
  in the class~$\mathcal{C}$ and let~$J_* := \star S_*$
  be the dual~$2$-form. 
  Then, there exists a sequence $\{u_\eps\}_\eps$ in~$W^{1,2}(M, \, E)$
  such that $J(u_\eps)\to \pi J_*$ %as~$\eps\to 0$
  in~$W^{-1,p}(M)$ for any~$p$ with~$1 \leq p < n/(n-1)$ and
  \[
   \limsup_{\eps\to 0}\frac{E_\eps(u_\eps)}{\abs{\log\eps}}
   \leq \pi \abs{J_*}\!(M) 
  \]
 \end{enumerate}
\end{mainthm}

The proof of Theorem~\ref{maingoal:E} depends heavily on
analogous $\Gamma$-convergence results obtained in the Euclidean
setting~\cite{JerrardSoner-GL, ABO2}. The most delicate point 
is, probably, showing that the limit Jacobian~$J_*$
satisfies~$\star J_*\in\mathcal{C}$, for this is a
global topological property that cannot be deduced
from localisation arguments. The proof of this fact 
is contained in Section~\ref{sect:homology}.

Once Theorem~\ref{maingoal:E} is proved, the upper bound~(ii)
in the statement of Theorem~\ref{maingoal:G} follows immediately.
On the other hand, the proof of the lower bound in Theorem~\ref{maingoal:G}
relies on the London equation, as well as Theorem~\ref{maingoal:E}.
More precisely, given a sequence~$\{(u_\eps, \, A_\eps)\}$ 
that satisfies the logarithmic energy bound~\eqref{eq:bound-Geps},
we construct a sequence of \emph{bounded} sections $\{v_\eps\}$ and a sequence of
$1$-forms~$\{B_\eps\}$ in such
a way that $G_\eps({\RRR{v_\eps}}, \, B_\eps) \leq G_\eps(u_\eps, \, A_\eps)$,
the difference~$J(u_\eps, \, A_\eps) - J({\RRR{v_\eps}}, \, B_\eps)$
is small in~$W^{-1,p}(M)$ and, most importantly,
the curvatures~$F_{B_\eps}$ satisfy the London equation~\eqref{London}. 
%{{\RRR Such equation is elliptic for $F_{B_\eps}$ and gauge-invariant. (Indeed, it involves 
%only gauge-invariant quantities.)
%}}
%~\eqref{London}.
%The forms~$B_\eps$ are obtained by minimising~$G_\eps(u_\eps, \, \cdot)$
Specifically, the sections $v_\eps$ are obtained from the corresponding $u_\eps$ by a truncation argument while the 1-forms~$B_\eps$ are obtained by minimising the auxiliary energy functionals 
\[
	\mathcal{F}_\eps(B) := \int_M \abs{\D_B v_\eps}^2 + \abs{F_B}^2\,\vol_{g}
\]
in a suitable class, while keeping~$v_\eps$ fixed (see Section~\ref{sec:maingoalG} for full details).
%By exploiting the properties of the London equation, {\RRR{---Un po' vago, si potrebbe considerare di precisare un po' il concetto di ``properties'' of the London equation.}}
By exploiting the ellipticity and the gauge-invariance of the London equation, 
and up to a suitable choice of gauge,
we can then show that the difference~$G_\eps(v_\eps, \, B_\eps) - E_\eps(v_\eps)$
is of order smaller than~$\abs{\log\eps}$.
%, which allows us to conclude the proof 
Once this is done, the conclusion follows by Theorem~\ref{maingoal:E}.

The paper is organised as follows. In Section~\ref{sec:preliminaries},
we set some notation and recall a few useful properties of
the Jacobian, $J(u, \, A)$. 
Section~\ref{sec:maingoalE} is devoted to the proof 
of Theorem~\ref{maingoal:E}, while Section~\ref{sec:maingoalG} 
contains the proof of Theorem~\ref{maingoal:G} and Corollary~\ref{cor:London}.
The paper is completed by a series of appendices,
which contain a review of some background material
on Sobolev spaces for sections of a vector bundle and 
currents, as well as the proof of a few technical results.

\numberwithin{equation}{section}
\numberwithin{definition}{section}
\numberwithin{theorem}{section}

\section{Preliminaries}
\label{sec:preliminaries}

We will frequently encounter Sobolev spaces of sections of a bundle.
For the convenience of the reader, we provide the definitions
and state some basic properties in Appendix~\ref{sect:bundles}.
The notation we use is fairly self-explanatory. 
However, we stress that the symbols such as~$W^{1,p}(M, \, E)$
denote Sobolev spaces of~$W^{1,p}$-\emph{sections} of the bundle~$E\to M$.
In case~$E$ is a trivial bundle, $E=M\times\C$ or~$E=M\times\R$,
we write~$W^{1,p}(M, \, \C)$, $W^{1,p}(M, \, \R)$
instead of~$W^{1,p}(M, \, M\times\C)$, $W^{1,p}(M, \, M\times\R)$.

The notation we use for differential forms is rather standard, too.
For instance, we denote as $\#\colon\Lambda^k\T^* M\to\Lambda_k\T M$,
$\flat\colon \Lambda_k\T M \to \Lambda^k\T^* M$ the isometric 
isomorphisms between vectors and forms induced by the metric on~$M$, 
and by $*\colon \Lambda^k\T^*M\to\Lambda^{n-k}\T^* M$
the Hodge dual operator induced by the metric and the orientation of~$M$.
We define an operator
\begin{equation*} %\label{musicalstar}
 \star\colon \Lambda^k\T^*M\to\Lambda_{n-k}\T^* M, \qquad
 \star\colon \Lambda_k\T^*M\to\Lambda^{n-k}\T^* M
\end{equation*}
as follows: for any~$k$-form~$\omega$ and any~$k$-vector~$v$,
\begin{equation} \label{musicalstar}
 \star\omega := (*\omega)^\#, \qquad \star v := *(v^\flat)
\end{equation}
The operator~$\star$ can be extended to
an operator between currents and form-valued distributions;
see Appendix~\ref{sect:Hodge} for details.
We denote by~$\d$, $\d^*$ the exterior differential and  codifferential
of forms, respectively. The codifferential of a~$k$-form~$\omega$ is
defined by~$\d^*\omega := (-1)^{n(k-1)+1} *\d* \, \omega$.

We recall a few basic notions about currents in Appendix~\ref{sect:currents}.
Given a current~$S$, we will write~$\partial S$ for its boundary
(defined as in~\eqref{boundary} below), $\M(S)$ for its mass
(see~\eqref{mass} below) and~$\F(S)$ for its (integer-multiplicity) flat norm
(see~\eqref{flat} below).

It is convenient to revisit the definition of Jacobian, \eqref{pointwise-Jac}.
Given a pair~$(u, \, A)\in W^{1,2}(M, \, E)\times W^{1,2}(M, \, \T^*M)$
such that~$\D_Au \in L^2(M, \, \T^* M\otimes E)$,
we define the \emph{gauge-invariant pre-Jacobian} as the~$1$-form
\begin{equation} \label{magneticpreJac}
 j(u, \, A) := \ip{\D_A u}{i u} 
  = \ip{\D_0 u}{i u} - \abs{u}^2 A
\end{equation}
Under the assumptions above, $j(u, \, A)$ is integrable,
so it makes sense to consider its differential
in the sense of distributions.
We define the (distributional) \emph{gauge-invariant Jacobian} as
\begin{equation} \label{magneticJac}
 J(u, \, A) := \frac{1}{2}\d j(u, \, A) + \frac{1}{2}F_A
\end{equation}
In case the pair~$(u, \, A)$ is smooth,
% $(u, \, A)\in (L^\infty \cap W^{2,2})(M,\,E) \times W^{1,2}(M, \, \T^*M)$ 
% --- or, more generally, if~$G_\eps(u, \, A) < +\infty$ ---
an explicit computation shows that the Jacobian 
as defined by~\eqref{magneticJac} agrees with~\eqref{pointwise-Jac}.
The same remains true if~$(u, \, A)$ satisfies~$G_\eps(u, \, A) <+\infty$,
by a truncation and density argument.
However, \eqref{magneticJac} allows us to define the Jacobian
for a broader class of pairs~$(u, \, A)$
--- for instance, when~$u\in W^{1,1}(M, \, E)\cap L^\infty(M, \, E)$
and~$A\in L^1(M, \, \T^*M)$.

\begin{remark}\label{rk:locality-jac}
Since, for any vector field $X$ on $M$, $(\D_{A, X} u)(p)$ depends only on $X(p)$ and the values of $u$ along any smooth curve representing $X(p)$ \cite[p.~501]{Lee}, it follows that $j(u,\,A)$ is a local operator and commutes with restrictions. Moreover, since the local representations $F_A$ are independent of any chosen local trivialization because $E \to M$ is of rank 1 \cite[Chapter~V,~Remark~3.2]{DeMailly}, we see that also $J(u,A)$ is a local operator and commutes with restrictions. 
\end{remark} 

\begin{remark}\label{rk:diff-jJAB}
    Let~$u\colon M\to E$, $A\colon M\to\T^*M$, $B\colon M\to\T^* M$
    be such that~$J(u, \, A)$, $J(u, \, B)$ are well-defined
    --- for instance, $u\in (W^{1,1}\cap L^\infty)(M, \, E)$, 
    $A\in L^1(M, \, \T^*M)$, $B\in L^1(M, \, \T^*M)$.
	As an immediate consequence of~\eqref{magneticpreJac}--\eqref{magneticJac},
    we have
	\begin{equation}\label{eq:diff-jJAB}
		\begin{aligned}
			& j(u, A) - j(u, B) = (B - A) \abs{u}^2 \\
			& J(u, A) - J(u,B) = \frac{1}{2} \d((A-B) (1-\abs{u}^2)).
		\end{aligned}
	\end{equation}
	In particular, if~$\abs{u}=1$ a.e. then~$J(u, \, A) = J(u, \, B)$.
\end{remark}

The distributional Jacobian 
is continuous with respect to suitable notions of weak convergence.
For the convenience of the reader, we recall a result
in this direction. We denote by~$j(u)$, $J(u)$ the pre-Jacobian
and Jacobian with respect to the reference connection~$\D_0$, i.e.
\begin{equation} \label{reference-Jac}
 j(u) := j(u, \, 0) = \ip{\D_0 u}{i u}, \qquad
 J(u) := J(u, \, 0) = \frac{1}{2}\d j(u) + \frac{1}{2} F_0
\end{equation}
where~$F_0$ is the curvature of the reference connection~$\D_0$.
$j(u)$ and~$J(u)$ are well-defined for any~$u\in (W^{1,p}\cap L^q)(M, \, E)$,
with~$p\in [1, \, +\infty]$ and~$q := p^\prime \in [1, \, +\infty]$
such that $1/p + 1/q = 1$.

\begin{prop} \label{prop:contjac}
 Let~$p\in [1, \, +\infty]$ and~$q := p^\prime \in [1, \, +\infty]$
 be such that $1/p + 1/q = 1$.
 If the sequence~$\{u_\eps\}\subset(W^{1,p}\cap L^q)(M, \, E)$
 is such that $u_\eps\to u$ weakly in~$W^{1,p}(M)$ and strongly in~$L^q(M)$,
 then~$J(u_\eps)\rightharpoonup^* J(u)$ weakly$^*$ in the sense of distributions.
\end{prop}
\begin{proof} 
 Thanks to~\eqref{reference-Jac}, it suffices to show that
 $j(u_\eps)\rightharpoonup^* j(u)$ weakly$^*$ in the sense of distributions.
 This follows from the H\"older inequality.
\end{proof}

In fact, the Jacobian~$J(u, \, A)$ satisfies suitable
continuity properties as a function
of both~$u$ and~$A$ (see e.g.~\cite[Proposition~3.1]{ParisePigatiStern}).

Finally, we introduce a distinguished
(homology) class of~$(n-2)$ currents, $\mathcal{C}\in H_{n-2}(M; \, \Z)$,
as follows. Let~$w\colon M\to E$ be a \emph{smooth} section of~$E$.
We assume that~$w$ is transverse to the zero-section of~$E$;
that means, for any point~$p\in M$ such that~$w(p) = 0$,
the differential~$\d_p w$ induces a \emph{surjective} linear map
\[
 \d_pw\colon \T_p M\to E_p\simeq\C
\]
where~$E_p$ is the fibre of~$E$ at~$p$. Such a section~$w$ exists;
in fact, by Thom's transversality theorem 
(see e.g.~\cite[Theorem~14.6]{BrockerJanich}), any smooth section
can be approximated (e.g., uniformly) by transverse sections.
Transversality implies that the inverse image~$Z := w^{-1}(0)$
is a smooth manifold without boundary, of dimension~$n - 2$.
As both~$M$ and~$E$ are oriented manifolds (the orientation on~$E$
is the one induced by the complex structure on each fibre),
the manifold~$Z$ can be given an orientation, in a natural way
\cite[Proposition~12.7]{BottTu}. As a consequence, 
there is a well-defined $(n-2)$-current~$\llbracket Z\rrbracket$,
carried by~$Z$, with unit multiplicity.
We have~$\partial\llbracket Z\rrbracket = 0$, 
because~$Z$ is manifold without boundary. We define
\begin{equation} \label{C}
 \begin{split}
  \mathcal{C} := \big\{\llbracket Z\rrbracket + \partial R &\colon 
  R \textrm{ is an integer-multiplicity, rectifiable $(n-1)$-current} \\
  &\qquad \textrm{with } \M(R) + \M(\partial R) < +\infty \big\}
 \end{split}
\end{equation}
The class~$\mathcal{C}$ does not depend on the choice
of~$w$ (see~\cite[Proposition~12.8]{BottTu}).
Moreover, by the boundary rectifiability theorem
(see e.g.~\cite[Theorem~6.3]{Simon-GMT}), all the elements 
of~$\mathcal{C}$ are integer-multiplicity, rectifiable 
$(n-2)$-currents with no boundary. 
In the topological jargon, $\mathcal{C}\in H_{n-2}(M; \, \Z)$
is Poincar\'e-dual to the first Chern class~$c_1(E)\in H^2(M; \, \Z)$.
Equivalently, $\mathcal{C}$ is Poincar\'e-dual to the Euler class of~$E$, 
regarded as a real bundle over~$M$ with the orientation
induced by the complex structure.
The following statement motivates our interest in the class~$\mathcal{C}$.

\begin{prop} \label{prop:homologyPPS}
%  Let~$p\in [1, \, +\infty]$. 
 Let~$u\in W^{1,1}(M, \, E)$ be such that
 $\abs{u}=1$ a.e. If~$J(u)$ is a bounded measure, then
 $\frac{1}{\pi}\star J(u)$ is an integer-multiplicity rectifiable
 $(n-2)$-current and
 \[
  \frac{1}{\pi} \star J(u) \in\mathcal{C}
 \]
\end{prop}

Proposition~\ref{prop:homologyPPS} establishes a 
topological relationship between the singular set of a unit-norm section
$M\to E$ and the zero set of a generic \emph{smooth} section.
A similar principle applies to several different contexts;
see, for instance, \cite[Theorem~3.8]{ABO1}
and~\cite[Proposition~5.3]{ABO2} for results in the Euclidean setting,
\cite[Proposition~2.8]{BaldoOrlandi1997} for sections of real line bundles,
\cite{BaldoOrlandi1999} for sections of a principal $\Z$- or~$\Z/p\Z$-bundle,
\cite[Corollary~3.3]{ParisePigatiStern} for sections of
complex line bundles. For the reader's convenience, we reproduce below 
the key steps of the proof in our context, emphasising the topological aspects.

\begin{proof}[Sketch of the proof of Proposition~\ref{prop:homologyPPS}]
 Let~$w$ be a smooth section of~$E\to M$ that is transverse
 to the zero section of~$E$, as above. Let~$u_0 := w/\abs{w}$.
 The section~$u_0$ is well-defined and smooth away from~$Z := w^{-1}(0)$.
 Moreover, $u_0\in W^{1,p}(M, \, E)$ for any~$p\in [1, \, 2)$.
 Indeed, since~$w$ is transverse to the zero section of~$E$,
 the differential~$\d w$ restricted to the normal bundle~$\mathrm{N}Z$ of~$Z$
 is a fibre-wise isomorphism~$\mathrm{N}Z\to E$. As~$M$ is compact,
 it follows that there exists a constant~$C$ such that, for any~$x\in M$,
 \begin{equation} \label{homology1}
  \dist(x, \, Z) \geq C \abs{w(x)}
 \end{equation}
%  The inequality~\eqref{homology1} can be proved, for instance,
%  by an argument by contradiction.
 Therefore,
 \begin{equation} \label{homology2}
  \int_{M} \abs{\D_0 u_0}^p \, \vol_g
   \lesssim \int_{M} \frac{\abs{\D_0 w}^p}{\abs{w}^p} \, \vol_g
   \lesssim \norm{\D_0w}_{L^\infty(M)}^p 
    \int_{M} \frac{\vol_g}{\dist^p(\cdot, \, Z)} 
 \end{equation}
 The integral at the right-hand side is finite for any~$p\in [1, \, 2)$,
 because~$Z$ has codimension~$2$ (see e.g.~\cite[Lemma~8.3]{ABO2}).
 Moreover, for a suitable orientation of~$Z$
 (as in~\cite[Proposition~12.7]{BottTu}), there holds
 \begin{equation} \label{homology3}
  \star J(u_0) = \pi \llbracket Z \rrbracket
 \end{equation}
 As the Jacobian is a local operator (by Remark~\ref{rk:locality-jac}),
 it suffices to check that~\eqref{homology3} is satisfied in an arbitrary
 coordinate neighbourhood~$U\subseteq M$. Due to Remark~\ref{rk:diff-jJAB},
 $J(u_{0|U})$ is equal to the Jacobian of~$u_{0|U}$ with respect to the
 flat connection, $\bar{J} := \frac{1}{2} \d\langle \d u_0, iu_0\rangle$.
 For the latter, we have~$\star\bar{J} = \pi\llbracket Z\rrbracket$ in~$U$
 (this computation is similar to, 
 e.g.,~\cite[Example~3.4]{JerrardSoner-Jacobians}) 
 and hence \eqref{homology3} follows.
 
 Let~$u\in W^{1,1}(M)$ be such that~$\abs{u}=1$ a.e.~in~$M$.
 We define a map~$\Phi\colon M\to\C$ as
 \[
  \Phi := \ip{u}{u_0} + i \, \ip{u}{i u_0}
 \]
 so that~$u = \Phi u_0$. As~$\abs{u}= \abs{u_0} = 1$ a.e.,
 we have~$\abs{\Phi}=1$ a.e.~in~$M$. Moreover, $\Phi\in W^{1,1}(M)$
 and, by explicit computation,
 \begin{equation} 
  \star J(u) = \star \bar{J}(\Phi) + \star J(u_0)
  \stackrel{\eqref{homology3}}{=} \star \bar{J}(\Phi) + \pi \llbracket Z\rrbracket
  \label{homology4}
 \end{equation}
 where~$\bar{J}(\Phi) := \frac{1}{2} \d\ip{\d\Phi}{i\Phi}$.
 Let~$E\subseteq M$ be the set of points where~$\Phi$ is not approximately
 differentiable. Then, $E$ is negligible (with respect to the volume measure)
 and~$M\setminus E$ can be written as a finite union 
 of closed sets, $M\setminus E = \bigcup_{j=1}^{+\infty} F_j$, such that
 $\Phi$ is Lipschitz-continuous on each~$F_j$ (see e.g.~\cite[3.1.18]{Federer}, 
 \cite[Theorems~11 and~12]{Hajlasz2000}).
 As a consequence, for~$\H^1$-a.e.~$y\in\S^1$, the 
 set~$N_y := \Phi^{-1}(y)\setminus E$ is countably~$\H^{n-1}$-rectifiable
 and, by the coarea formula, $\H^{n-1}(N_y) <+\infty$
 (the details of the argument are as in~\cite[Section~7.5]{ABO1}). 
 For a.e.~$y\in\S^1$, we have 
 \begin{equation} \label{homology5}
  \star \bar{J}(\Phi) = \pi \, (-1)^{n-1} \, \partial\llbracket N_y\rrbracket
 \end{equation}
 Indeed, for each coordinate neighboorhood~$U\subseteq M$,
 we have~$\star \bar{J}(\Phi) = \pi \, (-1)^{n-1} \, \partial\llbracket N_y\rrbracket$
 in~$U$ due to~\cite[Theorem~3.8]{ABO1}. As the Jacobian is a local
 operator, \eqref{homology5} follows by partition of unity. 
 In particular, $\frac{1}{\pi}\star \bar{J}(\Phi)$
 is the boundary of an integer-multiplicity, rectifiable~$(n-1)$-current.
 Suppose now that~$J(u)$ is a bounded measure. Then,
 $\bar{J}(\Phi)$ is a bounded measure, too.
 By combining~\eqref{homology4} and~\eqref{homology5},
 we obtain that~$\frac{1}{\pi}\star \bar{J}(\Phi)$ belongs to
 the class~$\mathcal{C}$ defined by~\eqref{C}.
 By Federer and Fleming's boundary rectifiability 
 theorem~\cite[4.2.16(2)]{Federer},
 if~$J(u)$ has finite mass, then~$\frac{1}{\pi}\star J(u)$
 is an integer-multiplicity rectifiable $(n-2)$-current.
 This completes the proof.
\end{proof}

\section{$\Gamma$-convergence for the functional~$E_\eps$}
\label{sec:maingoalE}

The aim of this section is to prove Theorem~\ref{maingoal:E}.

\subsection{Compactness for the Jacobians}
\label{sec:compactnessE}

\subsubsection{A truncation argument}

In some of the arguments below, it will be useful to assume
that sequence~$u_\eps$ is uniformly bounded in~$L^\infty(M, \, E)$.
Fortunately, it is possible to reduce to this case 
by means of a classical truncation argument.
We formulate this argument in terms of 
pairs~$(u_\eps, \, A_\eps)$, for later use.
Let~$\{(u_\eps, \, A_\eps)\} \subset W^{1,2}(M, \, E)\times W^{1,2}(M, \, \T^*M)$ 
be a sequence that satisfies
\begin{equation} \label{trunc:hp}
 \sup_{\eps> 0} \frac{G_\eps(u_\eps, \, A_\eps)}{\abs{\log\eps}} < +\infty
\end{equation}
For each $\eps>0$, we define~$v_\eps\in W^{1,2}(M, \, E)$ as
\begin{equation} \label{trunc:v}
 v_\eps := \begin{cases}
              u_\eps & \textrm{where } \abs{u_\eps} \leq 1 \\[8pt]
              \dfrac{u_\eps}{\abs{u_\eps}}
               & \textrm{where } \abs{u_\eps} > 1.
             \end{cases}
\end{equation}
By construction, $\abs{v_\eps} \leq 1$ in~$M$ for all $\eps > 0$.

\begin{lemma} \label{lemma:truncation}
 The sequence~$\{ v_\eps \}$ defined by~\eqref{trunc:v} satisfies
 \begin{equation} \label{trunc:energy}
  G_\eps(v_\eps, \, A_\eps) \leq G_\eps(u_\eps, \, A_\eps)
 \end{equation}
 for any~$\eps> 0$. Moreover, under the assumption~\eqref{trunc:hp}
 there holds
 \begin{equation} \label{trunc:J}
  J(u_\eps, \, A_\eps) - J(v_\eps, \, A_\eps) \to 0
  \qquad \textrm{in } W^{-1, p}(M) 
  \quad \textrm{for any } p \textrm{ such that } 1 \leq p < \frac{n}{n-1}
 \end{equation}
 as~$\eps\to 0$.
\end{lemma}

In the proof of Lemma~\ref{lemma:truncation} and in the sequel,
we will make repeated use of the following observation.

\begin{lemma} \label{lemma:Du}
 For any~$u\in W^{1,2}(M, \, E)$ and any~$A\in W^{1, 2}(M, \, \T^*M)$,
 there holds
 \[
  \D_A u = \frac{\d(\abs{u})}{\abs{u}} u + \frac{j(u, \, A)}{\abs{u}^2} \, i u
 \]
 a.e.~in the set~$\{u\neq 0\}$.
\end{lemma}
\begin{proof}
 In~$\{u\neq 0\}$, the pair~$(u, \, iu)$ is an orthogonal
 frame for (the real bundle associated with)~$E$. Therefore, 
 keeping in mind that the connection~$\D_A$ 
 is compatible with the metric (i.e., \eqref{eq:metricity}), we obtain
 \[
  \D_A u 
   = \ip{\D_A u}{u} \frac{u}{\abs{u}^2} 
    + \ip{\D_A u}{iu} \frac{iu}{\abs{u}^2}
   = \d\left(\frac{1}{2} \abs{u}^2\right) \frac{u}{\abs{u}^2}
   + j(u, \, A) \frac{iu}{\abs{u}^2}
 \]
 in~$\{u\neq 0\}$.
\end{proof}

\begin{proof}[Proof of Lemma~\ref{lemma:truncation}]
 Towards the proof of~\eqref{trunc:energy}, we observe that
 Lemma~\ref{lemma:Du} implies
 \begin{equation} \label{trunc1}
  \begin{split}
   \D_{A_\eps} v_\eps 
   = - \frac{\d(\abs{v_\eps})}{\abs{v_\eps}^2} u_\eps 
    + \frac{1}{\abs{u_\eps}} \D_{A_\eps} u_\eps
   = j(u_\eps, \, A_\eps) \frac{i u_\eps}{\abs{u_\eps}^3}
  \end{split}
 \end{equation}
 in the set~$\{\abs{u_\eps} > 1\}$. As a consequence,
 \begin{equation} \label{trunc2}
   \abs{\D_{A_\eps} v_\eps}^2 
   = \frac{\abs{j(u_\eps, \, A_\eps)}^2}{\abs{u_\eps}^4}
    \leq \abs{\d(\abs{u_\eps})}^2
     + \frac{\abs{j(u_\eps, \, A_\eps)}^2}{\abs{u_\eps}^2}
   = \abs{\D_{A_\eps} u_\eps}^2
 \end{equation}
 in~$\{\abs{u_\eps} > 1\}$. On the other hand, in~$\{\abs{u_\eps} > 1\}$ we 
 obviously have $0 = (1 - \abs{v_\eps}^2)^2 \leq (1-\abs{u_\eps}^2)^2$, 
 hence the inequality~\eqref{trunc:energy} follows.
 From~\eqref{trunc1}, we obtain $j(v_\eps, \, A_\eps) = j(u_\eps, \, A_\eps) \abs{u_\eps}^{-2}$ in~$\{\abs{u_\eps} > 1\}$ and hence,
 \begin{equation} \label{trunc3}
  \begin{split}
   \norm{j(u_\eps, \, A_\eps) - j(v_\eps, \, A_\eps)}_{L^1(M)}
   &= \norm{\left(\abs{u_\eps}^2 - 1\right)^+ j(v_\eps, \, A_\eps)}_{L^1(M)} \\
   &\leq \norm{\abs{u_\eps}^2 - 1}_{L^2(M)} \norm{\D_{A_\eps} v_\eps}_{L^2(M)} \\
   &\hspace{-.17cm} \stackrel{\eqref{trunc2}}{\leq} 
    \norm{\abs{u_\eps}^2 - 1}_{L^2(M)} \norm{\D_{A_\eps} u_\eps}_{L^2(M)}
  \end{split}
 \end{equation}
 By the energy estimate~\eqref{trunc:hp}, we deduce
 that $\norm{j(u_\eps, \, A_\eps) - j(v_\eps, \, A_\eps)}_{L^1(M)}
 \lesssim \eps \abs{\log\eps}$. Since
 $J(u_\eps, \, A_\eps) - J(v_\eps, \, A_\eps) 
 = \frac{1}{2}\d(j(u_\eps, \, A_\eps) - j(v_\eps, \, A_\eps))$, it follows
 \begin{equation} \label{trunc4}
  \norm{J(u_\eps, \, A_\eps) - J(v_\eps, \, A_\eps)}_{W^{-1,1}(M)}
  \lesssim \eps\abs{\log\eps}
 \end{equation}
 Let us consider the $2$-form~$\sigma(u_\eps, \, A_\eps)$
 defined by $\sigma(u_\eps, \, A_\eps)[X, \, Y] := 
 \ip{i\D_{A_\eps, X} u_\eps}{ \D_{A_\eps, Y}u_\eps}$
 for any smooth vector fields~$X$, $Y$.
 By a direct computation, we have
 \begin{equation} \label{trunc5}
  J(u_\eps, \, A_\eps) = \sigma(u_\eps, \, A_\eps) 
   + \frac{1}{2}(1 - \abs{u_\eps}^2) F_{A_\eps}
 \end{equation}
 From~\eqref{trunc5} and the energy estimate~\eqref{trunc:hp}, we deduce
 \begin{equation} \label{trunc6}
  \norm{J(u_\eps, \, A_\eps) - J(v_\eps, \, A_\eps)}_{L^1(M)}
  \lesssim \abs{\log\eps}
 \end{equation}
 Lemma~\ref{lemma:interpolation} implies
 \begin{equation*} %\label{trunc7}
  \begin{split}
   &\norm{J(u_\eps, \, A_\eps) - J(v_\eps, \, A_\eps)}_{W^{-1,p}(M)} \\
   &\qquad \lesssim \norm{J(u_\eps, \, A_\eps) - J(v_\eps, \, A_\eps)}
    _{W^{-1,1}(M)}^\alpha 
    \norm{J(u_\eps, \, A_\eps) - J(v_\eps, \, A_\eps)}_{L^1(M)}^{1-\alpha}
   \stackrel{\eqref{trunc4}, \, \eqref{trunc6}}{\lesssim}
    \eps^\alpha\abs{\log\eps}
  \end{split}
 \end{equation*}
 where~$\alpha = 1 + n/p - n$.
 Therefore, \eqref{trunc:J} follows.
\end{proof}

\subsubsection{The Jacobians are compact in~$W^{-1,p}(M)$}

Throughout the rest of this section, we consider
a sequence $\{u_\eps \}\subset W^{1,2}(M, \, E)$ that satisfies the energy bound
\begin{equation} \label{compE:hp}
  \sup_{\eps>0} \frac{E_\eps(u_\eps)}{\abs{\log\eps}} < +\infty
\end{equation}
We work with the reference connection~$\D_0$ on~$E$ and
denote by~$j(u_\eps) := j(u_\eps, \, 0)$, $J(u_\eps) := J(u_\eps, \, 0)$
the pre-Jacobian and Jacobian of~$u_\eps$ with respect to~$\D_0$
(see~\eqref{reference-Jac}).
Our first goal is to prove that~$\{ J(u_\eps)\}$ is compact 
in~$W^{-1,p}(M)$, for any~$p < n/(n-1)$.

\begin{lemma} \label{lemma:compactnessE}
 Let~$u_\eps\in W^{1,2}(M, \, E)$ be a sequence that satisfies~\eqref{compE:hp}.
 Then, there exists a (non-relabelled) subsequence and a 
 bounded measure~$J_*$, with values in~$2$-forms, such that
 $J(u_\eps)\to \pi J_*$ in~$W^{-1,p}(M)$ for any~$p$ such that 
 $1 \leq p < n/(n-1)$.
\end{lemma}

In the Euclidean setting, compactness results analogue to
Lemma~\ref{lemma:compactnessE} are well-known~\cite{JerrardSoner-GL, ABO2}.
We will deduce Lemma~\ref{lemma:compactnessE} from its Euclidean counterparts
by a localisation argument.
Let~$U\subset M$ be a smooth, contractible domain in~$M$,
all contained in a coordinate chart of~$M$.
By working in local coordinates, we identify~$U$ with 
a subset of~$\R^n$, equipped with a smooth Riemannian metric~$g$.
Moreover, as~$U$ is contractible,
the bundle~$E\to M$ trivialises over~$U$.
Therefore, we can (and do) identify sections of~$E$ with maps~$U\to\C$.
The reference connection~$\D_0$, restricted to~$U$,
may be written as
\begin{equation} \label{gamma0}
 \D_0 = \d - i\gamma_0 \qquad \textrm{in } U,
\end{equation}
where~$\d$ is the Euclidean connection on~$\R^n$
(that is, $\d u = \d\Re(u) + i\d\Im(u)$ for any~$u\colon U\to\C$)
and~$\gamma_0\colon U\to(\R^n)^*$ is a smooth real-valued $1$-form.
It will be useful to compare the restriction of~$E_\eps$ to~$U$,
that is
\begin{equation} \label{BBH-GL-U}
 E_\eps(u, \, U) := \int_{U} \left(\frac{1}{2} \abs{\D_0 u}^2
   + \frac{1}{4\eps^2}(1-\abs{u}^2)^2 \right) \vol_g.
\end{equation}
with its Euclidean counterpart,
\begin{equation} \label{BBH-GL-bar}
 \bar{E}_\eps(u, \, U) := \int_{U} \left(\frac{1}{2} \abs{\d u}^2
   + \frac{1}{4\eps^2}(1-\abs{u}^2)^2 \right) \d x.
\end{equation}
The integral in~\eqref{BBH-GL-bar}
is taken with respect to the Lebesgue measure~$\d x$, 
not the volume form~$\vol_g$ induced by the metric~$g$.
The functional~\eqref{BBH-GL-bar} is precisely the Ginzburg-Landau
functional, in the simplified form that was 
introduced by Bethuel, Brezis and H\'elein~\cite{BBH}.
Given~$u\in W^{1,2}(U, \, \C)$, we denote the pre-Jacobian and
the Jacobian of~$u$ with respect to the flat connection~$\d$ as
\begin{equation} \label{usualJac}
 \bar{\jmath}(u) := \ip{\d u}{i u}, \qquad 
 \bar{J}(u) := \frac{1}{2} \d\bar{\jmath}(u)
\end{equation}
The quantities~$j(u)$, $\bar{\jmath}(u)$
and~$J(u)$, $\bar{J}(u)$, respectively, are related to each other by
\begin{align}
 j(u) &= \bar{\jmath}(u) - \gamma_0 \abs{u}^2
  \label{compare_preJac} \\
 J(u) &= \bar{J}(u) + \frac{1}{2}\d\left(\gamma_0 (1 - \abs{u}^2)\right)
  \label{compare_Jac}
\end{align}
where~$\gamma_0$ is given by~\eqref{gamma0}.

\begin{remark}\label{rk:locality-jac2}
Notice that, if $u \in W^{1,2}(M,E)$ is a global section of $E$, then $J(u)$ is local and commutes with restrictions by Remark~\ref{rk:locality-jac} (with $A = 0$). In particular, $J(u_{|U}) = J(u)_{|U}$ for any open set $U \subset M$ such that $E$ is trivial over $U$.
\end{remark} 

We recall a well-known compactness result for the Euclidean 
Ginzburg-Landau func\-tion\-al~\eqref{BBH-GL-bar}.
For any~$\alpha\in (0, \, 1)$, we let~$C^{0,\alpha}_0(U)$
be the space of $\alpha$-H\"older continuous
functions~$\varphi\colon U\to\R$ such that~$\varphi=0$ on~$\partial U$.
We let~$(C^{0,\alpha}_0(U))^\prime$ denote the topological 
dual of~$C^{0,\alpha}_0(U)$. 

\begin{theorem}[{\cite{ABO2, JerrardSoner-GL}}] \label{th:GL}
 Let~$(u_\eps)_{\eps>0}$ be a sequence in~$W^{1,2}(U, \, \C)$
 such that
 \[
  \sup_{\eps > 0}\frac{\bar{E}_\eps(u_\eps; \, U)}{\abs{\log\eps}} < +\infty 
 \]
 Then, there exists a bounded measure~$\pi J_*$,
 with values in~$2$-forms, and a (non-relabelled) countable 
 subsequence such that the following properties hold:
 \begin{enumerate}[label=(\roman*)]
  \item $\bar{J}(u_\eps)\to \pi J_*$ in~$(C^{0,\alpha}_0(U))^\prime$
  for any~$\alpha\in(0, \, 1)$.
  
  \item $\star J_*$ is an integer-multiplicity, rectifiable $(n-2)$-current 
  with finite mass, which satisfies~$\partial (\star J_*) = 0$ in~$U$.
  
  \item There holds
  \[
   \pi \abs{J_*}\!(U) \leq \liminf_{\eps\to 0} 
   \frac{\bar{E}_\eps(u_\eps; \, U)}{\abs{\log\eps}}
  \]
 \end{enumerate}
\end{theorem}

\begin{remark} \label{rk:Calpha}
By Sobolev embedding, the linear map 
$\mathcal{R} \colon (C^{0,\alpha}_0(U))^\prime \to W^{-1,p}(U)$,
sending each element $L$ of $(C^{0,\alpha}_0(U))^\prime$ to its restriction 
$L\vert_{W^{1,p'}_0(U)} \in W^{-1,p}(U)$,
is continuous (and surjective) for every
% By Sobolev embedding, $(C^{0,\alpha}_0(U))^\prime$ embeds
% continuously into~$W^{-1,p}(U)$ for
%  any exponent~$p$ whose H\"older conjugate~$p^\prime$ 
%  satisfies $1 - n/p^\prime\geq\alpha$ --- that is, for
 \[
  1 \leq p \leq \frac{n}{\alpha + n - 1} .
 \]
 Therefore, Theorem~\ref{th:GL}
 implies that~$J(u_\eps)\to \pi J_*$ in~$W^{-1,p}(U)$
 for any~$p$ with~$1 \leq p < n/(n-1)$.
\end{remark}

We deduce Lemma~\ref{lemma:compactnessE} from Theorem~\ref{th:GL}.

\begin{proof}[Proof of Lemma~\ref{lemma:compactnessE}]
 Let~$\{u_\eps \}\subset W^{1,2}(M, \, E)$ be a sequence that
 satisfies~\eqref{compE:hp}. We assume throughout the proof that $p \geq 1$.
 
 \setcounter{step}{0}
 \begin{step}[Local convergence]
 Let~$U\subset M$ be a contractible,
 smooth, open subset of~$M$, which we identify with a subset of~$\R^n$.
 Since the manifold~$M$ is compact and smooth, there exists 
 a constant~$C$ (depending on~$M$ only) such that $\d x \leq C \vol_g$.
 Writing~$\d u = \D_0 u +i\gamma_0 u$, we deduce
 \begin{equation} \label{compE1}
  \bar{E}_\eps(u_\eps, \, U)
  \lesssim E_\eps(u_\eps, \, U) + \int_U \abs{\gamma_0}^2 \abs{u_\eps}^2 .
 \end{equation}
 If a sequence~$\{u_\eps \}$ satisfies the energy estimate~\eqref{compE:hp},
 then $\{u_\eps \}$ is uniformly bounded in~$L^2(M)$. Then, \eqref{compE1} implies
 \[
  \sup_{\eps > 0} \frac{\bar{E}_\eps(u_\eps, \, U)}{\abs{\log\eps}} < +\infty.
 \]
 By Theorem~\ref{th:GL}, we may extract a subsequence 
 and find a bounded measure~$J_U$ on~$U$, with values in $2$-forms, such that
 \begin{equation*}
  \bar{J}({u_\eps}_{|U})\to \pi J_U \qquad \textrm{in } W^{-1,p}(U)
  \quad \textrm{for any } p < \frac{n}{n-1}.
 \end{equation*}
 Recall from~\eqref{compare_Jac} that 
 $J({u_\eps}_{|U}) - \bar{J}({u_\eps}_{|U}) 
 = \frac{1}{2}\d(\gamma_0(1 - \abs{u_\eps}^2))$.
 The energy estimate~\eqref{compE:hp} implies $\gamma_0(1 - \abs{u_\eps}^2) \to 0$
 in~$L^2(U)$ as~$\eps\to 0$ and hence,
 $J({u_\eps}_{|U}) - \bar{J}({u_\eps}_{|U})\to 0$ in~$W^{-1,2}(U)$ as~$\eps\to 0$.
 Thus, we have proved
 \begin{equation} \label{compE2}
  J({u_\eps}_{|U})\to \pi J_U \qquad \textrm{in } W^{-1,p}(U)
  \quad \textrm{for any } p < \frac{n}{n - 1}.
 \end{equation}
 \end{step}

 \begin{step}[Covering argument]
   Let $\{U_\alpha\}_{\alpha}$ be any open cover of $M$ in contractible open sets. Over each set $U_\alpha$, $E$ is trivial, as $U_\alpha$ is contractible for any index $\alpha$. Moreover, since $M$ is compact, we can extract from $\{U_\alpha\}$ a finite open cover $\{U_k\}_{k=1}^N$ of $M$ which is still trivializing for $E$. Let $\{ \rho_k \}_{k=1}^N$ be any partition of unity subordinate to the open cover $\{U_k\}$. Obviously, $J(u_\eps) = \sum_{k=1}^N \rho_k J(u_\eps)$ and, on the other hand, $\rho_k J(u_\eps) = \rho_k J(u_\eps)_{|U_k}$ for each $k=1,\dots,N$. According to Remark~\ref{rk:locality-jac2}, $J(u_\eps)$ is local and commutes with restrictions, i.e., $J(u_\eps)_{|U_k} = J({u_\eps}_{|U_k})$ for each $k = 1, \dots, N$. By Step~1, on each open set $U_k$ we have $J(u_\eps)_{|U_k} \to \pi J_{U_k}$ in $W^{-1,p}(U_k)$ for any $p < \frac{n}{n-1}$, where $J_{U_k}$ is a bounded measure on $U_k$ with values in $2$-forms. Thus, as $\eps \to 0$,
   \[
 	J(u_\eps) = \sum_{k=1}^N \rho_k J({u_\eps}_{|U_k}) \to \pi \sum_{k=1}^N \rho_k J_{U_k} =: \pi J_* \qquad \mbox{in } W^{-1,p}(M) \quad \mbox{for any } p < \frac{n}{n-1},
   \]
  where $J_*$ is a bounded measure on $M$ with values into $2$-forms which is well-defined because, as it can be easily checked, it is independent of the chosen partition of unity and of the chosen trivialization.
  \qedhere
 \end{step}
\end{proof}

\subsubsection{Identifying the homology class of~$\star J_*$}
\label{sect:homology}

Our next goal is to show that the limit of the Jacobians,
$\star J_*$, belongs to the homology class~$\mathcal{C}$
(the Poincar\'e dual of the first Chern class of~$E$; see~\eqref{C}).

\begin{prop} \label{prop:homologyJ}
 Let~$\{u_\eps \}\subset W^{1,2}(M, \, E)$ be a sequence that satisfies~\eqref{compE:hp}.
 We extract a (non-relabelled) subsequence in such a way 
 that~$J(u_\eps)\to \pi J_*$ in~$W^{-1,p}(M)$ for any~$p < n/(n-1)$.
 Then, $\star J_*$ is an integer-multiplicity 
 rectifiable current and~$\star J_*\in\mathcal{C}$.
\end{prop}

The proof of Proposition~\ref{prop:homologyJ} relies on the following result.
% The first one is borrowed from Parise, Pigati and Stern~\cite[Corollary~3.3]{ParisePigatiStern}.
% 
% \begin{prop}[\cite{ParisePigatiStern}] \label{prop:homologyPPS}
%  Let~$1 < p < n/(n-1)$. Let~$w\in W^{1,p}(M, \, E)$ be such that
%  $\abs{w}=1$ a.e. If~$J(w)$ is a bounded measure, then
%  $\frac{1}{\pi}\star J(w)$ is an 
%  integer-multiplicity rectifiable current that belongs to~$\mathcal{C}$.
% \end{prop}
% 
% On the other hand, we claim

\begin{lemma} \label{lemma:comp_u}
 Let~$1 < p < n/(n-1)$. Let $\{u_\eps \}\subset W^{1,2}(M, \, E)$ be a sequence
 that satisfies~\eqref{compE:hp}, and let~$J_*$
 be a bounded measure, with values in~$2$-forms, such that
 that~$J(u_\eps)\to \pi J_*$ in~$W^{-1,p}(M)$. Then, there exists
 $w_*\in W^{1,p}(M, \ E)$ such that~$\abs{w_*} = 1$ a.e.~and~$J(w_*) = \pi J_*$.
\end{lemma}

Proposition~\ref{prop:homologyJ} follows immediately from
Proposition~\ref{prop:homologyPPS} and Lemma~\ref{lemma:comp_u}.
It only remains to prove Lemma~\ref{lemma:comp_u}. We will need
an auxiliary result, again borrowed from~\cite{ParisePigatiStern}.
We denote by~$\Harm^1(M)$ the space of harmonic~$1$-forms on~$M$
and by~$\vol_{\S^1}$ the volume form of~$\S^1$.

\begin{remark}\label{rk:pullback-Phi}
	If $\Phi \in W^{1,1}(M,\,\S^1)$, then $\Phi^*(\vol_{\S^1})$ has a pointwise a.e. meaning, and we have the pointwise a.e. equality
	\[
		\Phi^*(\vol_{\S^1}) = -i \Phi^{-1} \d \Phi.
	\]
\end{remark}

\begin{lemma} \label{lemma:choiceofgauge}
 For any~$\varphi\in W^{1,2}(M, \, \R)$ and~$\xi\in\Harm^1(M)$,
 there exist a map~$\Phi\colon M\to\S^1$, as regular as $\varphi$, and a form~$\bar{\xi}\in\Harm^1(M)$
 such that
 \begin{gather}
  \Phi^*(\vol_{\S^1}) = -\d\varphi - \bar{\xi} \label{gauge-pullback} \\
  \norm{\xi - \bar{\xi}}_{L^2(M)} \leq C_M \label{gauge-bdd}
 \end{gather}
 where~$C_M > 0$ is a constant that depends only on~$M$.
\end{lemma}

\begin{remark} \label{rk:choiceofgauge}
 By Hodge theory, the space~$\Harm^1(M)$ has finite dimension.
 Therefore, the difference~$\xi-\bar{\xi}$ is bounded not only in~$L^2(M)$,
 but also in any other norm.
\end{remark}

\begin{proof}
	The proof follows very closely the argument of \cite[Lemma~3.4]{ParisePigatiStern}. Since, however, that lemma is designed to deal with a slightly different situation and cannot be used directly in our case, we provide full details for reader's convenience.
	
	We notice that, for all smooth maps $f,g : M \to \S^1$, the map $\phi := fg : M \to \S^1$ is still smooth, and we can pull back $\vol_{\S^1}$ by $\phi$. Next, since $\S^1$ is an Abelian Lie group with an invariant volume form, it holds $\phi^*(\vol_{\S^1}) = f^*(\vol_{\S^1}) + g^*(\vol_{\S^1})$. For any smooth function $\psi: M \to \R$ and $f : M \to \S^1$ harmonic, if we set $g := e^{i \psi}$, we have $\phi = f e^{i\psi}$ and, by the previous formula and the translational invariance of $\vol_{\S^1}$, we end up with $\phi^*(\vol_{\S^1}) = f^*(\vol_{\S^1}) + \d\psi$. Moreover, $f^*(\vol_{\S^1})$ is a harmonic one-form on $M$ (because $f : M \to \S^1$ is harmonic, see \cite[Example~4.2.6]{BairdWood}). Following \cite{ParisePigatiStern} and \cite[Example~3.3.8]{BairdWood}, we observe that $f$ can be chosen so that $\| \xi - f^*(\vol_{\S^1}) \|_{L^\infty(M)} \leq C(M)$, where $C(M) > 0$ is a constant depending only on $M$. Therefore, if $\varphi$ is smooth, the conclusion follows setting $\psi = -\varphi$, $\Phi = e^{-i\varphi} f$ and $\bar{\xi} = f^*(\vol_{\S^1})$. If $\varphi$ is only of class $W^{1,2}(M,\R)$, the conclusion is however still true thanks to a standard density argument.
	Finally, by definition, it easily seen that $\Phi$ is as regular as $\varphi$.   
\end{proof}

%{\BBB Proof of Lemma~\ref{lemma:choiceofgauge}? At least cite~\cite[Lemma~4.3]{ParisePigatiStern} properly\ldots} Now, it is proved.

\begin{proof}[Proof of Lemma~\ref{lemma:comp_u}]
 We can assume without loss of generality that~$\{u_\eps \}$
 is bounded in~$L^\infty(M)$, independently of~$\eps$
 --- for otherwise, we replace each map~$u_\eps$ by the truncated map~$v_\eps$
 defined in~\eqref{trunc:v} and apply Lemma~\ref{lemma:truncation}
 (with~$A_\eps = 0$). Then, it follows
 that~$j(u_\eps)\in L^2(M, \, \T^* M)$ and that
 \begin{equation} \label{compu0}
  \norm{j(u_\eps)}_{L^2(M)} \leq 
  \norm{\D_0u_\eps}_{L^2(M)} \norm{u_\eps}_{L^\infty(M)}
  \lesssim \abs{\log\eps}^{1/2}
 \end{equation}
 due to the energy estimate~\eqref{compE:hp}.
 We are going to construct a sequence of maps of the form
 \[
  w_\eps = \rho_\eps \, \Phi_\eps u_\eps
 \]
 for suitable 
 functions~$\rho_\eps\colon M\to\R$
 and maps~$\Phi_\eps\colon M\to\S^1$, such that~$J(w_\eps)\to \pi J_*$
 in~$W^{-1,p}(M)$, $\abs{w_\eps}\to 1$ a.e. and
 $w_\eps$ is \emph{bounded} in~$W^{1,p}(M)$.
 Then, we will obtain a map~$w_*$ with the desired properties
 by passing to the (weak) limit in the~$w_\eps$'s.
 We split the proof into several steps.
 
 \setcounter{step}{0}
 \begin{step}
  We consider the Hodge decomposition of~$j(u_\eps)$
  --- that is, we write (in a unique way)
  \begin{equation} \label{compu11}
   j(u_\eps) = \d\varphi_\eps + \d^*\psi_\eps + \xi_\eps 
  \end{equation}
  for some (co-exact) form~$\varphi_\eps\in W^{1,2}(M, \, \R)$,
  some (exact) form~$\psi_\eps\in W^{1,2}(M, \, \Lambda^2\T^*M)$, 
  and some~$\xi_\eps\in\Harm^1(M)$.
  This decomposition (recalled in Proposition~\ref{prop:LpHodgeDec}) is orthogonal in~$L^2(M)$.
  %{\color{blue}{\cite[\textbf{Add precise reference}]{Morrey}}}.
 % We can assume without loss of generality that~$\psi_\eps$
 % is exact (otherwise, we replace~$\psi_\eps$ with its orthogonal 
 % projection onto the subspace of exact~$2$-forms).
 % {\color{blue}\textbf{SISTEMARE}}
  Then, $\psi_\eps$ is closed and~$L^2$-orthogonal to all harmonic~$2$-forms.
  By taking the differential in~\eqref{compu11}, we obtain
  \begin{equation} \label{compu12}
   -\Delta\psi_\eps = \d\d^*\psi_\eps = \d j(u_\eps)
   = 2J(u_\eps) - 2F_0
  \end{equation}
  where~$F_0$ is the curvature of the reference connection~$\D_0$.
  By Lemma~\ref{lemma:compactnessE}, we know that~$J(u_\eps)$
  in bounded in~$W^{-1,p}(M)$, independently of~$\eps$.
  By applying Lemma~\ref{lemma:elliptic_reg_bis}, we deduce
  \begin{equation} \label{compu13}
   \norm{\psi_\eps}_{W^{1,p}(M)}
    \lesssim \norm{J(u_\eps) - F_0}_{W^{-1,p}(M)} \leq C_p
  \end{equation}
  for some constant~$C_p$ that depends on~$p$ 
  (and on $F_0$), but not on~$\eps$.
 \end{step}
 
 \begin{step}
  In this step, we construct suitable maps~$\Phi_\eps\colon M\to\S^1$,
  in such a way that~$j(\Phi_\eps u_\eps)$ is bounded in~$L^p(M)$.
  We do so by applying Lemma~\ref{lemma:choiceofgauge}.
  By Lemma~\ref{lemma:choiceofgauge}, for each $\eps > 0$
  there exist a map~$\Phi_\eps\colon M\to\S^1$ and a 
  form~$\bar{\xi}_\eps\in\Harm^1(M)$ such that
  \begin{gather}
   \Phi_\eps^*(\vol_{\S^1}) = -\d\varphi_\eps - \bar{\xi}_\eps \label{compu21} \\
   \norm{\xi_\eps - \bar{\xi}_\eps}_{L^2(M)} \leq C_M \label{compu22}
  \end{gather}
  where~$C_M$ is a constant depending only on~$M$.
  We consider the section~$\Phi_\eps u_\eps$. We have
  \begin{equation*} 
   \begin{split}
    \norm{\D_0(\Phi_\eps u_\eps)}_{L^2(M)}
    \lesssim \norm{\D_0 u_\eps}_{L^2(M)} + \norm{\Phi^*_\eps(\vol_{\S^1})}_{L^2(M)}
    \stackrel{\eqref{compu21}}{\lesssim} 
     \norm{\D_0 u_\eps}_{L^2(M)} + \norm{\d\varphi_\eps + \bar{\xi}_\eps}_{L^2(M)}
   \end{split}
  \end{equation*}
  As the decomposition in~\eqref{compu11} is orthogonal in~$L^2$,
  we obtain
  \begin{equation} 
   \begin{split}
    \norm{\d\varphi_\eps + \bar{\xi}_\eps}_{L^2(M)}
    \lesssim \norm{j(u_\eps)}_{L^2(M)} 
     + \norm{\xi_\eps -\bar{\xi_\eps}}_{L^2(M)}
    \stackrel{\eqref{compu0}, \, \eqref{compu22}}{\lesssim} \abs{\log\eps}^{1/2}
   \end{split}
   \label{compu27}
  \end{equation}
  and hence, recalling the energy estimate~\eqref{compE:hp},
  \begin{equation} \label{compu28}
   \begin{split}
    \norm{\D_0(\Phi_\eps u_\eps)}_{L^2(M)}
    \lesssim \abs{\log\eps}^{1/2}
   \end{split}
  \end{equation}
  We claim that
  \begin{equation} \label{compu24}
   \norm{j(\Phi_\eps u_\eps)}_{L^p(M)} \leq C_p
  \end{equation}
  for some constant~$C_p$ depending on~$p$, $M$ and $\D_0$, but not on~$\eps$.
  Indeed, due to~\eqref{compu11} and~\eqref{compu21}, we have
  \begin{equation} \label{compu23}
   \begin{split}
    j(\Phi_\eps u_\eps) = j(u_\eps) + \Phi_\eps^*(\vol_{\S^1}) \abs{u_\eps}^2
    = \d^*\psi_\eps + \left(1 - \abs{u_\eps}^2\right)
    \left(\d\varphi_\eps + \bar{\xi}_\eps\right)
    + \xi_\eps - \bar{\xi}_\eps
   \end{split}
  \end{equation}
  Let~$q > 2$ be such that $1/p = 1/q + 1/2$. 
  As~$\norm{u_\eps}_{L^\infty(M)} \leq C$, by generalised 
  H\"{o}lder's inequality and interpolation we obtain
  \begin{equation*} 
   \begin{split}
    \norm{\left(1 - \abs{u_\eps}^2\right)
     \left(\d\varphi_\eps + \bar{\xi}_\eps\right)}_{L^p(M)}
    &\leq \norm{1 - \abs{u_\eps}^2}_{L^q(M)} 
     \norm{\d\varphi_\eps + \bar{\xi}_\eps}_{L^2(M)}  \\
    & \leq \norm{ 1- \abs{u_\eps}^2}_{L^2(M)}^{2/q} \norm{1-\abs{u_\eps}^2}_{L^\infty(M)}^{1-2/q} \norm{\d\varphi_\eps + \bar{\xi}_\eps}_{L^2(M)}\\
    & \leq \norm{ 1-\abs{u_\eps}^2}_{L^2(M)}^{2/q} (1+C)^{1-2/q} \norm{\d\varphi_\eps + \bar{\xi}_\eps}_{L^2(M)} \\
    &\lesssim \norm{1 - \abs{u_\eps}^2}_{L^2(M)}^{2/q} 
     \norm{\d\varphi_\eps + \bar{\xi}_\eps}_{L^2(M)}, 
   \end{split}
  \end{equation*}
  %where the {\BBB\textbf{hidden constant (CHECK)}} 
  up to a multiplicative constant which
  depends only on $M$ and $p$. Hence,
  recalling~\eqref{compE:hp} and~\eqref{compu27},
  \begin{equation} 
   \norm{\left(1 - \abs{u_\eps}^2\right)
    \left(\d\varphi_\eps + \bar{\xi}_\eps\right)}_{L^p(M)}
    \lesssim \eps^{2/p- 1}\abs{\log\eps}^{1/p},
   \label{compu25}
  \end{equation}
  up to a constant depending only on $M$ and $p$.
  Now~\eqref{compu24} follows by~\eqref{compu13}, \eqref{compu22}, 
  \eqref{compu23}, \eqref{compu25} and Remark~\ref{rk:choiceofgauge}.
  As a byproduct of~\eqref{compu25}, we obtain
  \begin{equation} 
   \begin{split}
    J(\Phi_\eps u_\eps) - J(u_\eps) 
    &= \frac{1}{2}\d\left(j(\Phi_\eps u_\eps) - j(u_\eps)\right) \\
%     &\hspace{-.71cm}\stackrel{\eqref{compu11}, \, \eqref{compu23}}{=} 
    &=\frac{1}{2}\d\left(\left(1 - \abs{u_\eps}^2\right)
     \left(\d\varphi_\eps + \bar{\xi}_\eps\right)\right)
     \to 0 \qquad \textrm{in } W^{-1,p}(M).
   \end{split} 
   \label{compu26}
  \end{equation}
 \end{step}
 
 \begin{step}
  The estimate~\eqref{compu24} is not enough to guarantee that~$\Phi_\eps u_\eps$
  is bounded in~$W^{1,p}(M)$: we also need to control the differential
  of~$\abs{u_\eps}$. Although we cannot make sure 
  that~$\norm{\d\abs{u_\eps}}_{L^p(M)}$ is bounded, in general
  (see Remark~\ref{rk:compactnessu} below), we can construct
  suitable functions~$\rho_\eps\colon M\to\R$
  so that $w_\eps := \rho_\eps \, \Phi_\eps u_\eps$
  satisfies the desired estimates.
  
  For any~$\eps>0$, we take a smooth nonnegative
  function~$f_\eps\colon\R\to\R$ such that 
  \begin{gather}
   f_\eps(t) = 1 \quad \textrm{if } \abs{t - 1} \geq 2\eps^{1/2}, \qquad 
   f_\eps(t) = \frac{1}{t} \quad \textrm{if } \abs{t - 1} \leq \eps^{1/2} \label{compu31} \\
   \norm{f^\prime_\eps}_{L^\infty(\R)} \leq C \label{compu32}
  \end{gather}
  We define
  \begin{equation} \label{compu33}
   \rho_\eps := f_\eps(\abs{u_\eps}), \qquad
   w_\eps := \rho_\eps \, \Phi_\eps u_\eps 
  \end{equation}
  We have $j(w_\eps) = \rho_\eps^2 \, j(\Phi_\eps u_\eps)$. 
  Moreover, $\rho_\eps\to 1$ uniformly as~$\eps\to 0$,
  because of~\eqref{compu31}, \eqref{compu32}. 
  Therefore, from~\eqref{compu24} we deduce
  \begin{gather}
   \norm{j(w_\eps)}_{L^p(M)} \leq C_p \label{compu34} \\
   J(w_\eps) - J(\Phi_\eps u_\eps) \to 0 \qquad 
    \textrm{in } W^{-1,p}(M) \label{compu35}
  \end{gather}
  as~$\eps\to 0$. In particular, \eqref{compu26}, \eqref{compu35}
  and the fact that $J(u_\eps)\to \pi J_*$ in~$W^{-1,p}(M)$ imply
  \begin{equation} \label{compu36}
   J(w_\eps) \to \pi J_* \qquad \textrm{in } W^{-1,p}(M)
  \end{equation}
  as~$\eps\to 0$. We estimate the~$L^2$-norm of~$\D_0w_\eps$.
  By explicit computation, we have
  \begin{equation*} 
   \begin{split}
    \norm{\D_0 w_\eps}_{L^2(M)} \lesssim
    \norm{f^\prime_\eps}_{L^\infty(\R)} \norm{\d(\abs{u_\eps})}_{L^2(M)}
    + \norm{\D_0(\Phi_\eps u_\eps)}_{L^2(M)} 
   \end{split}
  \end{equation*}
  The second term at the right-hand side is bounded by~\eqref{compu28}.
  To estimate the other term, we recall~\eqref{compu32}
  and observe that~$\norm{\d(\abs{u_\eps})}_{L^2(M)}\leq 
  \norm{\D_0u_\eps}_{L^2(M)}$, by Lemma~\ref{lemma:Du}. Then,
  \begin{equation} \label{compu37}
   \begin{split}
    \norm{\D_0 w_\eps}_{L^2(M)} \lesssim \abs{\log\eps}^{1/2}
   \end{split}
  \end{equation}
%   Finally, we estimate the~$L^2$-norm of~$1 - \abs{w_\eps}^2$.
%   As we have both~$\abs{u_\eps} \leq 1$ and~$\abs{w_\eps}\leq 1$ a.e.~in~$M$,
%   we obtain
%   \[
%    \begin{split}
%     \norm{1 - \abs{w_\eps}^2}_{L^2(M)} 
%     \lesssim \norm{1 - \abs{u_\eps}^2}_{L^2(M)}
%      + \norm{\abs{u_\eps} - \abs{w_\eps}}_{L^2(M)}
%     \lesssim \norm{1 - \abs{u_\eps}^2}_{L^2(M)}
%      + \norm{f_\eps - 1}_{L^\infty(\R)}
%    \end{split}
%   \]
%   Now the energy estimate~\eqref{compE:hp} and~\eqref{compu32} give 
%   \begin{equation} \label{compu37}
%    \norm{1 - \abs{w_\eps}^2}_{L^2(M)} 
%     \lesssim \eps\abs{\log\eps}^{1/2} + \eps^{1/2}
%     \lesssim \eps^{1/2}
%   \end{equation}
  Finally, using~\eqref{compu32} and the fact
  that~$\abs{u_\eps} \to 1$ in~$L^2(M)$
  (due to our assumption~\eqref{compE:hp}), we deduce
  \begin{equation} \label{compu38}
   \abs{w_\eps}\to 1 \qquad \textrm{in } L^2(M)
  \end{equation}
  as~$\eps\to 0$.
 \end{step}
 
 \begin{step}
  Eventually, we will show that the sequence~$w_\eps$
  is bounded in~$W^{1,p}(M)$. As an intermediate step, we prove that
  \begin{equation} \label{compu41}
   \norm{\d(\abs{w_\eps})}_{L^p(M)} \to 0 \qquad \textrm{as } \eps\to 0.
  \end{equation}
  Indeed, let
  \[
   S_\eps := \left\{x\in M\colon 
    \abs{\abs{u_\eps(x)} - 1} \geq \eps^{1/2}  \right\} 
  \]
  By construction, $\abs{w_\eps} = \abs{u_\eps} f_\eps(\abs{u_\eps})$
  in~$M$ and~$\abs{w_\eps} = 1$ in~$M\setminus S_\eps$, so
  \[
   \norm{\d(\abs{w_\eps})}_{L^p(M)}
   \leq \left(\norm{f_\eps}_{L^\infty(\R)} 
    + \norm{f^\prime_\eps}_{L^\infty(\R)} \right) 
    \norm{\d(\abs{u_\eps})}_{L^p(S_\eps)} 
   \stackrel{\eqref{compu32}}{\lesssim} \norm{\d(\abs{u_\eps})}_{L^p(S_\eps)}
  \]
  Lemma~\ref{lemma:Du}, the H\"older inequality 
  and the energy estimate~\eqref{compE:hp} imply
  \begin{equation} \label{compu42}
   \norm{\d(\abs{w_\eps})}_{L^p(M)}
   \lesssim \norm{\D_0 u_\eps}_{L^p(S_\eps)}
%    \lesssim \left(\vol_g(S_\eps)\right)^{1/p - 1/2}
%     \norm{\D_0 u_\eps}_{L^2(S_\eps)}
   \lesssim \left(\vol_g(S_\eps)\right)^{1/p - 1/2}
    \abs{\log\eps}^{1/2}
  \end{equation}
  On the other hand, $(\abs{u_\eps}^2 - 1)^2\geq\eps$ on~$S_\eps$,
  so the energy estimate~\eqref{compE:hp} implies
  \begin{equation} \label{compu43}
   \vol_g(S_\eps)\lesssim \eps\abs{\log\eps}
  \end{equation}
  Combining~\eqref{compu42} with~\eqref{compu43}, we obtain~\eqref{compu41}.
 \end{step}

 \begin{step}
  We claim that
  \begin{equation} \label{compu51}
   \norm{\D_0 w_\eps}_{L^p(M)} \leq C_p
  \end{equation}
  Once~\eqref{compu51} is proved, the lemma follows. 
  Indeed, if~\eqref{compu51} holds, then we extract a subsequence
  such that $w_\eps\rightharpoonup w_*$ weakly in~$W^{1,p}(M)$.
  The limit~$w_*\in W^{1,p}(M, \, E)$ satisfies~$J(w_*) = \pi J_*$, 
  due to~\eqref{compu36} and Proposition~\ref{prop:contjac}.
  Moreover, the estimate~\eqref{compu38} implies
  that~$\abs{w_*} = 1$, as required.
  
  We proceed to the proof of~\eqref{compu51}.
  Let
  \[
   T_\eps := \left\{x\in M\colon \abs{w_\eps(x)}\leq \frac{1}{2}\right\} 
   = \left\{x\in M\colon \abs{u_\eps(x)}\leq \frac{1}{2}\right\}
  \]
  On the one hand, Lemma~\ref{lemma:Du} gives
  \begin{equation} 
   \begin{split}
    \norm{\D_0 w_\eps}_{L^p(M\setminus T_\eps)}
    \lesssim \norm{\d(\abs{w_\eps})}_{L^p(M\setminus T_\eps)} 
     + \norm{j(w_\eps)}_{L^p(M\setminus T_\eps)}
    \stackrel{\eqref{compu34}, \, \eqref{compu41}}{\lesssim} C_p
   \end{split}
   \label{compu52}
  \end{equation}
  On the other hand, the energy estimate~\eqref{compE:hp}
  implies that~$\vol_g(T_\eps)\lesssim\eps^2\abs{\log\eps}$.
  By applying the H\"older inequality, we obtain
  \begin{equation} \label{compu53}
   \norm{\D_0 u_\eps}_{L^p(T_\eps)}
   \lesssim \eps^{2/p - 1} \abs{\log\eps}^{1/p}
  \end{equation}
  Now, \eqref{compu52} and~\eqref{compu53} imply~\eqref{compu51}.
  \qedhere
 \end{step}
\end{proof}

%\begin{remark} %\label{rk:compactnessu} I PREFER THIS TO BE A COROLLARY RATHER THAN A REMARK
 As a byproduct of the arguments above, we have proved the following statement.
\begin{corollary}\label{cor:weak-conv-w}
 Let~$u_\eps\in W^{1,2}(M, \, E)$ be such that
 \begin{equation} \label{rk:compu:hp}
  \sup_{\eps> 0} \left(\norm{u_\eps}_{L^\infty(M)}
  + \frac{E_\eps(u_\eps)}{\abs{\log\eps}} \right) < +\infty
 \end{equation}
 Then, there exists functions~$\rho_\eps\colon M\to\R$,
 maps~$\Phi_\eps\colon M\to\S^1$ and a (non-relabelled) 
 subsequence such that $w_\eps : = \rho_\eps \, \Phi_\eps u_\eps$
 converges weakly in $W^{1,p}(M)$, for any~$p \in [1, \, n/(n-1))$,
 to a limit map~$w_*\in W^{1,p}(M, \, E)$
 with~$\abs{w_*} = 1$ a.e.
 Moreover, $J(u_\eps) - J(w_\eps) \to 0$ in~$W^{-1,p}(M)$
 for any~$p \in [1, \, n/(n-1))$.
%\end{remark}
\end{corollary}

Thus, any sequence 
$\{u_\eps\}$ bounded in $L^\infty(M)$ and with $E_\eps$-energy of order 
$\abs{\log\eps}$ can be split into a compact and a non-compact part (with respect to the weak $W^{1,p}$-topology). Moreover, Corollary~\ref{cor:weak-conv-w} asserts that the compact part stores the information necessary to determine the topological energy-concentration set, while the non-compact part of the sequence is, in this sense, ``topologically irrelevant''. 

On a qualitative level, the lack of compactness is due to wild 
oscillations in the phases made possible by the large amount of energy at 
disposal. At first glance, this inconvenient seems to be ``cured'' by the gauge transformations $\Phi_\eps$. However, we must emphasise that
the ``penalisations'' $\rho_\eps$, even if small in uniform norm as~$\eps \to 0$, play a subtle r\^ole. Indeed, Remark~\ref{rk:compactnessu} below points out that gauge transformations alone are in general not sufficient to perform the splitting and obtain compactness.

\begin{remark} \label{rk:compactnessu}
 Let~$u_\eps\in W^{1,2}(M, \, E)$ be a sequence that 
 satisfies~\eqref{rk:compu:hp}. In general, it may \emph{not}
 be possible to find maps~$\Phi_\eps\colon M\to\S^1$ such
 that~$\Phi_\eps u_\eps$ is bounded in $W^{1,p}(M)$.
 (A counterexample, when~$M = (0, \, 1)\subseteq\R$ and~$E$ 
 is the trivial bundle~$E = M\times\C$, is given by the sequence
 $u_\eps(x) := 1 + \eps \sin(\abs{\log\eps}^{1/2} x/\eps)$).
 However, if~$u_\eps$ are solutions of the Ginzburg-Landau 
 equations (i.e., critical points of~$E_\eps$) in Euclidean domains,
 then it is possible to obtain compactness for a sequence
 of the form~$\Phi_\eps u_\eps$ --- see~\cite[Section~4]{BethuelOrlandi}
%   ~\cite{BethuelBourgainBrezisOrlandi}
 and~\cite[Lemma~2.2]{BaldoOrlandiWeitkamp}, 
 \cite[Proposition~2.13]{BOSFactToulous}.
\end{remark}

\subsection{Lower bounds}

Again, we consider a sequence~$u_\eps\in W^{1,2}(M, \, E)$ 
that satisfies the energy bound~\eqref{compE:hp}.
By Lemma~\ref{lemma:compactnessE}, we know that
$J(u_\eps)\to \pi J_*$ in~$W^{-1,p}(M)$ for any~$p < n/(n-1)$.
The aim of this section is to prove the following

\begin{prop} \label{prop:liminfE}
 If the sequence~$\{u_\eps\}\subset W^{1,2}(M, \, E)$ 
 satisfies~\eqref{compE:hp}, then for any open
 set~$V\subset M$ there holds
 \[
  \pi \abs{J_*}\!(V) \leq \liminf_{\eps\to 0} 
  \frac{E_\eps(u_\eps; \, V)}{\abs{\log\eps}}
 \]
\end{prop}

We stress that, in the statement of Proposition~\ref{prop:liminfE},
the open set~$V$ may be chosen arbitrarily;
it does not need to be contained in a coordinate chart.
For instance, we may take~$V = M$. Therefore, 
once Proposition~\ref{prop:liminfE} is proved,
Statement~(i) of Theorem~\ref{maingoal:E} follows at once.

\begin{proof}[Proof of Theorem~\ref{maingoal:E}, Statement~(i)]
 The statement follows from Lemma~\ref{lemma:compactnessE}, 
 Proposition~\ref{prop:homologyJ} and Proposition~\ref{prop:liminfE}.
\end{proof}

In the rest of this section,
we deduce Proposition~\ref{prop:liminfE} from its
Euclidean counterpart~\cite{JerrardSoner-GL, ABO2},
by means of a localisation argument.

\begin{lemma} \label{lemma:local_liminf}
 Let~$\delta > 0$ be smaller than the injectivity radius of~$M$.
 Let~$U\subset M$ be a smooth, contractible domain, entirely contained in 
 a geodesic ball of radius~$\delta$. Then,
 \[
  \pi\abs{J_*}\!(U) \leq \left(1 + \mathrm{O}(\delta^2)\right) 
  \liminf_{\eps\to 0} \frac{E_\eps(u_\eps; \, U)}{\abs{\log\eps}}
 \]
\end{lemma}
\begin{proof}
 We identify~$U$ with a subset of~$\R^n$, by local coordinates,
 and write~$\D_0 = \d - i \gamma_0$, where~$\d$ is the flat connection on~$U$
 and~$\gamma_0$ is a smooth $1$-form that depends on~$\D_0$ only.
 We consider again the functional~$\bar{E}_\eps$ 
 given by~\eqref{BBH-GL-bar}. Thanks to~\eqref{eq:def-vol-norm-coord}
 (which we can apply as~$U$ is supposed to be contained in a geodesic ball),
 we can write
\[
\begin{split}
	\bar{E}_{\eps}(u_\eps; \, U) &= \int_U \frac{1}{2}\abs{\d u_\eps}^2 + \frac{1}{4\eps^2}\left(1-\abs{u_\eps}^2\right)^2 \,\d x \\
	&= \int_U \frac{1}{2}\abs{\D_0 u_\eps}^2 + \ip{\D_0 u_\eps}{i \gamma_0 u_\eps} + \abs{\gamma_0 u_\eps}^2 + \frac{1}{4\eps^2}\left(1-\abs{u_\eps}^2\right)^2 \,\d x \\
	&\leq (1+\mathrm{O}(\delta^2)) E_\eps(u_\eps; U) 
	+ (1+\mathrm{O}(\delta^{-2})) \int_U \abs{\gamma_0 u}^2 \vol_g \,, 
\end{split}
\]
The last line follows by Young's inequality (which gives, for each choice of $\sigma > 0$, the pointwise bound $\ip{\D_0 u_\eps}{i\gamma_0 u_\eps} \leq \frac{\sigma}{2} \abs{\D_0 u}^2 + \frac{1}{2\sigma} \abs{\gamma_0 u_\eps}^2$, and then the previous inequality follows by choosing $\sigma = \delta^2$).
Consequently, we obtain
\[
\begin{split}
	\frac{\bar{E}_\eps(u_\eps; \, U)}{\abs{\log \eps}} 
	\leq &(1+\mathrm{O}(\delta^2)) 
	 \frac{E_\eps(u_\eps; \, U)}{\abs{\log \eps}} 
	 + \frac{1+ C\delta^{-2}}{\abs{\log\eps}} 
	 \int_{U} \abs{\gamma_0 u_\eps}^2\vol_g.
\end{split}
\]
Since this holds for every $\eps > 0$, we can pass both sides to the $\liminf$ as $\eps \to 0$ keeping the inequality. On doing so, the last term on the right above vanishes, because the energy estimate~\eqref{compE:hp} implies that
$u_\eps$ is bounded in~$L^2(M)$. Therefore, we get the desired estimate recalling ($iii$) of Theorem~\ref{th:GL}.
\end{proof}

We deduce Proposition~\ref{prop:liminfE} from Lemma~\ref{lemma:local_liminf}
by applying a Vitali-Besicovitch-type covering theorem,
which we recall here. 

\begin{theorem}[Federer, \cite{Federer}]\label{th:covering}
 Let~$M$ be a compact Riemannian manifold. Let~$\mathcal{F}$
 be a collection of \emph{closed} geodesic balls in~$M$, such that
 \begin{equation} \label{hp:covering}
  \inf\left\{r > 0\colon \bar{B}_r(x_0)\in\mathcal{F} \right\} = 0 
  \qquad \textrm{for any } x_0\in M.
 \end{equation}
 Let~$\mu$ be a non-negative Borel measure on~$M$ and~$V\subset M$
 an open set. Then, there exists a countable 
 subfamily~$\mathcal{F}^\prime\subset\mathcal{F}$
 of pairwise-disjoint balls such that 
 \[
  \bigcup_{B\in\mathcal{F}^\prime} B\subset V, \qquad
  \mu\left(V\setminus\bigcup_{B\in\mathcal{F}^\prime} B\right) = 0
 \]
\end{theorem}

The proof of Theorem~\ref{th:covering} can be found
in~\cite[Theorem~2.8.14, Corollary~2.8.15]{Federer}. The statements
in~\cite{Federer} apply not only to compact Riemannian manifolds,
but also to a more general class of metric spaces, 
i.e.~`directionally limited' metric spaces.
(Moreover, they apply to outer measures as well as measures.)
However, the statement given here is sufficient
for our purposes.

The next lemma is an immediate consequence of 
Theorem~\ref{th:covering}.

\begin{lemma} \label{lemma:covering}
 For any~$\delta > 0$ and any open set~$V\subset M$, there exists a countable
 family~$\{U_j\}_{j\in\N}$ of pairwise-disjoint, \emph{open} geodesic balls,
 of radius less than~$\delta$, such that $U_j\subset V$ for any~$j$ and
 \[
  \abs{J_*}\!\left(V\setminus\bigcup_{j\in\N} U_j\right) = 0
 \]
\end{lemma}
\begin{proof}
 We apply Theorem~\ref{th:covering} to the bounded measure~$\mu := \abs{J_*}$
 and the collection of (closed) balls
 \[
  \begin{split}
   \mathcal{F} := \left\{\bar{B}_r(x_0)\colon 0 < r < \delta, \ x_0\in M
   \ \textrm{ such that } \ \abs{J_*}(\partial B_r(x_0)) = 0\right\}
  \end{split}
 \]
 Since~$\abs{J_*}$ is a finite measure, we have $\abs{J_*}(\partial B_r(x_0)) = 0$
 for a.e.~$r\in (0, \, \delta)$. Then, the assumptions 
 of Theorem~\ref{th:covering} are satisfied, and the lemma follows.
\end{proof}

\begin{proof}[Proof of Proposition~\ref{prop:liminfE}]
 Let~$V\subset M$ be an open set.
 Let~$\delta>0$ be a small parameter --- in particular, 
 smaller than the injectivity radius of~$M$.
 By Lemma~\ref{lemma:covering}, there exists a countable
 family~$\{U_j\}_{j\in\N}$ of pairwise-disjoint, open geodesic balls, 
 of radius smaller than~$\delta$, such that~$U_j\subset V$ for any~$j$ and
%  \begin{equation} \label{covering1bis}
%   \pi\abs{J_*}\!\left(V\setminus\bigcup_{j\in\N} U_j\right) = 0
%  \end{equation}
%  Then,
 \begin{equation} \label{covering2bis}
  \abs{J_*}\!(V) = \sum_{j\in\N} \abs{J_*}\!(U_j)  
 \end{equation}
%  Lemma~\ref{lemma:local_liminf} gives
%  \begin{equation} \label{covering3bis}
%    \pi\abs{J_*}\!(U_j) \leq \left(1 + \mathrm{O}(\delta^2)\right) 
%   \liminf_{\eps\to 0} \frac{E_\eps(u_\eps; \, U_j)}{\abs{\log\eps}}
%  \end{equation}
%  for any~$j$. 
 From~\eqref{covering2bis}, Lemma~\ref{lemma:local_liminf}
 and Fatou's lemma, we deduce
 \[
  \pi\abs{J_*}\!(V) \leq \left(1 + \mathrm{O}(\delta^2)\right) 
  \sum_{j\in\N} \liminf_{\eps\to 0} 
   \frac{E_\eps(u_\eps; \, U_j)}{\abs{\log\eps}}
  \leq \left(1 + \mathrm{O}(\delta^2)\right) 
  \liminf_{\eps\to 0} \frac{E_\eps(u_\eps; \, V)}{\abs{\log\eps}}
 \]
 Letting~$\delta\to 0$, the proposition follows.
\end{proof}

\subsection{Upper bounds}

The goal of this section is to prove Statement~(ii) of Theorem~\ref{maingoal:E}.
First, we introduce some notation. 
Let~$u\colon M\to E$ be a section of the bundle~$E\to M$,
and let~$\D_0$ be our reference (smooth) connection on~$E$. 
Let~$X\subset M$ be a closed set. 
Following~\cite{ABO1, ABO2}, we say that~$u$ has a 
\emph{nice singularity at~$X$} (with respect to~$\D_0$)
if~$\abs{u} = 1$ in~$M\setminus X$, $u$ is locally 
Lipschitz on~$M\setminus X$ and there
exists a constant~$C>0$ such that
\[
 \abs{\D_0 u(x)} \leq C \dist(x, \, X)^{-1}
 \qquad \textrm{for any } x\in M\setminus X.
\]
If~$u$ has a nice singularity at~$X$ with respect to~$\D_0$,
then~$u$ has a nice singularity at~$X$ with respect to
any smooth connection~$\D$, because~$\D$ can be written 
as~$\D = \D_0 -iA$ for some smooth $1$-form~$A$.
Therefore, there is no ambiguity in saying
that~$u$ has a nice singularity at~$X$
without specifying the reference connection~$\D_0$.
If~$X$ is a finite union of submanifolds of dimension~$q$ 
or less and~$u$ has a nice singularity at~$X$, then
$u\in (L^\infty\cap W^{1, p})(M, \, E)$ for any~$p < n - q$.
This is a consequence of the following lemma:

\begin{lemma} \label{lemma:distanceW1p}
 If~$X\subset M$ is contained in a finite union
 of Lipschitz submanifolds of dimension~$q$, then
 \[
  \int_{M} \frac{\vol_g}{\dist(\cdot, \, X)^p} < +\infty
 \]
 for any~$p\in [1, \, n-q)$.
\end{lemma}
\begin{proof}
 When~$M = \R^n$, the proof may be found 
 e.g.~in~\cite[Lemma~8.3]{ABO2}. When~$M$ is a compact,
 smooth manifold, we reduce to the Euclidean case 
 by working in coordinate charts.
\end{proof}

\begin{remark} \label{rk:Jac_nice_sing}
 If~$u$ has a nice singularity on a closed Lipschitz set~$X$ of dimension~$n-2$
 at most, then~$u\in (L^\infty\cap W^{1,1})(M, \, E)$
 (because of Lemma~\ref{lemma:distanceW1p}) and~$J(u)$ is well-defined.
 Moreover, $J(u)$ is supported on~$X$. Indeed, for any smooth vector field~$v$
 we have~$\ip{\D_{0,v} u}{u} = \d(\abs{u}^2/2) = 0$ at
 almost every point of~$M\setminus X$. This implies
 $J(u)[v, \, w] = \ip{i \D_{0,v} u}{\D_{0,w} u} = 0$
 a.e. in~$M\setminus X$, for any smooth vector fields~$v$, $w$. 
 %{\RRR{Forse qui siamo rimasti alla vecchia definizione di $J$ (cfr. \eqref{pointwise-Jac}).}}
\end{remark}

We consider smooth triangulations on~$M$, as defined 
e.g.~in~\cite[Definition~8.3]{Munkres}.
Given a triangulation~$\mathcal{T}$ of~$M$
and an integer~$q\in\{0, \, 1, \, \ldots, \, n\}$, we call~$\mathcal{T}_q$
the set of all~$q$-dimensional simplices of~$\mathcal{T}$.
We denote by~$\Sk_q\mathcal{T}$ the~$q$-dimensional skeleton of~$\mathcal{T}$,
that is, 
\[
 \Sk_q\mathcal{T} := \bigcup_{K\in\mathcal{T}_q} K
\]
Recall that~$\mathcal{R}_q(M)$ is the set of integer-multiplicity
rectifiable~$q$-chains in~$M$. We say that a chain~$S\in\mathcal{R}_q(M)$
is \emph{polyhedral} if there exist a triangulation~$\mathcal{T}$ of~$M$ and 
a function~$\theta\colon\mathcal{T}_q\to\Z$ such that
\begin{equation} \label{polyhedralchain}
 S = \sum_{K\in\mathcal{T}_q} \theta(K)\llbracket K\rrbracket
\end{equation}
When~\eqref{polyhedralchain} holds, we say that the current~$S$
is \emph{carried} by the triangulation~$\mathcal{T}$.
To construct a recovery sequence for~$E_\eps$, we borrow
a result by Parise, Pigati and Stern~\cite[Proposition~4.2]{ParisePigatiStern}.
For the convenience of the reader, we reproduce the statement here, 
using the notation we introduced above.

\begin{prop}[\cite{ParisePigatiStern}] \label{prop:approximation_chains}
 For any cycle~$S\in\mathcal{R}_{n-2}(M)$,
 there exists a sequence of polyhedral $(n-2)$-cycles~$S_j$, homologous to~$S$,
 such that~$\F(S_j-S)\to 0$ and~$\M(S_j)\to\M(S)$ as~$j\to+\infty$.
 Moreover, if~$S$ is a polyhedral $(n-2)$-cycle 
 in the class~$\mathcal{C}$ and~$\mathcal{T}$
 is a triangulation that carries~$S$,
 then there exists a section~$u\colon M\to E$ that has
 a nice singularity at~$\Sk_{n-2}\mathcal{T}$
 and satisfies~$\star J(u) = \pi S$.
\end{prop}

\begin{remark}
In the statement of~\cite[Proposition~4.2]{ParisePigatiStern},
the authors do not say explicitly that the approximating chains~$S_j$
and~$S$ are homologous, but this fact is contained in the proof. 
% (p.~18, l.~4). 
\end{remark}

Let~$\mathcal{T}$ be a triangulation of~$M$ and 
let~$\gamma >0$ be a small parameter.
For any simplex~$K\in\mathcal{T}_{n-2}$, we define
\begin{equation} \label{V_K}
 V_K := \left\{x\in M\colon 
  \dist(x, \, K) \leq \gamma\dist(x, \, \partial K) \right\}
\end{equation}
The set~$V_K$ is the closure of a Lipschitz domain in~$M$.
For~$\gamma$ small enough, $V_K$ is contained in a 
tubular neighbourhood of~$K$ and it retracts by deformation
onto~$K$. In particular, $V_K$ is contractible for~$\gamma$ small enough.

\begin{lemma} \label{lemma:V_K}
 If~$\gamma > 0$ is small enough, then,
 for any~$K\in\mathcal{T}_{n-2}$ and~$K^\prime\in\mathcal{T}_{n-2}$,
 the intersection~$V_K\cap V_{K^\prime}$ is either empty or 
 a common boundary~$(n-3)$-face of~$K$, $K^\prime$.
\end{lemma}
\begin{proof}
 We claim that, for any~$K\in\mathcal{T}_{n-2}$ 
 and~$K^\prime\in\mathcal{T}_{n-2}$, there exists a constant~$C>0$
 such that
 \begin{equation} \label{VK1}
  \dist(x, \, \partial K) \leq C \dist(x, \, K^\prime)
 \end{equation}
 for any~$x\in K$.
 Indeed, let~$x_0\in K$ be given. 
 It suffices to find a constant~$C$ that satisfies~\eqref{VK1}
 for any~$x$ close enough to~$x_0$; we will then able to 
 find~$C$ that satisfies~\eqref{VK1} for any~$x\in K$,
 as claimed, because~$K$ is compact.
 If~$x_0\in K\setminus K^\prime$, then the quotient
 $\dist(\cdot, \, \partial K)/\dist(\cdot, \, K^\prime)$
 is bounded in a neighbourhood of~$x_0$, and~\eqref{VK1} follows.
 Suppose now that~$x_0\in K\cap K^\prime$. By definition
 (see~\cite[Definition~8.3]{Munkres}, $\mathcal{T}$
 is the homeomorphic image of a simplicial complex in~$\R^\ell$
 (for some~$\ell \geq n$), via a piecewise smooth map
 with injective differential. (As the differential of the parametrisation
 is injective, it follows that~$K$, $K^\prime$ meet at a
 non-zero angle at~$x_0$.) Upon composition with the parametrisation,
 we may assume without loss of generality that~$K$, $K^\prime$
 are affine simplices in an affine plane of dimension~$n-1$.
 Then, a routine computation shows that~\eqref{VK1}
 is satisfied in a neighbourhood of~$x_0$, for some constant~$C$
 that depends on the angle between~$K$ and~$K^\prime$ at~$x_0$.
 This completes the proof of~\eqref{VK1}.
 As there are only finitely
 many simplices in~$\mathcal{T}$, the constant~$C$ in~\eqref{VK1}
 may be chosen uniformly with respect to~$K$, $K^\prime$.
 
 Now, let~$K\in\mathcal{T}_{n-2}$, $K^\prime\in\mathcal{T}_{n-2}$
 be such that $V_{K}\cap V_{K^\prime}\neq\emptyset$.
 Let~$x\in V_{K}\cap V_{K^\prime}$, and let
 \[
  d := \dist(x, \, K), \qquad d^\prime := \dist(x, \, K^\prime)
 \]
 As~$x\in V_{K}$, we have
 \[
  d \leq \gamma\dist(x, \, \partial K)
  \stackrel{\eqref{VK1}}{\leq} C\gamma \, d^\prime
 \]
 Similarly, $d^\prime \leq C\gamma \, d$ and hence,
 $d + d^\prime \leq C\gamma(d + d^\prime)$. If we take~$\gamma < 1/C$,
 then $d + d^\prime = 0$, that is, $x\in K\cap K^\prime$
 (and, in fact, $x \in \partial K \cap \partial K^\prime$).
 Conversely, the definition~\eqref{V_K} immediately implies
 that $K\cap K^\prime\subset V_K\cap V_{K^\prime}$.
 This proves the lemma.
\end{proof}

We can now prove Statement~(ii) of Theorem~\ref{maingoal:E}.

\begin{proof}[Proof of Theorem~\ref{maingoal:E}, Statement~(ii)]
 Let~$S_*$ be an integer-multiplicity rectifiable $(n-2)$-cycle
 in the homology class~$\mathcal{C}$, and let~$J_* := \star S_*$
 be the dual $2$-form. Let~$\delta > 0$.
 We claim that, for any~$\eps > 0$, there exists
 $u_\eps\in W^{1,2}(M, \, E)$ such that 
 $J(u_\eps)\to \pi J_*$ in~$W^{-1,1}(M)$ as~$\eps\to 0$ and
 \begin{equation} \label{Elimsup0}
  \limsup_{\eps\to 0}\frac{E_\eps(u_\eps)}{\abs{\log\eps}}
  \leq \pi \, (1 + \delta) \, \M(S_*) 
 \end{equation}
 Once~\eqref{Elimsup0} is proved, Lemma~\ref{lemma:compactnessE}
 implies that~$J(u_\eps)\to \pi J_*$ in~$W^{-1,p}(M)$ not only for~$p=1$,
 but also for any~$p < n/(n-1)$. 
 Then, Statement~(ii) in Theorem~\ref{maingoal:E} follows
 from a diagonal argument.
 
 Thanks to Proposition~\ref{prop:approximation_chains},
 and up to a diagonal argument, we may assume without loss
 of generality that~$S_*$ is a polyhedral $(n-2)$-cycle
 in the homology class~$\mathcal{C}$. Let~$\mathcal{T}$
 be a triangulation that carries~$S_*$. Up to a subdivision,
 we may assume without loss of generality that each simplex~$K\in\mathcal{T}$
 has diameter~$\mathrm{diam} \, K \leq \delta$.
 For each $(n-2)$-simplex~$K\in\mathcal{T}_{n-2}$, 
 we define~$V_K$ as in~\eqref{V_K}. (We choose the parameter~$\gamma$
 small enough, so that sets~$V_K$ have pairwise disjoint interiors,
 as in Lemma~\ref{lemma:V_K}.) Let
 \[
  V := \bigcup_{K\in\mathcal{T}_{n-2}} V_K
 \]
 By Proposition~\ref{prop:approximation_chains}, there exists
 a section~$u\colon M\to E$ that has
 a nice singularity at~$\Sk_{n-2}\mathcal{T}$
 and satisfies~$J(u) = \pi J_*$. 
 
 \medskip
 \noindent
 \textit{Construction of the recovery sequence out of~$V$.}
 We claim that~$u$ is of class~$W^{1,2}$ in~$M\setminus V$.
 Indeed, let~$x\in M\setminus V$. By definition of~$V_K$,
 such a point~$x$ satisfies
 \[
  \dist(x, \, K) \geq \gamma\dist(x, \, \partial K)
  \geq \gamma \dist(x, \, \Sk_{n-3}\mathcal{T})
  \qquad \textrm{for any } K\in\mathcal{T}_{n-2}
 \]
 and hence, $\dist(x, \, \Sk_{n-2}\mathcal{T})\geq\gamma\dist(x, \, \Sk_{n-3}\mathcal{T})$. As~$u$ has a nice singularity
 at~$\Sk_{n-2}\mathcal{T}$, we obtain
 \begin{equation} \label{Elimsup5}
  E_\eps(u; \, M\setminus V)
  = \frac{1}{2}\int_{M\setminus V} \abs{\D_0 u}^2 \vol_g < +\infty
 \end{equation}
 because of Lemma~\ref{lemma:distanceW1p}. Moreover,
 as the Jacobian is a local operator (Remark~\ref{rk:locality-jac2}),
 Remark~\ref{rk:Jac_nice_sing} implies
 \begin{equation} \label{Elimsup6}
  J(u_{|M\setminus K}) = J(u)_{| M\setminus V} = 0
 \end{equation}
 In view of~\eqref{Elimsup5}, \eqref{Elimsup6}, 
 it makes sense to take~$u_\eps := u$ on~$M\setminus V$.
 
 \medskip
 \noindent
 \textit{Construction of the recovery sequence on each~$V_K$.}
 Let~$K\in\mathcal{T}_{n-2}$. As~$\mathrm{diam}\,\mathcal{T} \leq \delta$,
 we have~$\mathrm{diam}\, V_K \leq C\delta$, for some constant~$C$
 that depends only on~$\gamma$. When~$\delta$
 small enough, the diameter of~$V_K$ is smaller than
 the injectivity radius of~$V_K$. By working in geodesic 
 normal coordinates, we may identify~$V_K$ with a subset of~$\R^n$.
 We also identify sections~$V_K\to E$ with complex-valued maps
 defined on a subset of~$\R^n$. 
 The restriction~$u_{|\partial V_K}$ has a nice singularity at~$\partial K$
 and hence, $u\in W^{1,p}(\partial V_K, \, E)$ for any~$p<2$,
 because of Lemma~\ref{lemma:distanceW1p}. By Sobolev embeddings,
 it follows that~$u\in W^{1/2,2}(\partial V_K, \, E)$.
 Therefore, we are in position
 to apply $\Gamma$-convergence results for the
 Ginzburg-Landau functional in Euclidean settings.
 Thanks to~\cite[Theorem~5.5 p.~32 and Remark~i p.~33]{ABO2},
 there exists a sequence of complex-valued maps~$u_\eps^K$,
 of class~$W^{1,2}$ in the interior of~$V_K$, such that
 $u_\eps^K = u$ on~$\partial V_K$ (in the sense of traces),
 \begin{align}
  &\bar{J}(u_\eps^K) - \pi J_{*| V_K} \to 0 
   \qquad \textrm{in } W^{-1,1}(M) \label{Elimsup1} \\
  &\lim_{\eps\to 0} \frac{1}{\abs{\log\eps}} 
   \int_{V_K} \left(\frac{1}{2} \abs{\d u_\eps^K}^2 
   + \frac{1}{4\eps^2}\left(1 - \abs{u_\eps^K}^2\right)^2\right) \d x
   = \pi \, \M(S_*\res V_K) \label{Elimsup2}
 \end{align}
 Here, as in~\eqref{usualJac},
 $\bar{J}(u_\eps^K) := \frac{1}{2}\d\ip{\d u_\eps^K}{iu_\eps^K}$
 and~$\d x$ denotes the Lebesgue measure on~$V_K$.
 The reference connection~$\D_0$ can be written as~$\D_0 = \d - i\gamma_0$
 on~$V_K$, for some smooth $1$-form~$\gamma_0$. Therefore,
 $J(u_\eps^K) - \bar{J}(u_\eps^K) = \frac{1}{2}\d(\gamma_0(1 - |u_\eps^K|^2))$.
 The energy estimate~\eqref{Elimsup2} implies that
 $J(u_\eps^K) - \bar{J}(u_\eps^K) \to 0$ in~$W^{-1,2}(M)$ and hence,
 \begin{equation} \label{Elimsup3}
  J(u_\eps^K) - \pi J_{*| V_K} \to 0 
   \qquad \textrm{in } W^{-1,1}(M) 
 \end{equation}
 We recall that the volume form~$\vol_g$,
 in geodesic coordinates on a ball of radius~$r$, 
 satisfies $\vol_g \leq (1 + Cr^2)\d x$, where~$C$ is a 
 constant that depends only on the curvature of~$M$, not on~$r$.
 Then, keeping in mind that the difference~$\D_0 - \d$
 is bounded independently of~$\eps$, from~\eqref{Elimsup2}
 we deduce
 \begin{equation} \label{Elimsup4}
  \limsup_{\eps\to 0} \frac{1}{\abs{\log\eps}} 
   E_\eps(u_\eps^K; \, V_K) \leq \pi \, (1 + \delta)\, \M(S_*\res V_K)
 \end{equation}
 for~$\delta$ small enough.
 
 \medskip
 \noindent
 \textit{Conclusion.}
 Recall that we have defined~$u_\eps := u$ in~$M\setminus K$.
 Now, we take~$u_\eps := u_\eps^K$ on each~$V_K$.
 Since the~$V_K$'s have pairwise disjoint interiors,
 the section~$u_\eps$ is well-defined and is of class~$W^{1,2}$.
 As the Jacobian is a local operator, from~\eqref{Elimsup6}
 and~\eqref{Elimsup3} we deduce that
 $J(u_\eps)\to \pi J_*$ in~$W^{-1,1}(M)$. The inequality~\eqref{Elimsup0}
 follows from~\eqref{Elimsup5} and~\eqref{Elimsup4}.
 This completes the proof.
\end{proof}

\section{$\Gamma$-convergence for the functional~$G_\eps$}\label{sec:maingoalG}

In this section, we prove Theorem~\ref{maingoal:G} relying on Theorem~\ref{maingoal:E} 
and Corollary~\ref{cor:London} relying, in turn, on Theorem~\ref{maingoal:G}. 
Since the recovery sequence in the proof of (ii) of Theorem~\ref{maingoal:E} essentially works as well for the functional $G_\eps$, almost all the work in this section is addressed to the proof of Statement~(i) of Theorem~\ref{maingoal:G}. The key idea is that the given sequence $\{(u_\eps, A_\eps)\}$ in Statement~{(i)} of Theorem~\ref{maingoal:G} can always be replaced, in the proof, with a better sequence to the purpose, which we call an \emph{optimised sequence} (Definition~\ref{def:reduced-seq}). The relevant properties of optimised sequences are listed
in Lemma~\ref{lemma:complete-reduction}. They allow, in the end, both to find a lower bound for the $\liminf$ of the rescaled energies (see Corollary~\ref{cor:convergence-jacobians} and Remark~\ref{rk:conv-jac}) and to show the flat convergence of gauge-invariant Jacobians (see Lemma~\ref{lemma:complete-reduction}, Remark~\ref{rk:reduction}, and Corollary~\ref{cor:convergence-jacobians}) with much less effort than needed dealing directly with the ``original'' sequence $\{(u_\eps, A_\eps)\}$ in Statement~{(i)} of Theorem~\ref{maingoal:G}.

\vskip5pt

We begin this section dealing with an auxiliary problem, involved in the construction of optimised sequences. As a preliminary step, recall from Remark~\ref{rk:geometric-W22} that an exact $2$-form $\psi \in W^{1,2}(M, \Lambda^2 \T^*M)$ with $\d^* \psi \in W^{1,2}(M, \T^*M)$ actually belongs to $W^{2,2}(M, \, \Lambda^2\T^*M)$. In fact, there holds more: by Remark~\ref{rk:est-exact-forms} there is a constant $C > 0$, depending only on $M$, such that the estimate
\begin{equation}\label{eq:est-exact-2forms}
	\norm{\psi}_{W^{2,2}(M)} \leq C \norm{\d^* \psi}_{W^{1,2}(M)}
\end{equation}
holds for all exact 2-forms $\psi \in W^{2,2}(M,\,\Lambda^2\T^*M)$.
 
Now we can introduce the mentioned auxiliary problem. Let $\D_0$ be a smooth reference connection, let $A \in W^{1,2}(M, \T^*M)$, and set
\begin{equation}\label{eq:class-A}
\begin{split}
	[A] := &\left\{ B \in W^{1,2}(M, \T^*M) : B = A + \d^* \psi \right. \\
	& \left.\mbox{ for some exact 2-form } \psi \in W^{2,2}(M, \Lambda^2 \T^*M) \right\}.
\end{split}
\end{equation}

\begin{lemma}\label{eq:class-A-closed}
	For all $A \in W^{1,2}(M,\,\T^*M)$, the corresponding set $[A]$, defined in~\eqref{eq:class-A}, is not empty and it is sequentially weakly closed in $W^{1,2}(M, \T^*M)$.
\end{lemma}

\begin{proof}
	Since $A \in [A]$ for every $A \in W^{1,2}(M,\,\T^*M)$, $[A]$ is always not empty. Fix $A \in W^{1,2}(M,\,\T^*M)$, assume $\{B_j\} \subset [A]$ is a weakly convergent sequence in $W^{1,2}(M,\,\T^*M)$, and let $B$ denote the weak limit of $\{B_j\}$. Clearly, $B \in W^{1,2}(M,\,\T^*M)$. We have to show that $B$ writes as $A + \d^* \psi$ for an exact $2$-form $\psi \in W^{2,2}(M,\,\Lambda^2 \T^* M)$. Since $\{B_j\}$ is weakly convergent in $W^{1,2}(M,\,\T^*M)$, $\{B_j\}$ is bounded in $W^{1,2}(M,\,\T^*M)$. Consequently, $\{\d^* \psi_j\}$ is a bounded sequence in $W^{1,2}(M,\,\T^*M)$ and, since each $\psi_j$ is exact,  by~\eqref{eq:est-exact-2forms} the sequence $\{\psi_j\}$ is bounded in $W^{2,2}(M,\,\Lambda^2 \T^*M)$. Thus, $\{\psi_j\}$ contains a (not relabeled) subsequence which converges weakly in $W^{2,2}(M,\,\Lambda^2 \T^*M)$ to some $\psi \in W^{2,2}(M,\,\Lambda^2 \T^*M)$. By the compactness of $M$, the embedding $W^{2,2}(M,\,\Lambda^2 \T^*M) \hookrightarrow W^{1,2}(M,\,\Lambda^2 \T^*M)$ is compact. By Hodge-decomposition~\eqref{eq:LpHodgeDec}, this implies that the co-exact and harmonic parts of $\psi$ are zero, hence $\psi$ is still exact. From the above, we have $B_j \rightharpoonup A + \d^* \psi$ weakly in $W^{1,2}(M, \T^*M)$ as $j \to \infty$. By the uniqueness of the weak limit, it follows that $B = A + \d^* \psi$. Since $\psi$ is exact, we have eventually $B \in [A]$, proving that $[A]$ is sequentially weakly closed. 
\end{proof}

Fix a smooth reference connection $\D_0$, a 1-form $A \in W^{1,2}(M,\,\T^*M)$, and a section $u \in W^{1,2}(M,\,E)$. Let $\mathcal{F}(\,\cdot\, ; u, A, \D_0) : [A] \to \R$ be the functional defined as follows:
\begin{equation}\label{eq:def-F}
	\forall B \in [A], \qquad \mathcal{F}(B; u, A, \D_0) := \int_M \abs{\D_B u}^2 + \abs{F_B}^2 \,\vol_g.
\end{equation}

\begin{prop}\label{prop:existence-min-F}
	For every choice of a reference smooth connection $\D_0$, of $A \in W^{1,2}(M, \T^*M)$ and of $u \in W^{1,2}(M,E)$, the functional $\mathcal{F}( \,\cdot\,; u, A, \D_0) : [A] \to \R$ defined in~\eqref{eq:def-F} (where $[A]$ is the set defined by~\eqref{eq:class-A}) has a minimiser over $[A]$. Furthermore, if $\D_A u \in L^2(M,\,\T^*M \otimes E)$, then any minimiser $B$ of $\mathcal{F}(\,\cdot\,; u, A, \D_0)$ over $[A]$ satisfies the London equation, i.e.,
	\[
		-\Delta F_{B} + F_{B} = 2J(u, B)
	\]
	in the sense of $\mathcal{D}^\prime(M)$. In addition, there holds
	\begin{equation}\label{eq:GBlessGA}
		G_\eps(u, B) \leq G_\eps(u, A)
	\end{equation}
	for every $\eps > 0$.
\end{prop}

\begin{proof}
	To shorten notations, let us omit the dependence on $u$, $A$, and $\D_0$ and simply write $\mathcal{F}$ in place of $\mathcal{F}( \,\cdot\,; u, A, \D_0)$. From Lemma~\ref{eq:class-A-closed}, $[A]$ is not empty and sequentially weakly closed. Obviously, $\mathcal{F}$ is bounded below, and thus, to prove existence of minimisers of $\mathcal{F}$ over $[A]$, we only have to prove that $\mathcal{F}$ is sequentially coercive and sequentially lower semicontinuous over $[A]$. Once this is done, the conclusion follows immediately by the direct method in the calculus of variations.
	
	\emph{Step~1: $\mathcal{F}$ is sequentially coercive.} As $[A]$ is sequentially weakly closed in the reflexive space $W^{1,2}(M,\T^*M)$, it suffices to prove that $\mathcal{F}(B) \geq C_1 \norm{B}_{W^{1,2}(M)}^2 - C_2$ for $B \in [A]$, where $C_1$ and $C_2$ are positive constants independent of $B$. Thus, let $B \in [A]$. Then, $B = A + \d^* \psi$ for some exact $2$-form $\psi \in W^{2,2}(M, \Lambda^2 \T^*M)$. Notice that $F_B = F_A + \d \d^* \psi$ and that $\d \d^* \psi = -\Delta \psi$ because $\psi$ is an exact form belonging to $W^{2,2}(M,\Lambda^2 \T^*M)$. Also, recall that $\norm{\Delta \psi}_{L^2(M)}$ is a norm equivalent to $\norm{\psi}_{W^{2,2}(M)}$ on exact forms of class $W^{2,2}$. Thus, using Young's inequality,
	\begin{equation}\label{eq:F-coercive}
		\begin{split}
		 	\mathcal{F}(B) &= \int_M \abs{\D_B u}^2 + \abs{F_B}^2 \, \vol_g \\
		 	&= \int_M \abs{\D_A u - i (\d^*\psi) u}^2 + \abs{F_A + \d \d^* \psi}^2 \,\vol_g \\
		 	&= \int_M \left\{ \abs{\D_A u}^2 - 2 \ip{\D_A u}{i(\d^*\psi)u} + \abs{(\d^* \psi)u}^2 \right.\\
		 	&\left. \quad + \abs{F_A}^2 + 2 \ip{F_A}{\d \d^* \psi} + \abs{\d \d^* \psi}^2 \right\} \,\vol_g \\
		 	&\geq C_1 \norm{\Delta \psi}_{L^2(M)}^2 - \mathcal{F}(A) \\
		 	&\geq C_1 \left( \norm{A}_{W^{1,2}(M)}^2 + \norm{\d^* \psi}_{W^{1,2}(M)}^2 \right) - C_2(u, A, \D_0) \\
		 	&\geq C_1 \norm{B}_{W^{1,2}(M)}^2 - C_2(u, A, \D_0),
		 \end{split}
	\end{equation}
	where $C_1, C_2 > 0$ are constants independent of $B$ (that is, independent of $\psi$). Thus, $\mathcal{F}$ is sequentially coercive over $[A]$, for every choice of $\D_0$, $A$, and $u$.
	
	\emph{Step~2: $\mathcal{F}$ is sequentially weakly lower semicontinous.} Let $\{B_j\}$ be a sequence in $[A]$. Then, for each $j \in \N$ there is an exact $2$-form $\psi_j \in W^{2,2}(M,\,\Lambda^2 \T^*M)$ such that $B_j = A + \d^* \psi_j$. Assume $\{B_j\}$ converges weakly to $B \in W^{1,2}(M, \T^*M)$. Then, $B \in [A]$, $\{\psi_j\}$ converges weakly in $W^{2,2}(M,\,\Lambda^2 \T^*M)$ to some $\psi \in W^{2,2}(M,\,\Lambda^2 \T^*M)$, and $B = A + \d^*\psi$. Since $M$ is compact, the embedding $W^{2,2}(M,\,\Lambda^2 \T^*M) \hookrightarrow W^{1,2}(M,\,\Lambda^2 \T^*M)$ is compact, and therefore $\d^*\psi_j \to \d^* \psi$ strongly in $L^2(M,\,\T^*M)$ as $j \to \infty$ and almost everywhere on a (not relabelled) subsequence. From Fatou's lemma it then follows that
	\[
		\int_M \abs{(\d^* \psi)u}^2 \,\vol_g \leq \liminf_{j \to \infty} \int_M \abs{(\d^* \psi_j)u}^2 \,\vol_g.
	\]
On the other hand, $\Delta \psi_j \to \Delta \psi$ weakly in $L^2(M,\,\Lambda^2 \T^*M)$, and hence we have (recall that $\d \d^* \omega = -\Delta \omega$ if $\omega$ is an exact form of class $W^{2,2}$) 
	\[
		\norm{\d \d^* \psi}_{L^2(M)} \leq \liminf_{j \to \infty} \norm{ \d \d^* \psi_j}_{L^2(M)},
	\]	
	by the sequential weak lower semicontinuity of the $L^2$-norm. From the third line in~\eqref{eq:F-coercive} we get
	\[
		\mathcal{F}(B) \leq \liminf_{j \to \infty} \mathcal{F}(B_j),
	\] 
	and then, by the direct method in the calculus of variations, there exist a minimiser $B$ of $\mathcal{F}$ over $[A]$. Then, it is obvious that~\eqref{eq:GBlessGA} holds.
	
	\emph{Step~3: If $\mathcal{F} \not\equiv +\infty$, every minimiser satisfies the London equation and~\eqref{eq:GBlessGA}.} Assume, in addition to the previous hypotheses, that $\D_A u \in L^2(M,\, \T^*M \otimes E)$. Then, $\mathcal{F}(A) < +\infty$, hence minimisers of $\mathcal{F}$ over $[A]$ have finite energy (in fact, slightly more mildly, we could have required directly finite energy for minimisers in the statement). Let $B$ be a minimiser of $\mathcal{F}$ over $[A]$. To see that $F_B$ satisfies the London equation, take any $\eta \in C^\infty(M,\,\Lambda^2\T^*M)$ and set $B_t := B + t \d^* \eta$ for $t \in \R$. By Hodge decomposition, only the exact part of $\eta$ contributes to $B_t$, whence $\mathcal{F}(B) \leq \mathcal{F}(B_t)$ for every $t \in \R$ and every choice of $\eta \in C^\infty(M,\,\Lambda^2\T^*M)$. Thus, we must have $\left.\frac{\d}{\d t}\right\vert_{t=0} \mathcal{F}(B_t) = 0$, and so
	\begin{equation}\label{eq:integral-EL-F}
	%\begin{split}
		\int_M \left( \ip{\D_{B} u}{-i(\d^*\eta) u} + \ip{F_{B}}{\d\d^*\eta} \right) \vol_g = \int_M \ip{-j(u, B) + \d^* F_{B}}{\d^*\eta}\,\vol_g = 0.
	%\end{split}
	\end{equation}  
	As~\eqref{eq:integral-EL-F} holds for any $\eta \in C^\infty(M,\,\Lambda^2 \T^*M)$, it follows
	\[
		\d(\d^* F_{B} - j(u,B)) = 0,
	\]
	in the sense of distributions, whence (as $F_{B}$ is closed)
	\[
		-\Delta F_{B} + F_{B} = 2J(u, B),
	\]
	in the sense of distributions.
\end{proof}

\begin{remark}
	For $u \equiv 0$ the functional $\mathcal{F}$ reduces to the Yang-Mills functional $\mathcal{YM}(B) := \int_M \abs{F_B}^2 \vol_g$, considered over the class $[A]$; i.e., $\mathcal{F}(B\,; 0, A,\D_0) = \mathcal{YM}(B)$ for $B \in [A]$. An argument very similar to that in the proof of Proposition~\ref{prop:existence-min-F} yields existence of $\U(1)$-Yang-Mills connections over $E \to M$ in $[A]$, for every $A \in W^{1,2}(M,\,\T^*M)$. Indeed, the existence of a minimiser $B$ of $\mathcal{YM}$ over $[A]$ was already proved in Step~1 and Step~2. The fact that for any such  minimiser $F_B$ satisfies the Yang-Mills equations $\d^* F_B = 0$ in the sense of distributions follows easily similarly to as in Step~3. Indeed, if $t \in \R$ and $\psi \in C^\infty(M,\,\Lambda^2\T^*M)$ is any exact 2-form, then $B_t := B + t \d^*\psi \in [A]$ and  $F_{B_t} = F_B + t \d (\d^* \psi)$. Hence,
	\[
		\frac{1}{2}\left.\frac{\d}{\d t}\right\vert_{t=0} \mathcal{YM}(B_t) = \int_M \ip{F_B}{\d(\d^*\psi)} \vol_g =\int_M \ip{\d^* F_B}{\d^* \psi} \vol_g = 0, 
	\]
	for every $\psi \in C^\infty(M,\,\Lambda^2 \T^*M)$. On the other hand, for any $\omega \in C^\infty(M,\,\T^*M)$, we have $\omega = \d \varphi + \d^*\psi + \xi$ for some $\varphi \in C^\infty(M)$, $\psi \in C^\infty(M,\,\Lambda^2 \T^*M)$ exact, and $\xi \in \Harm^1(M)$, hence
	\[
		\ip{\d^* F_B}{\omega}_{\mathcal{D}',\mathcal{D}} = \ip{\d^* F_B}{\d^*\psi}_{\mathcal{D}',\mathcal{D}} = 0
	\]
	for every $\omega \in C^\infty(M,\,\T^*M)$. Thus, $\d^* F_B = 0$ in the sense of $\mathcal{D}'(M)$.
\end{remark}

Since we want to prove a $\Gamma$-convergence result, and since $\abs{\log\eps}$ is the energy scaling of minimisers of $G_\eps$ according to Remark~\ref{rk:energy-bound-min}, it is natural to be concerned with sequences $\{(u_\eps, A_\eps)\} \subset W^{1,2}(M,E) \times W^{1,2}(M,\T^*M)$ satisfying
\begin{equation} \label{G:hp}
 \sup_{\eps>0} \frac{G_\eps(u_\eps, \, A_\eps)}{\abs{\log\eps}} < +\infty.
\end{equation}
We are now going to explore the consequences of this additional assumption. %The first is convergence of Jacobians in the $W^{-1,1}$-topology as $\eps \to 0$.
\begin{lemma}\label{lemma:conv-jac-w^(-1,1)}
	Let $\{(u_\eps, A_\eps)\}$ be a sequence in $W^{1,2}(M,\,E)\times W^{1,2}(M,\,\T^*M)$. For each $\eps > 0$, let $B_\eps$ be a minimiser of $\mathcal{F}(\,\cdot\,; u_\eps, A_\eps, \D_0)$ over $[A_\eps]$, where $[A_\eps]$ is defined as in~\eqref{eq:class-A}. Assume that $\{(u_\eps, A_\eps)\}$ satisfies~\eqref{G:hp}. Then,
	\begin{equation}
		J(u_\eps, A_\eps) - J(u_\eps, B_\eps) \to 0 \quad \mbox{in } W^{-1,1}(M)
	\end{equation}
	as $\eps \to 0$.
\end{lemma}

\begin{proof}
	For each $\eps > 0$ there is an exact form $\psi_\eps \in W^{2,2}(M,\,\Lambda^2 \T^*M)$ such that $B_\eps = A_\eps + \d^* \psi_\eps$, and 
	\[
		-\Delta \psi_\eps = \d \d^* \psi_\eps = \d(B_\eps - A_\eps) = F_{B_\eps} - F_{A_\eps}.
	\]
	By \eqref{G:hp} and the minimality of $B_\eps$, we have
	\begin{equation}\label{eq:bound-B-aux-energy}
		\int_M \left( \abs{\D_{B_\eps} u_\eps}^2 + \abs{F_{B_\eps}}^2 \right) \,\vol_g \leq \int_M \left( \abs{\D_{A_\eps} u_\eps}^2 + \abs{F_{A_\eps}}^2 \right) \, \vol_{g} \lesssim \abs{\log\eps},
	\end{equation}
	up to a constant depending only on $M$, and therefore
	\[
		\norm{-\Delta \psi_\eps}_{L^2(M)} \lesssim \abs{\log\eps}^{1/2}.
	\]
	%We can certainly assume that $\psi_\eps$ is $L^2$-orthogonal to all harmonic $2$-forms (otherwise, we could reduced to that case by adding a harmonic $2$-form without affecting $-\Delta \psi_\eps$ nor $\d^* \psi_\eps = B_\eps - A_\eps$). The advantage of doing this is that the restriction of $-\Delta$ to the space of $2$-forms in $W^{2,2}(M,\Lambda^2 \T^*M)$ that are $L^2$-orthogonal to all harmonic $2$-forms is a continuous bijection onto $L^2(M, \Lambda^2 \T^*M)$. As a consequence, by the Open Mapping Theorem, it has a continuous inverse and hence
	By~\eqref{eq:est-exact-2forms},
	\[
		\norm{\psi_\eps}_{W^{2,2}(M)} \lesssim \norm{\Delta \psi_\eps}_{L^2(M)} \lesssim \abs{\log\eps}^{1/2},
	\]
	%By Lemma~\ref{lemma:classic-elliptic-reg}, we have $\norm{\psi_\eps}_{W^{2,2}(M)} \lesssim \abs{\log \eps}^{1/2}$ for all $\eps > 0$, and in particular
	whence 
	\[
        \norm{B_\eps - A_\eps}_{L^2(M)} \lesssim \abs{\log \eps}^{1/2}.
	\]
	Then,
	\[
		\|(A_\eps - B_\eps)(1-\abs{u_\eps}^2)\|_{L^1(M)} \lesssim \abs{\log \eps}^{1/2} \abs{\log \eps}^{1/2} \eps
	\]
for all $\eps > 0$, and the right hand side tends to zero sending $\eps \to 0$. Consequently, recalling~\eqref{eq:diff-jJAB}, it follows
	\[
		J(u_\eps, A_\eps) - J(u_\eps, B_\eps) \to 0 \quad \mbox{in } W^{-1,1}(M) \quad \mbox{as } \eps \to 0,
	\]
	which is the conclusion.
\end{proof}

\begin{definition}\label{def:reduced-seq}
	For any given sequence $\{(u_\eps, A_\eps)\} \subset W^{1,2}(M,\,E) \times W^{1,2}(M,\,\T^* M)$, we define a new sequence $\{(v_\eps, B_\eps)\} \subset (L^\infty \cap W^{1,2})(M,\,E) \times W^{1,2}(M,\,\T^*M)$ as follows. First, we replace each $u_\eps$ with $v_\eps$, the essentially bounded section associated with $u_\eps$ by \eqref{trunc:v}. Secondly, we replace each $A_\eps$ with a corresponding minimiser $B_\eps$ of $\mathcal{F}(\,\cdot\,; v_\eps, A_\eps, \D_0)$ over $[A_\eps]$, where $[A_\eps]$ is defined as in~\eqref{eq:class-A}. We call $\{(v_\eps, B_\eps)\}$ an \emph{optimised sequence} associated with $\{(u_\eps, A_\eps)\}$.
\end{definition}

\begin{lemma}[Reduction lemma]\label{lemma:complete-reduction}
	For any sequence $\{(u_\eps, A_\eps)\} \subset W^{1,2}(M,E) \times W^{1,2}(M, \T^* M)$, let $\{(v_\eps, B_\eps)\} \subset (L^\infty \cap W^{1,2})(M,E) \times W^{1,2}(M, \T^* M)$ be an associated optimised sequence (defined in Definition~\ref{def:reduced-seq}). Then, $\{(v_\eps, B_\eps)\}$ satisfies 
	\begin{equation}\label{eq:less-energy}
		G_\eps (v_\eps, B_\eps) \leq G_\eps(u_\eps, A_\eps)
	\end{equation}
	for any $\eps > 0$. Consequently,
	\begin{equation}\label{eq:liminfGeps}
		\liminf_{\eps \to 0} \frac{G_\eps(v_\eps, B_\eps)}{\abs{\log\eps}} \leq \liminf_{\eps \to 0} \frac{G_\eps(u_\eps, A_\eps)}{\abs{\log\eps}}.
\end{equation}		
	Moreover, under the assumption \eqref{G:hp} there holds
	\begin{equation}\label{eq:conv-jac}
		J(u_\eps, A_\eps) - J(v_\eps, B_\eps) \to 0 \quad \mbox{in } W^{-1,p}(M) \mbox{ for any } p \mbox{ such that } 1 \leq p < \frac{n}{n-1} 
	\end{equation}
	as $\eps \to 0$.
\end{lemma}

\begin{proof}
	Inequality \eqref{eq:less-energy} follows by construction. Indeed, by \eqref{eq:GBlessGA} and \eqref{trunc:energy} we have
	\[
		G_\eps(v_\eps, B_\eps) \leq G_\eps(v_\eps, A_\eps) \leq G_\eps(u_\eps, A_\eps)
	\]
	for all $\eps > 0$. Then, obviously, the inequality remains true dividing both sides by $\abs{\log\eps}$ and passing both sides to the $\liminf$, i.e., inequality \eqref{eq:liminfGeps} holds.
	
	In view of Lemma~\ref{lemma:truncation} and, more specifically, of \eqref{trunc:J}, to prove \eqref{eq:conv-jac} it is enough to prove that
	\[
		J(u_\eps, A_\eps) - J(u_\eps, B_\eps) \to 0 \quad \mbox{in } W^{-1,p}(M) \mbox{ for any } p \mbox{ such that } 1 \leq p < \frac{n}{n-1}
	\]
	as $\eps \to 0$. The case $p = 1$ is proved in Lemma~\ref{lemma:conv-jac-w^(-1,1)}. We will rely on that case and interpolation to extend the same conclusion to all $1\leq p < \frac{n}{n-1}$. To this purpose, we notice that~\eqref{trunc5}, \eqref{G:hp} and H\"older's inequality imply
	\begin{equation}\label{eq:L1-est-Jac(u_eps,A_eps)}
		\norm{J(u_\eps, A_\eps)}_{L^1(M)} \lesssim \abs{\log\eps}.	
	\end{equation}
	On the other hand, from~\eqref{trunc5}, \eqref{G:hp}, and \eqref{eq:bound-B-aux-energy} it follows
	\begin{equation}\label{eq:L1-est-Jac(u_eps,B_eps)}
		\norm{J(u_\eps, B_\eps)}_{L^1(M)} \lesssim \abs{\log\eps}.
	\end{equation}
	Thus, by Lemma~\ref{lemma:conv-jac-w^(-1,1)}, \eqref{eq:L1-est-Jac(u_eps,A_eps)}, \eqref{eq:L1-est-Jac(u_eps,B_eps)}, and Lemma~\ref{lemma:interpolation} we get
	\[
		\norm{J(u_\eps, A_\eps) - J(u_\eps, B_\eps)}_{W^{-1,p}(M)} \to 0
	\]
	as $\eps \to 0$, for any $1 \leq p < \frac{n}{n-1}$.
\end{proof}

\begin{remark}\label{rk:reduction}
In view of~\eqref{eq:liminfGeps} and~\eqref{eq:conv-jac}, it is always possible to pass to an optimised sequence in the proof of Statement~{(i)} of Theorem~\ref{maingoal:G} (i.e., to replace, in the proof, the given sequence $\{(u_\eps, A_\eps)\}$ by any associated optimised sequence). As we will see later, a decisive advantage of doing so is that the curvature of the connections of an optimised sequence satisfy the London equation~\eqref{London}.
 %Therefore, we will assume $\{(u_\eps, A_\eps)\}$ is already optimised according to the previous lemmas, and we shall always do so in the sequel. The neat advantage of doing so relies in that, we can assume $A_\eps$ satisfies the London equation~\eqref{London}, and this extremely simplifies the proof of Theorem~\ref{maingoal:G}.
\end{remark}

%\begin{remark}
%	We record for future reference that, of course, any sequence $\{(u_\eps, A_\eps)\}$ of pairs $(u_\eps, A_\eps)$  minimizing, for each $\eps > 0$, the corresponding energy functional $G_\eps$ is already optimised.
%\end{remark}

It is convenient to introduce the following notation: if $\Phi \in W^{2,2}(M,\,\S^1)$ is a gauge transformation and $A \in W^{1,2}(M, \T^*M)$, we set (cf.~Remark~\ref{rk:pullback-Phi})
\begin{equation}\label{eq:A-gauged}
	\Phi \cdot A := A + \Phi^*(\vol_{\mathbb{S}^1}) 
	= A - i \Phi^{-1} \d \Phi.
\end{equation}

The next technical lemma produces a family of gauge transformations playing a crucial r\^{o}le in the whole rest of this section.
\begin{lemma}\label{lemma:bound_A_W12}
	Let $\{(u_\eps,A_\eps)\} \subset (L^\infty \cap W^{1,2})(M,E)\times W^{1,2}(M,\T^*M)$ be an optimised sequence satisfying \eqref{G:hp}. Then, for each $\eps > 0$ there exists gauge transformations $\Phi_\eps \in W^{2,2}(M,\, \S^1)$ so that there holds
	\begin{equation}\label{bounds_A_W12}
		\norm{\Phi_\eps \cdot A_\eps}_{W^{1,2}(M)} \lesssim \abs{\log\eps}^{1/2}
	\end{equation}
	for all $\eps > 0$, whence
	\begin{equation}\label{eq:bound_Eeps}
		E_\eps(\Phi_\eps u_\eps) \lesssim \abs{\log\eps}
	\end{equation}
	for all $\eps > 0$.
\end{lemma}

\begin{proof}
	By Lemma~\ref{lemma:choiceofgauge}, for all $\eps > 0$ we find a gauge transformation such that we can write $\Phi_\eps \cdot A_\eps = \d^*\psi_\eps + \zeta_\eps$, where $\zeta_\eps \in \Harm^1(M)$. Indeed, denote $A_\eps = \d \varphi_\eps + \d^* \psi_\eps + \xi_\eps$ the Hodge decomposition of $A_\eps$, where $\varphi \in W^{2,2}(M)$ is co-exact, $\psi \in W^{2,2}(M,\Lambda^2 \T^*M)$ is exact, and $\xi \in \Harm^1(M)$. Then, Lemma~\ref{lemma:choiceofgauge} provides us, for each $\eps > 0$, with a gauge transformation $\Phi_\eps \in W^{2,2}(M,\,\S^1)$ and a harmonic 1-form $\bar{\xi}_\eps$ so that $\Phi_\eps^*(\vol_{\S^1}) = - \d \varphi_\eps - \bar{\xi}_\eps$. In addition, by Remark~\ref{rk:choiceofgauge}, $\|\zeta_\eps\|_{L^\infty(M)} \leq C_M$, where $\zeta_\eps := \xi_\eps - \bar{\xi}_\eps$ and $C_M$ is a constant depending only on $M$ (whence of course $\sup_{\eps > 0} \norm{\zeta_\eps}_{L^\infty(M)} \leq C_M$). %Moreover, by Hodge decomposition~\eqref{eq:LpHodgeDec}, we may assume that (for each $\eps > 0$) $\psi_\eps \in W^{2,2}(M, \Lambda^2 \T^* M)$ is exact. %, hence closed and orthogonal to all harmonic $2$-forms. 
	Let $\eps > 0$ be arbitrary, and recall that we denote $F_0$ the curvature 2-form of the reference connection $\D_0$. As $\psi_\eps$ is exact and of class $W^{2,2}$, there holds
	\begin{equation}\label{eq:laplace_psi}
		-\Delta \psi_\eps = \d \d^* \psi_\eps = \d( \Phi_\eps \cdot A_\eps ) = \d A_\eps = F_{A_\eps} - F_0,
	\end{equation}
	by~\eqref{eq:curvature}. %, where $F_0$ is the curvature of the reference connection $\D_0$. 
	In particular, $F_{A_\eps} - F_0$ is an exact $2$-form (hence, orthogonal to every harmonic $2$-form) and therefore we deduce from Lemma~\ref{lemma:classic-elliptic-reg} that for all $\eps > 0$
	\[
		\norm{\psi_\eps}_{W^{2,2}(M)} \lesssim \norm{F_\eps - F_0}_{L^2(M)} \lesssim \abs{\log\eps}^{1/2},
	\]
	where the last inequality follows from~\eqref{G:hp}. Hence, recalling~\eqref{eq:gaffney-closed},
	\[
		\norm{\Phi_\eps \cdot A_\eps}_{W^{1,2}(M)} \lesssim \abs{\log\eps}^{1/2}
	\]
	for all $\eps > 0$; i.e., \eqref{bounds_A_W12} is proved. From \eqref{bounds_A_W12}, \eqref{G:hp}, and the gauge-invariance of $G_\eps$,
	\[
	\begin{split}
		E_\eps(\Phi_\eps u_\eps) &= \int_M \left( \frac{1}{2}\abs{\D_0 (\Phi_\eps u_\eps)}^2 + \frac{1}{4\eps^2} \left( 1 - \abs{\Phi_\eps u_\eps}^2 \right)^2 \right) \,\vol_g \\
		%&= \int_M \left( \abs{\D_{\Phi_\eps \cdot A_{\eps} \Phi_\eps u_\eps}^2 + \ip{\D_0 (\Phi_\eps u_\eps)}{i (\Phi_\eps \cdot A_\eps) \Phi_\eps u_\eps} - \frac{1}{2}\abs{(\Phi_\eps \cdot A_\eps) \Phi_\eps u_\eps}^2  + \frac{1}{4\eps^2} \left( 1 - \abs{\Phi_\eps u_\eps}^2\right)^2 \right) \,\vol_g \\
		&\leq \int_M \left( \abs{\D_{\Phi_\eps \cdot A_\eps} (\Phi_\eps u_\eps)}^2 + \abs{\Phi_\eps \cdot A_\eps}^2  + \frac{1}{4\eps^2} \left( 1 - \abs{\Phi_\eps u_\eps}^2\right)^2 \right)  \,\vol_g \\
		&\leq \int_M\left( \abs{\D_{\Phi_\eps\cdot A_\eps} (\Phi_\eps u_\eps)}^2 + \abs{F_{\Phi_\eps \cdot A_\eps}}^2  + \frac{1}{4\eps^2} \left( 1 - \abs{\Phi_\eps u_\eps}^2\right)^2 \right)  \,\vol_g + \int_M \abs{\Phi_\eps \cdot A_\eps}^2 \,\vol_g \\
		&\leq 2G_\eps(\Phi_\eps u_\eps, \Phi_\eps \cdot A_\eps) + \norm{\Phi_\eps \cdot A_\eps}^2_{L^2(M)} \\
		&= 2G_\eps(u_\eps, A_\eps) + \norm{\Phi_\eps \cdot A_\eps}^2_{L^2(M)} \lesssim \abs{\log\eps}
	\end{split}
	\]
	for all $\eps > 0$.
\end{proof}

\begin{remark}
	In general, we cannot assert that the original sequence~$\{A_\eps\}$ satisfies~\eqref{bounds_A_W12}, as we have no control on the co-exact part of the Hodge decomposition of the maps $A_\eps$ (i.e., on $\varphi_\eps$).
\end{remark}

We immediately make use of Lemma~\ref{lemma:bound_A_W12} to prove the following corollary.

\begin{corollary}\label{cor:convergence-jacobians}
	Let $\{(u_\eps, A_\eps)\} \subset (L^\infty \cap W^{1,2})(M, E) \times W^{1,2}(M,\T^*M)$ be an optimised sequence satisfying \eqref{G:hp}. Then, up to extraction of a (not relabelled) subsequence, $J(u_\eps, A_\eps) \to \pi J_*$ in $W^{-1,p}(M)$ for any $1 \leq p < \frac{n}{n-1}$ as $\eps \to 0$, where $J_*$ is a bounded measure with values in $2$-forms. In addition, $\star J_*$ is an integer-multiplicity rectifiable $(n-2)$-cycle belonging to $\mathcal{C}$.
\end{corollary}

\begin{proof}
	Let $\{\Phi_\eps\}$ be the family of gauge transformations in Lemma~\ref{lemma:bound_A_W12}. In view of~\eqref{eq:bound_Eeps}, we can apply Theorem~\ref{maingoal:E} to the sequence $\{\Phi_\eps u_\eps\}$. By Statement~(i) of Theorem~\ref{maingoal:E}, we obtain a (not relabelled) subsequence $\{\Phi_\eps u_\eps\}$ and a bounded measure $J_*$ with values in $2$-forms such that $J(\Phi_\eps u_\eps) \to \pi J_*$ in $W^{-1,p}(M)$ for any $1 \leq p < \frac{n}{n-1}$ as $\eps \to 0$.
%Let then $J_*$ be the bounded measure with values in $2$-forms associated with the {\BBB{(not relabelled sub-)}}sequence $\{\Phi_\eps u_\eps\}$ by Statement~{(i)} of Theorem~\ref{maingoal:E}. Then, $J(\Phi_\eps u_\eps) \to \pi J_*$ in $W^{-1,p}(M)$ for any $1 \leq p < \frac{n}{n-1}$ as $\eps \to 0$. 
Since we already know from Statement~{(i)} of Theorem~\ref{maingoal:E} that $\star J_*$ is a current with the desired properties, it only remains to prove that $J(u_\eps, A_\eps) \to \pi J_*$ in $W^{-1,p}(M)$ for any $1 \leq p < \frac{n}{n-1}$ as $\eps \to 0$.
	
	 To prove this, we observe that, by the gauge-invariance of the Jacobians, we have $J(u_\eps, A_\eps) = J(\Phi_\eps u_\eps, \Phi_\eps \cdot A_\eps)$ for all $\eps > 0$. Thus, %for $J_*$ as in the statement, we have
	\[
	\begin{split}
		J(u_\eps, A_\eps) - J_* &= J(\Phi_\eps u_\eps, \Phi_\eps \cdot A_\eps) - J_* \\
		&= J(\Phi_\eps u_\eps, \Phi_\eps \cdot A_\eps) - J(\Phi_\eps u_\eps) + J(\Phi_\eps u_\eps) - J_*. 
	\end{split}
	\]
	for all $\eps > 0$. By definition of $J_*$, we need only to show that 
	\begin{equation}\label{eq:conv-jac-aux}
		J(\Phi_\eps u_\eps, \Phi_\eps \cdot A_\eps) - J(\Phi_\eps u_\eps)  \to 0 \quad \mbox{in } W^{-1,p}(M) \mbox{ for any } 1 \leq p < \frac{n}{n-1} \mbox{ as } \eps \to 0.
	\end{equation} 
	To this purpose, we recall that from \eqref{eq:diff-jJAB} we have 
	\[
		J(\Phi_\eps u_\eps, \Phi_\eps \cdot A_\eps) - J(\Phi_\eps u_\eps) = -\frac{1}{2}\d\left((\Phi_\eps \cdot A_\eps)(1-\abs{\Phi_\eps u_\eps}^2)\right) %= -\frac{1}{2}\d\left((\Phi_\eps \cdot A_\eps)(1-\abs{u_\eps}^2)\right)
	\] 
	for all $\eps > 0$. By \eqref{bounds_A_W12} and the energy estimate \eqref{eq:bound_Eeps} (or, equivalently, \eqref{G:hp}, since $\abs{\Phi_\eps u_\eps} = \abs{u_\eps}$ a.e.), for each $\eps > 0$ it holds
	\[
		\norm{(\Phi_\eps \cdot A_\eps)( 1- \abs{\Phi_\eps u_\eps}^2)}_{L^1(M)} \leq \norm{\Phi_\eps \cdot A_\eps}_{L^2(M)} \norm{1-\abs{\Phi_\eps u_\eps}^2}_{L^2(M)} \lesssim \eps \abs{\log \eps},
	\]
 whence
	\[
		\norm{J(\Phi_\eps u_\eps, \Phi_\eps \cdot A_\eps) - J(\Phi_\eps u_\eps)}_{W^{-1,1}(M)} \lesssim \eps \abs{\log\eps},
	\] 
	for all $\eps > 0$. Then \eqref{eq:conv-jac-aux} follows by using Lemma~\ref{lemma:interpolation} exactly as in the last part of Lemma~\ref{lemma:truncation}. The conclusion is now immediate by triangle inequality.
\end{proof}

\begin{remark}\label{rk:conv-jac}
	In view of~\eqref{eq:conv-jac}, Corollary~\ref{cor:convergence-jacobians} proves a half of Statement~(i) of Theorem~\ref{maingoal:G}.
\end{remark}

The following proposition is the last piece of information we need to combine the above results with Theorem~\ref{maingoal:E} to deduce the lower bound~\eqref{eq:lower-bound-Geps-intro}, concluding the proof of Theorem~\ref{maingoal:G}. Here we will use in a crucial way the fact that the curvatures $F_{A_\eps}$ satisfy the London equation~\eqref{London}, as alluded in Remark~\ref{rk:reduction}.
\begin{prop}\label{prop:boundedness-Aeps-W2p}
	Let $\{(u_\eps, A_\eps)\} \subset(L^\infty \cap W^{1,2})(M,\,E) \times W^{1,2}(M,\,\T^*M)$ be an optimised sequence satisfying \eqref{G:hp}. Then, up to extraction of a (not relabelled) subsequence:
	\begin{enumerate}[(i)]
		\item $\{F_{A_\eps}\}$ is a Cauchy sequence in $W^{1,p}(M)$ for any $1 \leq p < \frac{n}{n-1}$.
		\item If $\{\Phi_\eps\}$ is the family of gauge transformations in Lemma~\ref{lemma:bound_A_W12}, the sequence $\{ \Phi_\eps \cdot A_\eps \}$ is a bounded sequence in $W^{2,p}(M,\,T^*M)$ for any $1 \leq p < \frac{n}{n-1}$.
		\item Up to a further (not relabelled) subsequence, $\{\Phi_\eps \cdot A_\eps\}$ converges strongly in $W^{2,p}(M,\,\T^*M)$, for any $1 \leq p < \frac{n}{n-1}$, to some $A_*$, where $A_*$ writes
		\begin{equation}\label{eq:A_*}
		A_* = \d^* \psi_* + \zeta_*
		\end{equation}
		for $\psi_* \in W^{3,p}(M,\,\Lambda^2 \T^*M)$ an exact 2-form and $\zeta_*$ a harmonic 1-form.
\end{enumerate}	
\end{prop}

\begin{proof}
Let us set $F_\eps := F_{A_\eps}$ for convenience. Since each of the curvatures $F_\eps$ satisfies the London equation \eqref{London}, we have
	\[
		-\Delta (F_{\eps_n} - F_{\eps_m}) + (F_{\eps_n}-F_{\eps_m}) = 2J(u_{\eps_n}, A_{\eps_n}) - 2J(u_{\eps_m}, A_{\eps_m})
	\]
	for all indexes $n$, $m \in \N$. By Lemma~\ref{lemma:elliptic_reg}, we deduce the crucial estimate
	\[
		\norm{F_{\eps_n} - F_{\eps_m}}_{W^{1,p}(M)} \lesssim \norm{J(u_{\eps_n}, A_{\eps_n}) - J(u_{\eps_m}, A_{\eps_m})}_{W^{-1,p}(M)},
	\]
	where the right hand side tends to zero as $n, m \to \infty$ as a trivial consequence of Corollary~\ref{cor:convergence-jacobians}, and this holds for any $1 \leq p < \frac{n}{n-1}$. Thus, $\{F_{A_\eps}\}$ is a Cauchy sequence in $W^{1,p}(M)$ for any $1 \leq p < \frac{n}{n-1}$. This proves~(i). 
	
	To prove (ii), arguing as in Lemma~\ref{lemma:bound_A_W12}, by the choice of $\{\Phi_\eps\}$, for each gauge-transformed connection $\Phi_\eps \cdot A_\eps$ we have the decomposition $\Phi_\eps \cdot A_\eps = \d^* \psi_\eps + \zeta_\eps$ for an appropriate exact $2$-form $\psi_\eps$ satisfying~\eqref{eq:laplace_psi} and a suitable harmonic $2$-form $\zeta_\eps$. Furthermore, it holds $\sup_{\eps > 0} \norm{\zeta_\eps}_{L^\infty(M)} < +\infty$ (cf. Remark~\ref{rk:choiceofgauge}). By~\eqref{eq:laplace_psi}, Part~(i) and elliptic regularity (i.e., Lemma~\ref{lemma:classic-elliptic-reg}) it follows that $\{ \psi_\eps \}$ is a bounded sequence in $W^{3,p}(M)$ for any $1 \leq p < \frac{n}{n-1}$. Hence, again by \eqref{eq:laplace_psi} and $L^p$-Hodge decomposition (Proposition~\ref{prop:LpHodgeDec}), the sequence $\{\Phi_\eps \cdot A_\eps\}$ is bounded in $W^{2,p}(M)$ for any $1 \leq p < \frac{n}{n-1}$. 
	
	(iii) Writing again $\Phi_\eps \cdot A_\eps = \d^* \psi_\eps + \zeta_\eps$ exactly as in (ii), from \eqref{eq:laplace_psi} we have 
	\[
		- \Delta (\psi_{\eps_n} - \psi_{\eps_m}) = F_{\eps_n} - F_{\eps_m}
	\]
	for any $n$, $m \in \N$. Therefore, by Part~{(i)} it follows that $\{\Delta \psi_{\eps_n}\}$ is a Cauchy sequence in $W^{1,p}(M,\,\Lambda^2\T^*M)$ and, since each $\psi_{\eps_n}$ is exact, by Lemma~\ref{lemma:classic-elliptic-reg} and Remark~\ref{rk:est-exact-forms} it follows that $\{\psi_{\eps_n}\}$ is Cauchy sequence in $W^{3,p}(M,\,\Lambda^2\T^*M)$, for any $1 \leq p < \frac{n}{n-1}$. Hence, by the completeness of $W^{3,p}(M,\,\Lambda^2 \T^*M)$, $\{\psi_{\eps_n}\}$ converges strongly in $W^{3,p}(M,\,\Lambda^2 \T^*M)$, for any $1 \leq p < \frac{n}{n-1}$, to some $\psi_*$. By the $L^p$-Hodge-decomposition (Proposition~\ref{prop:LpHodgeDec}) and the strong $W^{3,p}$-convergence, $\psi_*$ is still an exact 2-form. 
	
	Next, since the space $\Harm^1(M)$ is finite-dimensional and $\{\zeta_{\eps_n}\}$ is bounded in the $L^\infty(M)$-norm, $\{\zeta_{\eps_n}\}$ is also bounded with respect to the $W^{2,p}(M)$-norm. In addition, we can also extract from $\{\zeta_{\eps_n}\}$ a (not relabelled) Cauchy sequence, which is \emph{a fortiori} a Cauchy sequence in $W^{2,p}(M,\,\T^*M)$, for any $1 \leq p < \frac{n}{n-1}$, and hence it converges in $W^{2,p}(M,\,\T^*M)$ to some $\zeta_*$. Again, by Hodge-decomposition and strong convergence, $\zeta_*$ is still a harmonic 1-form.
	
	Thus, up to a not relabelled subsequence, and letting $A_* := \d^* \psi_* + \zeta_*$, we have $\Phi_\eps \cdot A_\eps \to A_*$ in $W^{2,p}(M,\,\T^*M)$ as $\eps \to 0$, for any $1 \leq p < \frac{n}{n-1}$.
	%\[
	%	\norm{\Phi_{\eps_n} \cdot A_{\eps_n} - \Phi_{\eps_m} \cdot A_{\eps_m}}_{W^{2,p}(M)} \leq \norm{\d^* \psi_{\eps_n} - \d^*\psi_{\eps_m}}_{W^{2,p}(M)} + \norm{\zeta_n - \zeta_m}_{W^{2,p}(M)},
	%\]	
	%which implies that 
	%$\{\Phi_{\eps} \cdot A_{\eps}\}$ is a Cauchy sequence in $W^{2,p}(M,\,\T^*M)$.}
\end{proof}

\begin{remark}\label{rk:crit-points1}
	Lemma~\ref{lemma:bound_A_W12}, Corollary~\ref{cor:convergence-jacobians}, and Proposition~\ref{prop:boundedness-Aeps-W2p} are still valid, with the same proof, for any sequence  $\{(u_\eps,\,A_\eps)\} \subset (W^{1,2}\cap L^\infty)(M,\,E)\times W^{1,2}(M,\,\T^*M)$ of \emph{critical points} of $G_\eps$ that satisfies the logarithmic energy bound \eqref{G:hp}. Indeed, 
% 	as it is easy to check and as we will see in details in \cite{CanevariDipasqualeOrlandiII}, 
	for any critical pair $\{(u_\eps,\,A_\eps)\}$ with~$u_\eps\in L^\infty(M, \, E)$, $u_\eps$ satisfies $\norm{u_\eps}_{L^\infty(M)} \leq 1$ (by maximum principle, as in~e.g.~\cite[Proposition~II.2]{BethuelRiviere})
	and $F_{A_\eps}$ satisfies the London equation~\eqref{London}. These two facts are all that we are really using about optimised sequences in the proof of the mentioned results.
\end{remark}

We now have at disposal everything we need to prove Theorem~\ref{maingoal:G}.
\begin{proof}[Proof of Theorem~\ref{maingoal:G}]
(i) As emphasized in Remark~\ref{rk:reduction}, we can always pass to an optimised sequence associated with $\{(u_\eps, A_\eps)\}$. Hence, we may assume, for notational convenience, that $\{(u_\eps, A_\eps)\}$ is already optimised. Let $\{\Phi_\eps\} \subset W^{2,2}(M,\,\S^1)$ be, once again, the family of gauge transformations of Lemma~\ref{lemma:bound_A_W12}. By \eqref{eq:bound_Eeps}, we can apply Statement~(i) of Theorem~\ref{maingoal:E} to the sequence $\{\Phi_\eps u_\eps\}$. Let $J_*$ be the bounded measure with values in $2$-forms associated with $\{\Phi_\eps u_\eps\}$ by Statement~(i) of Theorem~\ref{maingoal:E}. The claimed convergence of the gauge-invariant Jacobians $J(u_\eps, A_\eps)$ to $J_*$ follows from Corollary~\ref{cor:convergence-jacobians} (and, back to the original sequence, Remark~\ref{rk:conv-jac}). Therefore, we still have to prove only the lower bound~\eqref{eq:lower-bound-Geps-intro}. 

To infer~\eqref{eq:lower-bound-Geps-intro}, we note that, by the gauge-invariance of $G_\eps(u_\eps, A_\eps)$, it suffices to prove that
 \begin{equation}\label{eq:aux-liminfG}
	\liminf_{\eps \to 0} \frac{E_\eps(\Phi_\eps u_\eps)}{\abs{\log\eps}} \leq \liminf_{\eps \to 0} \frac{G_\eps(\Phi_\eps u_\eps, \Phi_\eps \cdot A_\eps)}{\abs{\log\eps}},
\end{equation}
%where $\{\Phi_\eps\}$ is (again) the family of gauge transformations in the statement of Lemma~\ref{lemma:bound_A_W12} (and hence of Proposition~\ref{prop:boundedness-Aeps-W2p}, too). Indeed, by estimate~\eqref{eq:bound_Eeps}, we can apply Statement~(i) of Theorem~\ref{maingoal:E} 
Once we proved~\eqref{eq:aux-liminfG}, the conclusion follows immediately from Statement~(i) of Theorem~\ref{maingoal:E}.

Towards the proof of~\eqref{eq:aux-liminfG}, we notice that, by interpolation, Proposition~\ref{prop:boundedness-Aeps-W2p}, and the (continuous) Sobolev embedding $W^{1,2}(M, \T^* M) \hookrightarrow L^{2^*}(M, \T^*M)$ \cite[Theorem~1.3.3]{SchwarzG}, for each $\eps > 0$ we have
\[
\begin{split}
	\norm{\Phi_\eps \cdot A_\eps}_{L^2(M)} &\leq \norm{\Phi_\eps \cdot A_\eps}_{L^1(M)}^{\frac{2}{n+2}} \norm{\Phi_\eps \cdot A_\eps}_{L^{2^*}(M)}^{\frac{n}{n+2}} \\
	& \lesssim \norm{\Phi_\eps \cdot A_\eps}^{\frac{2}{n+2}}_{W^{2,1}(M)} \norm{\Phi_\eps \cdot A_\eps}_{W^{1,2}(M)}^{\frac{n}{n+2}} \lesssim \abs{\log{\eps}}^{\frac{n}{2(n+2)}},
\end{split}
\]
whence $\norm{\Phi_\eps \cdot A_\eps}_{L^2(M)} = \mathrm{o}\left(\abs{\log\eps}^\frac{1}{2}\right)$ as $\eps \to 0$. Since 
\begin{multline*}
	\int_M \frac{1}{2} \abs{\D_{\Phi_\eps \cdot A_\eps} (\Phi_\eps u_\eps)}^2 \,\vol_g \\
	= \int_M \frac{1}{2}\abs{\D_0 (\Phi_\eps u_\eps)}^2 + \ip{\D_0 (\Phi_\eps u_\eps)}{i (\Phi_\eps \cdot A_\eps)(\Phi_\eps u_\eps)} + \frac{1}{2} \abs{(\Phi_\eps \cdot A_\eps) (\Phi_\eps u _\eps)}^2 \,\vol_g,
\end{multline*}
and
\begin{multline*}
	\abs{\int_M \ip{\D_0 (\Phi_\eps u_\eps)}{i (\Phi_\eps \cdot A_\eps) (\Phi_\eps u_\eps)} + \frac{1}{2}\abs{(\Phi_\eps \cdot A_\eps) (\Phi_\eps u_\eps)}^2 \,\vol_g} \\
	\lesssim \norm{\D_0 (\Phi_\eps u_\eps)}_{L^2(M)} \norm{\Phi_\eps \cdot A_\eps}_{L^2(M)} + \norm{\Phi_\eps \cdot A_\eps}_{L^2(M)}^2,
\end{multline*}
we deduce
\begin{equation}\label{liminfG}
	\liminf_{\eps \to 0} \frac{E_\eps(\Phi_\eps u_\eps)}{\abs{\log\eps}} \leq \liminf_{\eps \to 0} \frac{G_\eps(\Phi_\eps u_\eps, \Phi_\eps \cdot A_\eps)}{\abs{\log\eps}} = \liminf_{\eps \to 0}  \frac{G_\eps(u_\eps, A_\eps)}{\abs{\log\eps}} .
\end{equation}
As remarked, the claimed conclusion now follows from (i) of Theorem~\ref{maingoal:E} (and, back to the original, not necessarily optimised, sequence $\{(u_\eps,A_\eps)\}$, inequality \eqref{eq:liminfGeps}).

(ii) Given any $S_* \in \mathcal{C}$, let $J_* := \star S_*$ be the corresponding measure with values in $2$-forms $J_*$ and denote $\{u_\eps\} \subset W^{1,2}(M,E)$ the recovery sequence given by (ii) of Theorem~\ref{maingoal:E}. Then the sequence $\{(u_\eps, 0)\} \subset W^{1,2}(M,E) \times W^{1,2}(M, \T^*M)$ satisfies $J(u_\eps, 0) \to \pi J_*$ in $W^{-1,p}(M)$ for any $1 \leq p < \frac{n}{n-1}$ as $\eps \to 0$ (simply because $J(u_\eps,0) = J(u_\eps)$ by definition). In addition, there holds
\[
	\limsup_{\eps \to 0} \frac{G_\eps(u_\eps,0)}{\abs{\log{\eps}}} 
	\leq %\pi \M(S_*) = 
	\pi \abs{J_*}(M),
\]
because the inequality holds for $E_\eps(u_\eps)$ by (ii) of Theorem~\ref{maingoal:E} and because for all $\eps > 0$ we have $G_\eps(u_\eps,0) = E_\eps(u_\eps) + \frac{1}{2}\norm{F_0}^2_{L^2(M)}$, where as usual $F_0$ is the curvature of the reference connection $\D_0$. The proof is finished.
\end{proof}

\begin{remark} \label{rk:local_liminf-G}
 There is a local analogue of the $\Gamma$-lower inequality
 for~$G_\eps$. More precisely, let~$\{(u_\eps, \, A_\eps)\}\subset 
 W^{1,2}(M, \, E)\times W^{1,2}(M, \, \T^* M)$ be a sequence
 that satisfies~\eqref{G:hp} and let~$V\subset M$ an arbitrary 
 open set. Up to extraction of a subsequence, assume
 that~$J(u_\eps, \, A_\eps)\to \pi J_*$ in~$W^{-1,p}(M)$
 for any~$p$ with~$1 \leq p < n/(n-1)$. Then, there holds
 \begin{equation} \label{local_liminf_G}
  \pi \abs{J_*}\!(V) \leq \liminf_{\eps\to 0} 
  \frac{G_\eps(u_\eps, \, A_\eps; \, V)}{\abs{\log\eps}}
 \end{equation}
 The proof of~\eqref{local_liminf_G} is completely analogous 
 to the proof of~\eqref{liminfG} above and is based on
 the local $\Gamma$-lower inequality for~$E_\eps$
 (Proposition~\ref{prop:liminfE}).
\end{remark}

Corollary~\ref{cor:London} is now an almost immediate 
consequence of Theorem~\ref{maingoal:G}.

\begin{proof}[Proof of Corollary~\ref{cor:London}]
 Let~$(u_\eps^{\min}, \, A_\eps^{\min})$ be a minimiser of~$G_\eps$
 in~$W^{1,2}(M, \, E)\times W^{1,2}(M, \, \T^*M)$, for any~$\eps > 0$. 
 As the class~$\mathcal{C}$ is non-empty, Theorem~\ref{maingoal:G}
 and a comparison argument imply that 
 \begin{equation} \label{corLondon1}
  G_\eps(u_\eps^{\min}, \, A_\eps^{\min}) \leq C \abs{\log\eps}
 \end{equation}
 for some~$\eps$-independent constant~$C$. 
 By applying Theorem~\ref{maingoal:G} again, we find a limit 
 $2$-form~$J_*$ such that, up to extraction of a subsequence,
 $J(u_\eps^{\min}, \, A_\eps^{\min})\to \pi J_*$ in~$W^{-1,p}(M)$
 for any~$p < \frac{n}{n-1}$. Moreover, $\star J_*\in\mathcal{C}$
 and, by general properties of~$\Gamma$-convergence,
 $\star J_*$ is a chain of minimal mass in~$\mathcal{C}$.
 
 As~$(u_\eps^{\min}, \, A_\eps^{\min})$ is a minimiser of~$G_\eps$
 it follows that~$\abs{u_\eps} \leq 1$ (by a truncation argument,
 along the lines of Lemma~\ref{lemma:truncation}) 
 and that~$\{(u_\eps^{\min}, \, A_\eps^{\min})\}$
 is an optimised sequence. Then, Proposition~\ref{prop:boundedness-Aeps-W2p}
 implies that the curvatures~$F_{A_\eps^{\min}}$ converge to a limit~$F_*$
 in~$W^{1,p}(M)$ for any~$p<\frac{n}{n-1}$.  
 Each~$F_{A_\eps^{\min}}$ satisfies the London equation
 \begin{equation} \label{corLondon2}
  -\Delta F_{A_\eps^{\min}} + F_{A_\eps^{\min}} 
   = 2J(u_\eps^{\min}, \, A_\eps^{\min})  
 \end{equation}
 (by the arguments of Proposition~\ref{prop:existence-min-F}).
 By passing to the limit as~$\eps\to 0$ in~\eqref{corLondon2},
 it follows that~$F_*$ satisfies the London equation
 \begin{equation} \label{Londonmille}
  -\Delta F_* + F_* = 2\pi J_*  
 \end{equation}
 It only remains to prove that the rescaled energy
 densities~$\mu_\eps := \mu_\eps(u_\eps^{\min}, \, A_\eps^{\min})$
 (defined in~\eqref{mu_eps}) converge to the total variation~$\pi\abs{J_*}$.
 Due to~\eqref{corLondon1}, we can extract a subsequence in such a
 way that~$\mu_\eps\rightharpoonup\mu_*$, weakly as measures, as~$\eps\to 0$.
 Equation~\eqref{local_liminf_G} implies that
 \begin{equation*} %\label{corLondon2}
  \pi\abs{J_*}\!(V) \leq \mu_*(V)
 \end{equation*}
 for any open set~$V\subseteq M$ such that~$\mu_*(\partial V) = 0$.
 As~$\mu_*$ is a bounded Borel measure, it follows that~$\pi\abs{J_*}\leq \mu_*$
 as measures. However, Statement~(ii) in Theorem~\ref{maingoal:G} 
 implies that
 \[
  \mu_*(M) = \lim_{\eps\to 0} 
  \frac{G_\eps(u_\eps^{\min}, \, A_\eps^{\min})}{\abs{\log\eps}} 
  \leq \pi\abs{J_*}\!(M)
 \]
 Therefore, $\mu_* = \pi\abs{J_*}$, as claimed.
\end{proof}

\begin{remark}
	For any sequence $\{(u_\eps,\,A_\eps)\} \subset (W^{1,2}\cap L^\infty) (M,\,E) \times W^{1,2}(M,\,\T^*M)$ of \emph{critical points} of $G_\eps$ that satisfy the logarithmic energy bound~\eqref{G:hp}, there exist bounded measures~$J_*$, $F_*$, with values in~$2$-forms,
	and a (non-relabelled) subsequence such that 
    \[
    J(u_\eps, \, A_\eps)\to \pi J_* 
    \quad \textrm{in } W^{-1,p}(M), \qquad
    F_{A_\eps} \to F_* \quad \textrm{in } W^{1,p}(M)
    \]
    for any~$p < n/(n-1)$. Moreover, $F_*$ satisfies the London equation~\eqref{Londonmille}. The proof of this claim follows 
    by the same arguments we used in the proof 
    of Corollary~\ref{cor:London}, word by word.
    Indeed, the assumption that~$(u_\eps^{\min}, \, A_\eps^{\min})$
    is  a sequence of minimisers is only needed to show that the limit~$J_*$ 
    is area-minimising in its homology class. 
    As for the rest, the arguments of Corollary~\ref{cor:London}
    only depend on the fact that~$F_{A_\eps}$ satisfies
    the London equation~\eqref{London} and that Proposition~\ref{prop:boundedness-Aeps-W2p} can be applied. But, according to Remark~\ref{rk:crit-points1}, both these facts continue to hold for any sequence of critical points satisfying \eqref{G:hp}.
\end{remark}

\paragraph*{Acknowledgements.}

G.C. and G.O. were partially supported by GNAMPA-INdAM.

\paragraph*{Data Availability Statement.} 

Data sharing not applicable to this article as no datasets were
generated or analysed during the current study.

\paragraph*{Declarations.} 

The authors have no competing interests to declare that are relevant to the content
of this article.

%====================================================
% 2022/03/26: Per ora, le equazioni EL non ci servono
%====================================================
%\appendix

%\section{Euler-Lagrange equations}
%\begin{definition}
%	A pair $(u, \D_A) \in W^{1,2}(M, E) \times W^{1,2}(M, \T^*M \otimes E)$ is a \emph{critical point} of the functional \eqref{magneticGL} if and only if 
%	\[
%		\left[\pd{}{t} G(u+tv, \D_A+sB) + \pd{}{s} G(u+tv, \D_A+sB)\right]_{t,s=0} = 0
%	\] 
%	for every section $v \in W^{1,2}(M,E)$ and every $B \in W^{1,2}(M, {\rm Ad} E)$.
%\end{definition}
%
%\begin{lemma}
%	A pair $(u, \D_A) \in W^{1,2}(M, E) \times W^{1,2}(M, \T^*M \otimes E)$ is a critical point of the functional \eqref{magneticGL} if and only if it solves
%	\[
%		\int_M \ip{D_A u}{D_A v} + \frac{1}{\varepsilon^2} (1-\abs{u}^2) \ip{u}{v} \vol_g = 0
%	\]
%	for all $v \in W^{1,2}(M,E)$, and
%	\[
%		\int_M \ip{F_A}{D_A B} \vol_g = 0,
%	\]
%	or, equivalently, if and only if $(u, \D_A)$ solves the system \eqref{EL-u}, \eqref{EL-A} when tested against arbitrary variations $(v, B) \in W^{1,2}(M,E) \times W^{1,2}(M, {\rm Ad} E)$.
%\end{lemma}

%\begin{proof}
%	The proof is standard but apparently lacking in the literature, so we carry it out in full details for the readers convenience.
%	
%	\emph{Step 1.} Given $(u, \D_A) \in W^{1,2}(M, E) \times W^{1,2}(M, \T^*M \otimes E)$, for $t,s \in \R$ and $(v, B) \in W^{1,2}(M, E) \times W^{1,2}(M, \T^*M \otimes {\rm Ad} E)$, we set $u_t := u + tv$, $\D_A^{(s)} := \D_A + s B$, and $\Phi_\varepsilon(t,s) := G_\eps(u_t, \D_A^{(s)})$.
%\end{proof}

\begin{appendix}

\section{Bundles on a closed Riemannian manifold}
\label{sect:bundles}

\subsection{Sobolev spaces of sections and differential forms}\label{subsec:sobolev}

%In this work, we shall be concerned with several spaces of sections of Hermitian vector bundles and differential forms. 
We recall below the main definitions and facts concerning several spaces of differential forms and, more broadly, of section of Hermitian vector bundles. We shall do so in a slightly greater generality than strictly needed in this work to make the presentation more transparent and the comparison with the relevant literature easier. The main reference for this appendix is \cite{Palais}, especially Chapters~4,~5,~9, and~19, where the abstract framework is developed in a much more general context (and using the language of category theory, that we avoid for reader's convenience). A more recent useful reference is \cite[Chapter~1]{Guneysu}. Both in \cite{Palais} and in \cite{Guneysu} equivalence with other approaches to Sobolev spaces of sections is discussed. (See also \cite[Chapter~10]{Nicolaescu}, \cite[Chapter~1]{SchwarzG}, \cite[Appendix~B]{Wehrheim}.)
\vskip3pt
Let $\K$ be either $\R$ or $\C$, and let $\pi : E \to M$ be a $\K$-vector bundle of rank $\ell$ over a $C^\infty$-smooth compact orientable Riemannian manifold $M = (M^n, g)$ without boundary. We assume the differentiable structure of $M$ is fixed once and for all. 

\paragraph{Regular bundle atlases.} 
Using the compactness of $M$, it is easily shown that it is always possible to find finite bundle atlases $\mathcal{A} = \{ (U_i, \varphi_i, \chi_i)\}_{i = 1}^N$, where $N \in \N$, the maps $\varphi_i \colon U_i \to \R^n$ are local charts and $\chi_i \colon  \pi^{-1}(U_i) =: E \vert_{U_i}  \to U_i \times \K^\ell$ local trivialisations, so that:
\begin{itemize}
	\item[(A1)] For all $i \in \{1,2,\dots,N\}$, $U_i$ is contractible and, moreover, $\varphi_i(U_i) =: \Omega_i \subset \R^n$ is a bounded contractible open set with smooth boundary.
	\item[(A2)] For all $i \in \{1,2,\dots,N\}$, $(U_i, \varphi_i)$ can be extended to a smooth chart $(V_i,\psi_i)$ contained in the differentiable structure of $M$, so that $\overline{U_i} \subset V_i$ and $\varphi_i = \psi_i \vert_{U_i}$. This ensures that, for all $i, j \in \{1, 2, \dots, N\}$, whenever $U_i \cap U_j \neq \varnothing$, each coordinate change $\varphi_i \circ \varphi_j^{-1} : \varphi_j(U_i \cap U_j) \to \varphi_i(U_i \cap U_j)$ is smooth up to the boundary of $\varphi_j(U_i \cap U_j)$ (hence, all its derivatives are bounded).
	\item[(A3)] For all $i \in \{1,2,\dots,N\}$, $E \to M$ trivialises over $U_i$, i.e, we have $E \vert_{U_i} \overset{\chi_i}{\simeq} U_i \times \K^\ell$ (actually, this comes for free, as the $U_i$ are contractible). In addition, each local trivialisation $\chi_i : \pi^{-1}(U_i) \to U_i \times \K^\ell$ %are smooth up to the boundary of $\overline{U}_i$. 
	extends to a smooth map over the corresponding coordinates patches $V_i$ associated with $U_i$ as in (A2) (up to slightly shrinking $V_i$, if necessary). 
	Consequently, for all $i, j \in \{1,2,\dots,N\}$ such that $U_i \cap U_j \neq \varnothing$, we have $\chi_i \circ \chi_j^{-1} \in C^\infty( \overline{(U_i \cap U_j)}\times \K^\ell, (U_i \cap U_j)\times \K^\ell)$. Recall that
	\[
		\widehat{\chi_i} := (\varphi_i, {\rm id}_{\K^\ell} ) \circ \chi_i : E \vert_{U_i} \to \Omega_i \times \K^\ell \qquad (i \in \{1,2,\dots,N\}),
	\]
	are the charts of $E$ (as an $(n + \ell)$-manifold if $\K = \R$, and as an $(n+2\ell)$ real manifold if $\K = \C$) and notice that, under our assumptions, all the derivatives of all coordinate changes $\widehat{\chi_i} \circ \widehat{\chi_j}^{-1} : \Omega_j \times \K^\ell \to \Omega_i \times \K^\ell$ are bounded. Consequently, denoting ${\rm pr}_2$ the projection onto the second factor of a product of the type $\Omega \times \K^\ell$, with $\Omega \subset \R^n$, the maps ${\rm pr}_2 \circ \widehat{\chi_i} \circ \widehat{\chi_j}^{-1} : \Omega_j \times \K^\ell \to \K^\ell$ are smooth and all their derivatives are bounded.
\end{itemize} 
For convenience, we refer to atlases of $M$ satisfying (A1) and (A2) as \emph{regular atlases} and to bundle atlases satisfying (A1)-(A3) as \emph{regular bundle atlases}. We call \emph{regular} the charts of regular atlases and of regular bundle atlases. 

\begin{remark}
	Since we always cover $M$ by contractible open sets $U_i$, \emph{every} vector bundle over $M$, and not only the given bundle $\pi : E \to M$, trivialises over them.
\end{remark}

%\begin{remark}\label{rk:part-unity}
%	If $\{(U_i,\varphi_i)\}_{i=1}^N$ is any atlas for $M$ satisfying (A1)-(A2) (for short, a \emph{regular atlas} for $M$), then every member of any partition of unity $\{\rho_i\}_{i = 1}^N$ subordinate to $\mathcal{A}$ belongs to $C^\infty(\overline{U_i})$.
%\end{remark}

\begin{remark}[Normal coordinates]\label{rk:normal-coordinates}
As it is well-know (see, e.g., \cite[pp.~166-167]{Spivak2}), around any point $x_0 \in M$, one can choose \emph{normal coordinates} so that in the geodesic ball $B_\delta(x_0)$ centered at $x_0$, it holds
 \[
 	g_{ij} = \delta_{ij} + \frac{1}{2}\sum_{k,l=1}^n \frac{\partial^2 g_{ij}}{\partial x^k \partial x_l}(x_0) x^k x^l + \mathrm{o}(\abs{x}^2)\,,
 \]
 whence $\sqrt{\det g} = 1 + \frac{1}{2} \sum_{i,k,l=1}^n \partial^2_{kl} g_{ii}(x_0) x^k x^l + \mathrm{o}(\abs{x}^2)$, and consequently
 \begin{equation}\label{eq:def-vol-norm-coord}
 	\vol_g = (1+\mathrm{O}(\delta^2))\d x\quad \mbox{in } B_\delta(x_0)\,.
 \end{equation}
Given any atlas of $M$, by compactness it can be refined so to have a regular atlas in which the local coordinates are normal coordinates.
\end{remark}

\paragraph{Sobolev spaces of sections of vector bundles}

%Now, assume $\pi : E \to M$ has given a fibre metric $h = \ip{\cdot}{\cdot}$ (Riemannian if $\K = \R$ and Hermitian if $\K = \C$). 
A \emph{measurable section} of a $\K$-vector bundle $E \to M$ of rank $\ell$ is a Borel measurable function $u\colon M \to E$ so that the equation $\pi \circ u(x) = x$ holds for almost every $x \in M$ with respect to the measure induced by $\vol_g$. Let us denote $\Gamma(M,\,E)$ the linear space of all such sections.

Fix a regular bundle atlas $\mathcal{A}$ for $E \to M$ and a partition of unity $\{\rho_i\}_{i=1}^N$ subordinate to $\mathcal{A}$. For $m \geq 0$ an integer and $p \in [1,\infty]$, we say (according to \cite[Chapter~4]{Palais}) that a measurable section $u \in \Gamma(M,\, E)$ belongs to $W^{m,p}_{\mathcal{A}}(M,\,E)$ if and only if every local representation 
\[
	u_i := {\rm pr}_2 \circ \widehat{\chi_i} \circ (\rho_i u ) \circ \varphi^{-1}_i : \Omega_i \subset \R^n \to \K^\ell \qquad (i \in {1,2,\dots,N})
\]
of $u$ belongs to the usual Sobolev space $W^{m,p}(\Omega_i, \,\K^\ell)$. We endow $W^{m,p}_{\mathcal{A}}(M,E)$ with the norm
\begin{equation}\label{eq:sobolev-norm-A}
	\norm{u}_{W^{m,p}_{\mathcal{A}}(M,E)} := \sum_{i=1}^N \sum_{0\leq \abs{\alpha} \leq m} \norm{ D^\alpha u_i}_{L^p(\Omega_i,\K^\ell)},
\end{equation}
where the inner sum runs over all multi-indexes $\alpha$ of length at most $m$. Of course, we let $L^p_{\mathcal{A}}(M,E) := W^{0,p}_{\mathcal{A}}(M,E)$.

\vskip3pt

Armed with the above definition and (A1)--(A3), it is not difficult to prove that (cf.~e.g.,~\cite{Palais, Guneysu, SchwarzG, Wehrheim})
\begin{itemize}
	\item[(S1)] $W^{m,p}_{\mathcal{A}}(M,E)$ is Banach space (Hilbert if $p = 2$), separable if $p \in [1,\infty)$, and reflexive if $p \in (1,\infty)$.
	\item[(S2)] $C^\infty(M,E)$ (i.e., the space of classical smooth sections of $E \to M$) is dense into $W^{m,p}_{\mathcal{A}}(M,E)$ for the norm $\norm{\cdot}_{W^{m,p}_{\mathcal{A}}(M,E)}$ for every integer $m \geq 0$ and every $p \in [1, \infty)$.
	\item[(S3)] The classical embedding (including compact embeddings) theorems hold.
	\item[(S4)] Any two partitions of unity subordinate to $\mathcal{A}$ induce equivalent norms, hence different choices for the partition of unity yield equivalent Banach spaces $W^{m,p}_{\mathcal{A}}(M,E)$, for which properties (S1)--(S3) hold.
	\item[(S5)] For any regular bundle atlases $\mathcal{A}_1$, $\mathcal{A}_2$ of $E \to M$, the sets $W^{m,p}_{\mathcal{A}_1}(M,E)$ and $W^{m,p}_{\mathcal{A}_2}(M,E)$ coincide. Moreover, the corresponding norms $\norm{\cdot}_{W^{m,p}_{\mathcal{A}_1}}$, $\norm{\cdot}_{W^{m,p}_{\mathcal{A}_2}}$, defined by \eqref{eq:sobolev-norm-A}, induce equivalent norms on $\Gamma(M,E)$. Then, $W^{m,p}_{\mathcal{A}_1}(M,E)$ and $W^{m,p}_{\mathcal{A}_2}(M,E)$ are actually equivalent Banach spaces, for which properties (S1)--(S4) hold.
\end{itemize}

\begin{remark}\label{rk:sobolev-topology}
In view of (S1)--(S5), we can identify a linear subspace $W^{m,p}(M,E)$ of $\Gamma(M,E)$, the members of which have finite $\norm{\cdot}_{W^{m,p}_\mathcal{A}}$-norm, independently of the chosen regular bundle atlas $\mathcal{A}$ (and of subordinate partitions of unity). All the norms $\norm{\cdot}_{W^{m,p}_\mathcal{A}}$ are equivalent on $W^{m,p}(M,E)$, and hence induce the same topology on $W^{m,p}(M,E)$, that is therefore independent of the regular bundle atlas $\mathcal{A}$ chosen for $E \to M$ (and of subordinate partitions of unity). In particular, properties (S1)--(S3) hold for $W^{m,p}(M,E)$, and in view of (S4)-(S5) we can choose for computations any regular bundle atlas and subordinate partitions of unity. 
\end{remark}

\begin{definition}\label{def:sobolev}
	Let $M$ be a smooth, compact, connected oriented, Riemannian manifold without boundary, of dimension $n \in \N$, and let $E \to M$ be a $\K$-vector bundle over $M$ of rank $\ell$. For $m \geq 0$ an integer and $p \in [1,\infty]$, we denote by $W^{m,p}(M,E)$ the linear space of sections of $E \to M$ having finite $\norm{\cdot}_{W^{m,p}_\mathcal{A}(M,E)}$-norm for some, and hence all (see Remark~\ref{rk:sobolev-topology}), choices of a regular bundle atlas $\mathcal{A}$. We provide the space $W^{m,p}(M,E)$ with the topology induced by any of the (equivalent) norms $\norm{\cdot}_{W^{m,p}_{\mathcal{A}(M,E)}}$. This topology does not depend on the chosen regular bundle atlas $\mathcal{A}$ (see again Remark~\ref{rk:sobolev-topology}). 
	
	We denote by $W^{-m,q}(M,\,E^\prime)$ the topological dual of $W^{m,p}(M,\,E)$, i.e.,
	\[
		W^{-m,q}(M,\,E^\prime) := \left( W^{m,p}(M,\,E) \right)^\prime	.
	\]
	 Here, $q := p'$ is the H\"{o}lder-conjugate exponent of $p$.
\end{definition}

\begin{remark}
To simplify notations, we will henceforth drop the subscript $\mathcal{A}$ when we denote the norm of a section $u \in W^{m,p}(M,E)$. More precisely, we will write $\norm{u}_{W^{m,p}(M,E)}$ to mean actually that we have fixed a regular bundle atlas $\mathcal{A}$ (and a subordinate partition of unity) and we are evaluating $\norm{u}_{W^{m,p}_\mathcal{A}(M,E)}$ accordingly to \eqref{eq:sobolev-norm-A}.
\end{remark}

\begin{remark}\label{rk:trivial-bundle}
	If $E \to M$ is a trivial $\K$-bundle of rank $\ell$, i.e., if $E = M \times \K^\ell$, we can identify $W^{m,p}(M,E)$ and $W^{m,p}(M,\K^\ell)$. Indeed, if $u : M \to M \times \K^\ell$ is a measurable section, then $\tilde{u} := {\rm pr_2} \circ \widehat{\chi} \circ u : M \to \K^\ell$ is a measurable function. Vice versa, if $\tilde{u} : M \to \K^\ell$ is a measurable function, the map $x \mapsto (x, \tilde{u}(x))$ gives rise to a section $u \in \Gamma(M,E)$. Moreover, the local representations of $u$ and $\tilde{u}$ coincide a.e. with respect to the Lebesgue measure. Hence, $u \in W^{m,p}(M,M\times \K^\ell)$ if and only if $\tilde{u} \in W^{m,p}(M,\K^\ell)$, and the norms are the same. Therefore, we can identify these two spaces, and we shall do so when convenient even without explicit mention.
\end{remark}

%-----Inclusa nella def di spazio di Sobolev
%\begin{definition}
%	Let $M$ be a smooth, compact, connected oriented, Riemannian manifold without boundary, of dimension $n \in \N$ and let $E \to M$ be a $\K$-vector bundle over $M$ of rank $\ell$. Given $m \geq 0$ an integer and $q \in [1,\infty]$, let $p$ be the H\"older conjugate exponent of $q$, i.e., $p = \frac{q}{q-1}$. We define the \emph{negative order Sobolev space} $W^{-m,p}(M,E)$ as the topological dual of $W^{m,q}(M,E)$, i.e.,
%	\[
%		W^{-m,p}(M,E) := \left( W^{m,q}(M,E) \right)^\prime.
%	\]
%\end{definition}

\paragraph{Sobolev spaces of differential forms} 
%We now specialise the previous setting to that of differential forms over $M$.
Let $0 \leq k \leq n$ be an integer and $E = \Lambda^k \T^* M$, the bundle of $k$-covectors over $M$. For $m \geq 0$ an integer and $p \in [1,\infty]$, let the space $W^{m,p}(M,\Lambda^k \T^*M)$ be defined according to Definition~\ref{def:sobolev}. Then, the spaces $W^{m,p}(M, \Lambda^k \T^* M)$ agree with the Sobolev spaces of differential $k$-forms over $M$ considered in \cite{Morrey, Scott, Budney, IwaniecScottStroffolini}. For the purpose of exposition, it is convenient to refer this definition of Sobolev spaces of differential $k$-forms over $M$ to as the \emph{classical} definition.

\begin{remark}
Recall that the bundles $\Lambda^k \T^*M \to M$ are constructed canonically starting from any atlas of $M$. In particular, any regular atlas $\{ (U_i, \varphi_i) \}_{i=1}^N$ for $M$ induces a regular bundle atlas $\mathcal{A}_k$ for $\Lambda^k \T^* M \to M$ for every $k \in \{0,1,\dots,n\}$. Thus, for convenience, in this paragraph we use the symbol $\mathcal{A}$ to denote the given atlas for $M$, i.e., $\mathcal{A} = \{(U_i,\varphi_i)\}_{i=1}^N$, although in the previous paragraph it has been used to denote bundle atlases.
%Consequently, slightly abusing the previously established notation, we can set $\mathcal{A} = \{(U_i,\varphi_i)\}_{i=1}^N$, and in this context specifying $\mathcal{A}$ and will be enough. %This explains why the constants in the statements of Proposition~\ref{prop:gaffney-closed} and Proposition~\ref{prop:higher-gaffney} below do not depend on $k$ (but they do depend on $n$ and $\mathcal{A}$).
\end{remark}
 
We set
\[
	(\omega, \eta) := \int_M \ip{\omega}{\eta}\, \vol_g := \int_M \omega \wedge *\eta,
\]
for any two measurable maps $\omega$, $\eta\colon M \to \Lambda^k \T^* M$ such that the right hand side is well-defined. We say that $\omega \in L^1(M, \Lambda^k \T^*M)$ has \emph{weak exterior differential} $\d \omega$ if there exists $\Omega \in L^1(M, \Lambda^{k+1} \T^*M)$ such that it holds
\[
	(\omega, \d^* \eta) = ( \Omega, \eta)
\]
for all smooth test forms $\eta \in C^\infty(M, \Lambda^{k+1} \T^*M)$. In such case, we set $\d \omega := \Omega$. Symmetrically, we say that $\omega \in L^1(M, \Lambda^k \T^*M)$ has \emph{weak exterior codifferential} $\d^* \omega$ if there exists $\Psi \in L^1(M, \Lambda^{k-1} \T^*M)$ such that the equation
\[
	(\omega,\d \eta) = (\Psi, \eta)
\]
holds for every $\eta \in C^\infty(M, \Lambda^{k-1} \T^*M)$, and we set $\d^*\omega := \Psi$. Clearly, when they exists, $\d \omega$ and $\d^* \omega$ are unique, and they coincide with the classical exterior differential and codifferential of $\omega$ if $\omega$ is smooth.

Thus, for $m \geq 1$, the operators
\[
	\begin{aligned}
		& \d : W^{m,p}(M, \Lambda^k \T^* M) \to W^{m-1,p}(M, \Lambda^{k+1} \T^*M), \\
		& \d^* : W^{m,p}(M, \Lambda^k \T^* M) \to W^{m-1,p}(M,\Lambda^{k-1} \T^* M),
	\end{aligned}
\]
are well-defined, linear and continuous. For $p \in (1,\infty)$, they are the unique extensions by linearity and density of the classical exterior differential and co-differential. Thus, by density, we have the following ``integration by parts'' formula: if $p, q \in (1,\infty)$ satisfy $1/p+1/q = 1$, then
\begin{equation}\label{eq:int-by-parts}
(\d \omega,\eta) \equiv \int_M \ip{\d \omega}{\eta} \, \vol_g \equiv \int_M \d \omega \wedge * \eta = \int_M \omega \wedge * \d^* \eta  \equiv (\omega, \d^* \eta)
\end{equation}
for all $\omega \in W^{1,p}(M, \Lambda^{k-1} \T^* M)$ and all $\eta \in W^{1,q}(M, \Lambda^k \T^* M)$. On the other hand, if $\omega \in W^{1,\infty}(M,\Lambda^k \T^*M)$, then $\omega$ is differentiable almost everywhere on $M$ and, since $M$ is compact, by H\"older inequality $\omega$ belongs to $W^{1,p}(M,\Lambda^k \T^*M)$ for every $p \in [1,\infty]$. Furthermore, for every form $\omega$ of class at least $W^{2,2}$, 
\[
	\d \d\omega = 0, \qquad \d^* \d^* \omega = 0 \quad \mbox{a.e. on } M.
\]
The explicit expressions of $\d$, $\d^*$ in coordinates will not be needed in this work. Formally, they are the same as in the classical case (see, e.g., \cite[Chapter~7]{Morrey}). They can be expressed in terms of the metric $g$ of $M$ on the Levi-Civita connection on $\T M$ (which classically induce canonical Riemannian metrics and linear connections over all tensor bundles of $M$), see e.g.~\cite[Chapter~1]{SchwarzG}.

\begin{remark}
	When defined, $\d$, $\d^*$ satisfy the same properties as in the classical case. For instance, they are local and commute with restrictions. %When applied to smooth enough forms, they also commute with pullbacks.
\end{remark} 

It has been firstly shown in \cite{Scott} that, for $M$ compact and without boundary, the classical definition of $W^{1,p}(M,\,\Lambda^k \T^*M)$ is equivalent to the following \emph{geometrical} one \cite[Proposition~4.11]{Scott}:
\[
	\mathcal{W}^{1,p}(M,\,\Lambda^k \T^*M) = \left\{ \omega \in L^p(M,\,\Lambda^k \T^* M) : \d \omega \in L^p(M,\,\Lambda^{k+1}M), \, \d^*\omega \in L^p(M,\,\Lambda^{k-1}M) \right\}\,.
\] 
The continuous embedding of $W^{1,p}(M,\,\Lambda^k \T^*M)$ into $\mathcal{W}^{1,p}(M, \,\Lambda^k \T^*M)$ is obvious, as the pointwise inequality (Eq.~(2.6) in~\cite{Scott})
\[
	\abs{\omega}^p + \abs{\d \omega}^p + \abs{\d^* \omega}^p \leq C(U,p) \abs{\nabla \omega}^p
\] 
holds a.e. within any open region $U \subset M$ compactly contained in a regular coordinate chart. Here, we are using the notation of \cite{Scott}, according to which $\abs{\nabla \omega}^p := \left( \sum \abs{ \pd{\omega_I}{x^k} }^2\right)^{\frac{p}{2}}$, $\omega = \sum \omega_I \d x^I$, $I = \{ 1 \leq i_1 < i_2 < \dots i_k \leq n \}$, and $\varphi \equiv (x^1, \dots, x^n)$ are the local (regular) coordinates on $U$.

The proof that $W^{1,p}(M, \Lambda^k \T^*M) = \mathcal{W}^{1,p}(M, \Lambda^k \T^*M)$ then amounts to prove the reverse continuous embedding. The latter  stems on the following $L^p$-version of \emph{Gaffney's inequality} (\cite[Proposition~4.10]{Scott}).

\begin{prop}[Gaffney's inequality]\label{prop:gaffney-closed}
	Let $M$ be a compact orientable smooth Riemannian manifold without boundary. Then a regular bundle atlas $\mathcal{A}$ of $M$ can be found so that the following happens: there exists a positive constant $C_p$, depending on $p$, $n$, and $\mathcal{A}$, so that
	\begin{equation}\label{eq:gaffney-closed}
		\| \omega \|_{W^{1,p}(M,\,\Lambda^k \T^*M)} \leq C_p ( \|\omega\|_{L^p} + \|\d \omega\|_{L^p} + \|\d^* \omega \|_{L^p})
	\end{equation}	 
	for all $\omega \in W^{1,p}(M,\,\Lambda^k \T^*M)$ and all integers $0 \leq k \leq n$.
\end{prop}

Gaffney's-type inequalities hold, for compact manifolds without boundary, also for higher order derivatives. More precisely, we have the following result.
\begin{prop}\label{prop:higher-gaffney}
	Let $M$ be a compact oriented smooth Riemannian manifold without boundary, endowed with the regular atlas $\mathcal{A}$ in Proposition~\ref{prop:gaffney-closed}. Then, for every $p \in (1,\infty)$ and every integer $m \geq 0$ there is a constant $C_{m,p}$, depending on $p$, $n$, $m$ and $\mathcal{A}$, so that
	\begin{equation}\label{eq:higher-gaffney}
		\norm{\omega}_{W^{m,p}(M, \Lambda^k \T^*M)} \leq C_{m,p} \sum_{s=0}^m \norm{(\d + \d^*)^s \omega}_{L^p(M)}
	\end{equation}
	for all $\omega \in W^{m,p}(M,\Lambda^k \T^*M)$ and all integers $0 \leq k \leq n$.
\end{prop} 

\begin{proof}[Sketch of the proof]
Take $\omega \in C^\infty(M,\Lambda^k \T^*M)$. For each open set $U_i$ of the covering provided by $\mathcal{A}$, inequality~\eqref{eq:higher-gaffney} is proved for the local representation of $\omega$ in \cite[Theorem~4.2.1]{Budney} (see also \cite[Proposition~4.2.2]{Budney}). Then the local inequalities can be glued together (using crucially the fact that $\mathcal{A}$ is regular) to give \eqref{eq:higher-gaffney} as in \cite[Proposition~4.10]{Scott}. This concludes the proof if $\omega$ is smooth. For general $\omega \in W^{m,p}(M,\Lambda^k \T^*M)$, the result follows by density.
\end{proof}

\begin{remark}\label{rk:geometric-W22}
In particular, the classical Sobolev spaces $W^{2,p}(M, \Lambda^k \T^*M)$ coincides, when $M$ is as in Proposition~\ref{prop:higher-gaffney}, with the space of measurable differential $k$-forms $\omega$ such that
\[
	\omega \in W^{1,p}(M,\Lambda^k \T^*M), \quad \d\omega \in W^{1,p}(M,\Lambda^{k+1} \T^*M), \quad \d^*\omega \in W^{1,p}(M,\Lambda^{k-1} \T^*M).
\]
%This will be important in Lemma~\ref{lemma:elliptic_reg_bis}.
\end{remark}

A $k$-form $\omega \in L^p(M, \Lambda^k \T^* M)$ such that $\d \omega$, $\d^* \omega$ exist in the weak sense and
\begin{equation}\label{eq:d-and-d*=0}
	\d \omega = 0 \quad \mbox{ and } \quad \d^* \omega = 0 
\end{equation} 
is called \emph{harmonic} (or a \emph{harmonic field}). Harmonic $k$-forms are automatically smooth \cite[Proposition~5.2]{Scott}, hence Eqs.~\eqref{eq:d-and-d*=0} actually hold in the classical sense. For each integer $k \geq 0$, we let
\begin{equation}\label{eq:Harm^k}
	\Harm^k(M) := \{ h \in C^\infty(M, \Lambda^k \T^*M) : \d h = \d^* h = 0  \}.
\end{equation}
Letting
\begin{equation}\label{eq:Delta}
	-\Delta := \d\d^* + \d^*\d 
\end{equation}
be the Laplace-Beltrami operator, we have the equivalent characterization (e.g., \cite[Proposition~5.2.4.3]{GiaquintaModicaSoucek-I})
\[
	\omega \in \Harm^k(M) \iff \Delta \omega = 0.
\] 
\begin{remark}
	Viewed as an operator from $W^{2,p}(M, \Lambda^k \T^* M)$ into $L^p(M, \Lambda^k \T^*M)$, $\Delta$ is a bounded (linear) operator for every $p \in (1,\infty)$ and any integer $k \geq 0$, with kernel $\Harm^k(M)$.
\end{remark}

%\begin{remark}
%	On sufficiently smooth forms, $\Delta$ commutes with $\d$, $\d^*$ and their composition. For instance, it is easy to see by approximation that if $p \in (1,\infty)$ and $\omega \in W^{3,p}(M, \Lambda^k \T^*M)$, then $\Delta(\d \omega) = \d(\Delta \omega)$ and $\Delta(\d^* \omega) = \d^*(\Delta \omega)$. 
%\end{remark}

For every integer $k \geq 0$, $\Harm^k(M)$ is a finite dimensional linear space (e.g., \cite[Remark~4.9]{IwaniecScottStroffolini}). Hence, any two norms are equivalent on $\Harm^k(M)$. We denote (cf.~\cite[(5.4)]{Scott})
\begin{equation}\label{eq:Harm^k-perp}
 (\Harm^k(M))^\perp := \left\{ \omega \in L^1(M, \Lambda^k \T^* M) \colon \forall h \in \Harm^k(M), \,\, (\omega,h) = 0 \right\}.
\end{equation}
As established in \cite[Lemma~5.6]{Scott}, for every $\omega \in L^1(M, \Lambda^k \T^*M)$ there is a unique $H(\omega) \in \Harm^k(M)$ such that
\[
	(\omega - H(\omega), h ) = 0 \quad \mbox{for all } h \in \Harm^k(M).
\]
We call $H(\omega)$ the \emph{harmonic part} of $\omega$. The assignment $\omega \mapsto H(\omega)$ defines a surjective linear operator (a projection, in fact) $H \colon L^1(M, \Lambda^k \T^*M) \to \Harm^k(M,\Lambda^k \T^*M)$, called \emph{harmonic projection}, which is bounded regardless of the norm on $\Harm^k(M)$ \cite[Proposition~5.9]{Scott}. %It follows from the above (cf.~\cite[Proposition~5.8]{Scott}) that, for every $p \in (1,\infty)$,
%\[
%	L^p(M, \Lambda^k \T^* M) = (I-H)(L^p) \otimes \Harm^k(M).
%\]

From the classical $L^2$-theory of differential forms as developed for instance in \cite[Chapter~6]{Warner}, there is a linear operator, called \emph{Green's operator},
\[
	G : C^\infty(M, \Lambda^k \T^*M) \to C^\infty(M, \Lambda^k \T^*M) \cap (\Harm^k(M))^\perp
\]
providing the unique solution in~$(\Harm^k(M))^\perp$ to the equation
\begin{equation}\label{eq:green-char}
	\Delta G(\omega) = \omega - H(\omega)
	\qquad \textrm{with } G(\omega)\in \Harm^k(M)^\perp
\end{equation}
%Following \cite{Warner,Scott}, for any $\omega \in L^p(M, \,\Lambda^k \T^*M)$, we define $G(\omega)$  as the unique solution to the equation $\Delta G(\omega) = \omega - H(\omega)$. 
With Gaffney's inequality~\eqref{eq:gaffney-closed} at hands, C.~Scott \cite{Scott} extended, for every $p \in (1,\infty)$, $G$ to a bounded linear operator, still called \emph{Green's operator} (and denoted with the same letter),
\begin{equation}\label{eq:green}
	G \colon L^p(M, \Lambda^k \T^*M) \to W^{2,p}(M, \Lambda^k \T^*M) \cap (\Harm^k(M))^\perp
\end{equation}
which still satisfies \eqref{eq:green-char} (\cite[Proposition~6.1]{Scott}).

\begin{remark}\label{rk:Delta-G}
For future reference, we notice that
\begin{equation}\label{eq:Delta-G}
	\Delta G(\omega) = G(\Delta \omega) = \omega - H(\omega)
\end{equation}
for every $\omega \in W^{2,p}(M, \Lambda^k \T^*M)$, $p \in (1,\infty)$, and $0 \leq k \leq n$ integer. For smooth forms, this is well-known (e.g., \cite[Proposition~6.10]{Warner}). For forms of class $W^{2,p}$, see \cite[Proposition~6.24]{Budney} or argue directly by density of $C^\infty(M,\Lambda^k \T^* M)$ in $L^p(M,\Lambda^k \T^* M)$ and the continuity of $G\colon L^p(M,\Lambda^k \T^* M) \to W^{2,p}(M,\Lambda^k \T^* M) \cap (\Harm^k)^\perp$ and $\Delta : W^{2,p}(M,\Lambda^k \T^* M) \to L^p(M,\Lambda^k \T^* M)$.
Equation~\eqref{eq:Delta-G} and the continuity of~$G$ imply
\begin{equation} \label{}
 \norm{\omega}_{W^{2,p}(M)} \leq C \left(\norm{\Delta\omega}_{L^p(M)}
  + \norm{H(\omega)}_{L^p(M)}\right)
\end{equation}
for any~$\omega\in W^{2,p}(M, \Lambda^k \T^*M)$
and some constant~$C$ that depends only on~$M$, $k$ and~$p$.
\end{remark}

Finally, the following fundamental result is proved in \cite{Scott}.
% he proved that for every $p \in (1,\infty)$ there is an operator %that for every $p \in (1,\infty)$ \emph{Green's operator} 
%\begin{equation}\label{eq:green}
%	G : L^p(M, \Lambda^k \T^*M) \to W^{2,p}(M,\Lambda^k \T^* M) \cap (\Harm^k(M))^\perp,
%\end{equation}
%which is linear and bounded, and satisfying~\eqref{eq:green-char}. From these facts, the next fundamental result follows.
\begin{prop}[{$L^p$-Hodge decomposition, \cite[Proposition~6.5]{Scott}}]\label{prop:LpHodgeDec}
	Let $M$ be a smooth compact oriented Riemannian manifold without boundary and $1 < p < \infty$. For any integer $0 \leq k \leq n$, we have
	\begin{equation}\label{eq:LpHodgeDec}
	\begin{split}
		L^p(M, \Lambda^k \T^* M) &= \Delta G(L^p) \oplus \Harm^k(M) \\
		&= \d \d^* G(L^p) \oplus \d^* \d G(L^p) \oplus \Harm^k(M).
	\end{split}
	\end{equation}
	Moreover,  $\d(W^{1,p}(M, \,\Lambda^{k-1}\T^*M)) = \d\d^* G(L^p)$ and $\d^*(W^{1,p}(M,\,\Lambda^{k+1}\T^*M)) = \d^* \d G(L^p)$. 
	
	Consequently, any $\omega \in L^p(M, \Lambda^k \T^* M)$ can be uniquely written as 
	\begin{equation}\label{eq:hodge-dec}
		\omega = \d \varphi + \d^* \psi + \xi,
	\end{equation}
	where $\varphi \in W^{1,p}(M, \Lambda^{k-1}\T^*M)$ is co-exact, $\psi \in W^{1,p}(M, \Lambda^{k+1}\T^*M)$ is exact and $\xi \in \Harm^k(M)$ is a harmonic $k$-form. In addition, there exists a constant $C > 0$, depending only on $p$, $k$ and $M$, such that there holds
	\begin{equation}\label{eq:norm-est-hodge-dec}
		\norm{\varphi}_{W^{1,p}(M)} + \norm{\psi}_{W^{1,p}(M)} + \norm{\xi}_{L^\infty(M)} \leq C \norm{\omega}_{L^p(M)}
	\end{equation}
	for every $\omega \in L^p(M,\,\Lambda^k\T^*M)$, where $\varphi$, $\psi$ and $\xi$ are as in \eqref{eq:hodge-dec}.
\end{prop}

\begin{remark}
	The choice of using the $L^\infty$-norm for the harmonic part of $\omega$ in estimate \eqref{eq:norm-est-hodge-dec} is somewhat arbitrary. However, since $\Harm^k(M)$ has finite dimension, the $L^\infty$-norm can be replaced by any other norm (up to enlarging $C$, if necessary).
\end{remark}

For later reference, we point out the following elliptic regularity lemma, immediate consequence of Proposition~\ref{prop:LpHodgeDec} and the Open Mapping Theorem.
 
\begin{lemma}\label{lemma:classic-elliptic-reg}
	Let~$j \geq 0$ and integer and $p\in (1, \infty)$. For every~$f\in W^{j,p}(M, \, \Lambda^k\T^*M)$ such that~$(f,\xi) = 0$ for any~$\xi\in\Harm^k(M)$, there exists a unique~$v\in W^{j+2,p}(M, \, \Lambda^k\T^*M)$ such that
 \begin{equation} \label{elliptic_eq_3}
  \begin{cases}
   -\Delta v  = f &\textrm{in the sense of } \mathcal{D}^\prime(M) \\
   \displaystyle\int_M \ip{v}{\xi} \vol_g = 0 
    &\textrm{for any } \xi\in\Harm^k(M)
  \end{cases}
 \end{equation}
 Moreover, $v$ satisfies
 \begin{equation}\label{eq:elliptic-est-poisson-2}
  \norm{v}_{W^{j+2,p}(M)} \leq C_{p,j} \norm{f}_{W^{j,p}(M)}
 \end{equation}
 for some constant~$C_{p,j}$ depending only on~$M$, $j$, $k$, $p$.
\end{lemma} 

\begin{proof}
	Existence and uniqueness readily follow from Proposition~\ref{prop:LpHodgeDec}. Thus, for any $j \geq 0$, 
	\[
		\Delta \colon W^{j+2,p}(M, \Lambda^k \T^* M) \cap (\Harm^k(M))^\perp \to W^{j,p}(M, \Lambda^k \T^*M) \cap (\Harm^k(M))^\perp
	\] 
	is a continuous bijection, hence a Banach space isomorphism by the Open Mapping Theorem. Estimate~\eqref{eq:elliptic-est-poisson-2} is an immediate consequence of this latter fact.
\end{proof}

%\begin{proof}
%	We only point out that the estimate is not explicitly written there but it follows directly from the fact that $\Delta : W^{j+2,p}(M, \Lambda^k \T^* M) \cap (\Harm^k(M))^\perp \to W^{j,p}(M, \Lambda^k \T^*M) \cap (\Harm^k(M))^\perp$ is a Banach space isomorphism. This latter fact is a straightforward consequence of the Open Mapping Theorem, since under our assumptions $\Delta$ is a one-to-one continuous linear operator.
%\end{proof}

\begin{remark}\label{rk:est-exact-forms}
	Clearly, if $\psi \in W^{2,p}(M,\,\Lambda^k \T^*M)$ is (co)exact, $\Delta \psi$ is (co)exact as well. Then the estimate ~\eqref{eq:elliptic-est-poisson-2}, the fact that $-\Delta \psi = \d \d^*\psi$, and Remark~\ref{rk:geometric-W22} imply that there is a constant $C > 0$, depending only on $k$, $p$ and $M$, such that the estimate
\begin{equation}\label{eq:est-exact-forms}
	\norm{\psi}_{W^{2,p}(M)} \leq C \norm{\d^* \psi}_{W^{1,p}(M)}
\end{equation}
holds for all exact $k$-forms $\psi \in W^{2,p}(M,\,\Lambda^k \T^*M)$. (If $\psi$ is co-exact, $\d^* \psi$ must be replaced by $\d \psi$ in the right hand side of~\eqref{eq:est-exact-forms}.) This fact is used in Section~\ref{sec:maingoalG}.
\end{remark}

To conclude this section, we notice that neither Gaffney's inequality nor the $L^p$-Hodge decomposition hold for $p = 1$, as shown in \cite{BaldoOrlandi-Hodge}. However, in the same paper it also proven that Green's operator exists even in this case as a map from measure $k$-forms into $W^{1,p}(M, \Lambda^k \T^*M)$ for every $1 < p < \frac{n}{n-1}$. Importantly, measure $k$-forms can be regarded as Radon vector measures on $M$ with values in $k$-forms, as \cite[Proposition~2.2]{BaldoOrlandi-Hodge} shows. In the language of the present paper, measure $k$-forms are simply $k$-currents with finite mass (cf.~\cite[Definition~2.1]{BaldoOrlandi-Hodge} and Appendix~\ref{sect:currents}).

%==================================================
% 2022/18/03: Commento perché non mi pare che serva
%==================================================
%{\BBB Sposto qui la variante per variet\`a con bordo. Tutto andr\`a poi legato con frasi di raccordo, la notazione andr\`a uniformata, e anche questi spazi di 
%forme con componente normale e tangenziale nulla andrebbero forse spigati meglio\ldots}
%
%\begin{prop}[{Eq.~(1.1) in \cite{IwaniecScottStroffolini}}]\label{prop:gaffney}
%	Let $\Omega$ be an open regular region of a closed $C^\infty$-smooth oriented Riemannian manifold $\mathcal{N}$ of dimension $n$. Denote $W^{1,p}_T(\bigwedge^l \Omega)$, $W^{1,p}_N(\bigwedge^l \Omega)$ the closures in the Sobolev norm of the spaces of smooth $l$-forms on $\Omega$, smooth $l$-forms on $\Omega$ with vanishing tangential and, respectively, vanishing normal component along $\partial M$. Let $p \in (1,\infty)$. Then there exists a constant $C_p(\Omega) > 0$, depending only on $\Omega$ and $p$, such that for each $\omega \in W^{1,p}_T(\bigwedge^l \Omega) \cup W^{1,p}_N(\bigwedge^l \Omega)$,  
%	\begin{equation}\label{eq:gaffney}
%		\|\omega \|_{W^{1,p}(\Omega)} \leq C_p(\Omega) (\| \d \omega \|_{L^p(\Omega)} + \| \d^* \omega\|_{L^p(\Omega)} + \|\omega \|_{L^p(\Omega)})\,.
%	\end{equation}
%\end{prop}
%=================================================

\subsection{Hermitian line bundles, connections, and weak covariant derivatives of Sobolev sections}

A \emph{Hermitian metric} on a complex line bundle $E \to M$ is an assignment, for each $x \in M$, of a positive definite Hermitian form $h_x : E_x \times E_x \to \C$ that is smooth in the sense that, for all sections $u_1$, $u_2 \in C^\infty(M,\,E)$, the function $x \mapsto h_x(u_1(x), u_2(x))$ is smooth. In this case we say that $E \to M$ is a \emph{Hermitian line bundle}. The typical fibre of $E \to M$ is of course $\C$. The structure group of a Hermitian line bundle $E \to M$ automatically reduces to $\U(1)$ \cite[pp.~280-1]{Lee}. Naturally, the metric allows to identify $E'$ and $E$. Associated with a Hermitian metric, there is a canonical scalar product, i.e., its real part, that we denote $\ip{\cdot}{\cdot}$. In other words, we set 
\[
	\ip{\cdot}{\cdot} := \frac{1}{2}(h + \bar{h}).
\]

%, which we consider endowed with the standard scalar product $a \cdot b := \frac{1}{2}(a\bar{b} + \bar{a}b)$ induced by the standard Hermitian form $h_0(a,b) := a\bar{b}$, where $a$, $b \in \C$. Then, upon choosing any non zero vector of $\C$ as standard basis, the structure group of $E \to M$ reduces to $\U(1)$ \cite[pp.~280-1]{Lee}. When there is no danger of confusion, we use the symbol $\ip{\cdot}{\cdot}$ to denote the Hermitian metric of $E \to M$, omitting to indicate the point $x$. Naturally, the metric allows to identify $E'$ and $E$.

A (smooth) \emph{connection} $\D$ on a vector bundle $E \to M$ is a linear map
\[
	\D \colon C^\infty(M, \,E) \to C^\infty(M, \, \T^*M \otimes E)
\]
satisfying Leibniz' rule:
\begin{equation}\label{eq:leibniz}
	\forall f \in C^\infty(M), \,\,\forall u \in C^\infty(M,\,E), \quad \D(fu) = \d f \otimes u + f \D u.
\end{equation}
For every fixed $u \in C^\infty(M,E)$, we can view $\D u$ a map taking a vector field on $M$ as argument and giving back a section $\D u(X)$ of $E \to M$. We set $\D_X u := \D u(X)$ and call $\D_X u$ the \emph{weak covariant derivative of $u$ with respect to $X$}.

A \emph{metric connection} on a Hermitian line bundle $E \to M$ is a connection $\D$ that is compatible with the metric, i.e., satisfying $\D h \equiv 0$. This implies
\begin{equation}\label{eq:metricity}
	\forall u,\, v \in C^\infty(M,E), \quad \d \ip{u}{v} = \ip{\D u}{v} + \ip{u}{\D v}.
\end{equation}
Explicitly,~\eqref{eq:metricity} means that, for every pair of sections $u$, $v \in C^\infty(M,\,E)$ and every smooth vector field $X \in C^\infty(M,\,\T M)$, there holds
\[
	X\ip{u}{v} = \ip{\D_X u}{v} + \ip{u}{\D_X v}.
\]
We recall the following important facts:%, that we will use in the sequel:
\begin{itemize}
	\item For every $u \in C^\infty(M,\,E)$ and every $X \in C^\infty(M, \, \T M)$, the value $(\D_X u)(x)$ of $\D_X u$ at each $x \in M$ depends only on $X(x)$ and the values of $u$ along any smooth curve representing $X(x)$ \cite[p.~502]{Lee}. In fact, $\D_A$ is a local operator and behaves naturally with respect to restrictions \cite[Section~12.1]{Lee}.
	\item Let $U \subset M$ be an open set so that $E \to M$ is trivial over $U$, $\chi_U : E\vert_U \to U \times \C$ a corresponding local trivialisation, and $e_U$ a reference section for $E \to M$ over $U$. Then every $u \in C^\infty(M, \, E)$ writes as $u = \widetilde{u} e_U$ for some smooth complex-valued function $\widetilde{u}$ and we have, with respect to the local trivialisation $\chi_U$,
	\[
		\D u = \left( \d \widetilde{u} + A_U \widetilde{u} \right) \otimes e_U \qquad \mbox{in } U,
	\]
	where $A_U \in C^\infty(U, \, \T^* U)$ is a complex-valued 1-form, called the \emph{connection 1-form} of $\D$ over $E \vert_U$ with respect to $\chi_U$. If $\widetilde{\chi}_U : E \vert_U \to U \times \C$ is another local trivialisation for $E \to M$ over $U$, set $g_U := \widetilde{\chi}_U \circ \chi_U^{-1}$ (notice that $g_U \in C^\infty(U \times \C, \, U \times \C)$). Denoting $\widetilde{A}_U$ the connection 1-form of $\D$ with respect to $\widetilde{\chi}_U$, then the transformation law
	\begin{equation}\label{eq:transf-law}
		A_U = \widetilde{A}_U + g_U^{-1} \d g_U
	\end{equation}
	holds. From \eqref{eq:transf-law} it is readily seen that $\d A_U$ does \emph{not} depend on the local trivialisation.
	\item It is a general fact that connections on a vector bundle $E \to M$ form an affine space modelled over $C^\infty(M, \T^*M \otimes {\rm End}(E))$, where ${\rm End}(E)$ is the bundle of endomorphisms of $E$. This means that, upon choosing a reference connection $\D_0$, any other connection on $E \to M$ writes as 
	\begin{equation}\label{eq:def-D_A-0}
		\D_A := \D_0 + A,
	\end{equation}
	for some $A \in C^\infty(M, \T^*M \otimes {\rm End}(E))$. In the case of a Hermitian line bundle with a reference metric connection $\D_0$, $\D_A$ is still a (smooth) metric connection if and only if $A$ belongs to the smaller space $C^\infty(M, \T^*M \otimes {\rm Ad}(E))$. Here, ${\rm Ad}(E)$ denotes the bundle of endomorphisms of $E$ which are skew Hermitian on each fiber. Thus, the typical fiber of ${\rm Ad}(E)$ is the Lie algebra of the structure group of $E$. In our case, ${\rm Ad}(E)$ is a trivial bundle with typical fiber $\U(1)$. As the latter can be identified with $i \R$,  $A$ should take purely imaginary values in any local trivialisation. In addition, we have
	\[
		C^\infty(M,\, \T^*M \otimes {\rm Ad}(E)) \simeq C^\infty(M, \, \T^* M),
	\]
	with canonical isomorphism. Thus, we can identify $A$ with a 1-form with purely imaginary coefficients. However, it is customary to assume instead that $A$ is real valued, writing $-iA$ in place of $A$. We then rewrite \eqref{eq:def-D_A-0} as
	\begin{equation}\label{eq:def-D_A}
		\D_A := \D_0 - i A.
	\end{equation}
	Explicitly, \eqref{eq:def-D_A} means that, for every $u \in C^\infty(M,\,E)$ and every smooth vector field $X$ on $M$, we have 
	\[
		\D_{A,\,X} u = \D_{0,\,X}u -i A(X) u.
	\] 
\end{itemize}
The \emph{curvature} $\D^2_A$ of a connection $\D_A$ is given by the following formula: for all $u \in C^\infty(M,\,E)$ and all $X, Y \in C^\infty(M,\,\T M)$,
\begin{equation}\label{eq:def-curvature}
	\D_A^2 u(X,Y) := \D_{A,\,X} \D_{A,\,Y} u - \D_{A,\,Y} \D_{A,\,X} u - \D_{A,\,[X,Y]} u.
\end{equation}
One easily checks \cite[Section~12.5]{Lee} that there exists a \emph{closed} ${\rm End}(E)$-valued 2-form $F_A$, called the \emph{curvature form} of $\D_A$, such that
\[
	\forall u \in C^\infty(M,\,E),\,\, \forall X, Y \in C^\infty(M,\,\T M), \quad \D^2_A u(X,Y) = F_A(X,Y) u.
\]
As for $A$, if $\D_A$ is a metric connection on a Hermitian line bundle, $F_A$ is an ${\rm Ad}(E)$-valued 2-form taking purely imaginary values in any local trivialisation. Thus, $F_A$ is identified with a 2-form on $M$ which is  assumed to be real-valued, replacing $F_A$ with $-i F_A$ in the above formula. Then, denoting $F_0$ the curvature form of the reference connection, it holds
\begin{equation}\label{eq:curvature}
	F_A = F_0 + \d A.
\end{equation}

So far, we have dealt with smooth sections and smooth connections. We now extend the previous discussion to Sobolev sections and connections. To this end, we have to define the concept of weak covariant derivative of a non-smooth section $u : M \to E$. For the moment being, we still assume $A$ is a smooth 1-form on $M$. %This goes along the lines of the classical definition of weak derivative of a function from an open set of $\R^n$ into $\R^m$.

The first ingredient we need is the extension, for every integer $0 \leq k \leq n$, of $\D_A$ to an operator from $C^\infty(M, \, \Lambda^k \T^*M \otimes E)$ into $C^\infty(M, \, \Lambda^{k+1} \T^*M \otimes E)$. This is standardly done by introducing the \emph{exterior covariant derivative} induced by $\D_A$, which we denote $\d^A$. For the definition of $\d^A$, we address the reader to \cite[Section~12.9]{Lee}. The properties of $\d^A$ are formally similar to those of $\D_A$ and they are summarised in \cite[Theorem~12.57]{Lee}. Here we only stress that, obviously, $\d^A$ coincides with $\D_A$ on $C^\infty(M, \, E)$, i.e., if $k=0$.  

Next, we extend~$*$ to an operator from $\Gamma(M, \,\Lambda^k \T^*M \otimes E)$ to $\Gamma(M,\,\Lambda^{n-k} \T^*M \otimes E)$, which we still denote~$*$. To this purpose, it is enough to define the action of $*$ on simple elements of $\Gamma(M,\,\Lambda^k \T^*M \otimes E)$ by letting
\begin{equation}\label{eq:star-dec}
	*( \omega \otimes u) := (* \omega) \otimes u
\end{equation}
for $u \in \Gamma(M,\,E)$ and $\omega \in \Gamma(M,\,\Lambda^k \T^*M)$. The rule \eqref{eq:star-dec}, extended linearly, gives a meaning to $* \sigma$ for every $\sigma \in \Gamma(M,\,\Lambda^k \T^*M \otimes E)$ and $0 \leq k \leq n$ integer. Using the scalar product associated with the metric of $E$ and the scalar product of $k$-forms %~\eqref{Hodge}
induced by the metric, we define 
\begin{equation}\label{eq:scalar-prod-dec}
	\langle\langle \omega_1(x) \otimes u_1(x), \omega_2(x) \otimes u_2(x) \rangle\rangle \vol_g := \ip{u_1(x)}{u_2(x)} (\omega_1(x) \wedge * \omega_2(x)),
\end{equation}
for every $u_1$, $u_2 \colon M \to E$ measurable sections and measurable $k$-forms $\omega_1$, $\omega_2$, and a.e. $x \in M$. From \eqref{eq:scalar-prod-dec} we define a corresponding $L^2$-product
\begin{equation}\label{eq:L2-prod-dec}
	(( \omega_1 \otimes u_1, \omega_2 \otimes u_2)) := \int_M \langle\langle \omega_1 \otimes u_1, \omega_2 \otimes u_2 \rangle\rangle \,\vol_g
\end{equation}
anytime the right hand side of \eqref{eq:L2-prod-dec} exists. Once again, extending (bi)linearly the rule \eqref{eq:L2-prod-dec}, we can define the $L^2$-product of arbitrary $\sigma_1$, $\sigma_2 \in \Gamma(M,\,\Lambda^k \T^*M \otimes E)$ by  letting
\begin{equation}\label{eq:def-L2-prod}
	(( \sigma_1, \sigma_2)) :=  \int_M \langle\langle \sigma_1, \sigma_2 \rangle\rangle\,\vol_g
\end{equation}
anytime the integral at right hand side exists.

With all this machinery at disposal, we can define the \emph{formal adjoint} of $\d^A$ as the operator 
\[
	(\d^A)^* : C^\infty(M, \, \Lambda^k \T^*M \otimes E) \to C^\infty(M, \,\Lambda^{k-1}\T^*M \otimes E)
\]
which is formally adjoint to $\D_A$ with respect to the $L^2$-product \eqref{eq:def-L2-prod}. An explicit computation yields (cf., e.g.,~\cite[Section~4.2]{Jost})
\[
	(\d^A)^* := (-1)^{n(k+1)+1} * \d^A * = (-1)^{n(k+1)+1} * (\d^0 - iA) *,
\]
where $\d^0$ denotes the exterior covariant derivative induced by the reference connection $\D_0$. For $k = 1$, we set $\D_A^* := (\d^A)^*$.

%We can finally define the concept of weak covariant derivative of a section $u \in W^{1,2}(M,\,E)$. We say that $u \in W^{1,2}(M,\,E)$ has \emph{weak covariant derivative} $\D_A u$ if there exists a section $\sigma_u \in L^2(M,\,\T^*M \otimes E)$ such that
We can finally define the concept of weak covariant derivative of a section $u \in L^1(M,\,E)$. We say that $u \in L^1(M,\,E)$ has \emph{weak covariant derivative} $\D_A u$ if there exists a section $\sigma_u \in L^1(M,\,\T^*M \otimes E)$ such that
\begin{equation}\label{eq:def-weak-cov-der}
	((\sigma_u, \tau)) = ((u, \D_A^* \tau))
\end{equation} 
for every $\tau \in C^\infty(M,\,\T^*M \otimes E)$. In such case, $\sigma_u$ is uniquely determined, and we set $\D_A u := \sigma_u$. In particular, if $u \in W^{1,2}(M,\,E)$, then \eqref{eq:def-weak-cov-der} defines a linear bounded operator
\[
	\D_A \colon W^{1,2}(M, \, E) \to L^2(M, \,\T^*M \otimes E).
\]
%so that $\D_A u := \sigma_u$. 

We are now in position to weaken the requirement $A \in C^\infty(M, \, \T^*M)$ in force so far. Indeed, if $u \in W^{1,2}(M, \, E)$ and $\D_0$ is a smooth reference connection, then $\D_0 u$ is well-defined through \eqref{eq:def-weak-cov-der} and belongs to $L^2(M,\,\T^*M \otimes E)$ (and in turn to $L^1(M,\,\T^*M \otimes E)$, by the compactness of $M$). For $A \in W^{1,2}(M, \,\T^*M)$, we have $Au \in L^1(M, \T^*M)$, hence we can define 
\[
	\D_A u := \D_0 u - i A u.
\]
Clearly, $\D_A u$ belongs to $L^1(M, \,\T^*M \otimes E)$. Moreover, if $u \in (L^\infty \cap W^{1,2})(M, \, E)$, then $\D_A u \in L^2(M, \,\T^*M \otimes E)$. Notice that, even in this context, $\D_A$ is a local operator and behaves naturally with respect to restrictions. 

For $A \in W^{1,2}(M,\,\T^*M)$, we \emph{define} the curvature 2-form of $\D_A$ by Equation~\eqref{eq:curvature}, i.e., we define $F_A := F_0 + \d A$. In such a way, $F_A \in L^2(M,\,\Lambda^2 \T^*M)$ for any $A \in W^{1,2}(M,\,\T^*M)$.

\begin{remark}
	Although of no use in this work, we record the following fact which will be needed in the forthcoming work \cite{CanevariDipasqualeOrlandiII}. %The adjoint of $\D_A \colon W^{1,2}(M, \, E) \to L^2(M, \,\T^*M \otimes E)$ is the operator (cf.,~e.g., \cite[Chapter~4]{Wehrheim})
%\[
%	\D_A^* \colon L^2(M, \,\T^*M \otimes E) \to W^{-1,2}(M, \,E)
%\]
%defined as follows: given $\sigma \in L^2(M, \T^*M \otimes E)$ and $u \in W^{1,2}(M,E)$, we set
We can associate with $\D_A$ its formal adjoint $\D_A^*$, i.e., 
its adjoint with respect to the $L^2$-product $((\cdot,\cdot))$, as follows. 
Given $A$ a 1-form on $M$ (not necessarily of class $W^{1,2}(M)$), 
$\sigma \in L^2(M,\,\T^*M \otimes E)$, and any section 
$u \in \Gamma(M,\,E)$ such that $\D_A u \in L^2(M,\,E)$, we let
\begin{equation}\label{eq:DA*}
	(\D_A^* \sigma)(u) := \int_M \langle\langle \sigma, \D_A u \rangle \rangle\,\vol_g.
\end{equation}
For every fixed $\sigma \in L^2(M,\,\T^*M \otimes E)$, the right hand side 
of~\eqref{eq:DA*} defines a linear 
form $u \mapsto (\D_A^*\sigma)(u)$ on the linear space 
\[
	W_A := \left\{ u \in L^2(M,\,E) : \D_A u \in L^2(M,\,\T^*M \otimes E) \right\}.
\]
By Schwarz inequality, such a linear form is continuous if $W_A$ is endowed with the norm 
$\norm{u}_{W_A} := \norm{u}_{L^2(M)} + \norm{\D_A u}_{L^2(M)}$. 
Furthermore, if $A \in L^2(M,\,\T^*M)$, then $W_A$ embeds continuously 
into $W^{1,1}(M,\,E)$. Moreover, if $A$ is a smooth 1-form, 
then $W_A = W^{1,2}(M,\,E)$, 
and~\eqref{eq:DA*} defines a continuous linear operator
$\D_A^* \colon L^2(M, \,\T^*M \otimes E) \to W^{-1,2}(M, \,E)$, 
cf.,~e.g., \cite[Chapter~4]{Wehrheim}.
%Then, $\D_A^* \sigma$ is plainly a bounded linear functional over $W^{1,2}(M, E)$.
\end{remark}

\begin{remark}
	The pointwise scalar product $\langle\langle \cdot, \cdot \rangle\rangle$ and the $L^2$-product $((\cdot, \cdot))$ will be simply denoted $\ip{\cdot}{\cdot}$ and $(\cdot,\cdot)$ respectively anytime no ambiguity arises.
\end{remark}

\section{Currents and form-valued distributions}
\label{sect:currents}

\paragraph{Currents.}

We recall some standard terminology for currents
on smooth manifolds, and refer 
e.g.~to~\cite{Federer, Simon-GMT} for more details.
We equip~$C^\infty(M, \, \Lambda^k\T^*M)$ with the topology
induced by the family of~$C^h$-seminorms, for all integers~$h\geq 1$.
The space of~$k$-currents~$\mathcal{D}^\prime(M, \, \Lambda_k\T M)$
is defined as the topological dual of smooth~$k$-forms,
$\mathcal{D}^\prime(M, \, \Lambda_k\T M) := 
(C^\infty(M, \, \Lambda^k\T^*M))^\prime$.
Currents come with a boundary operator, 
$\partial\colon\mathcal{D}^\prime(M, \, \Lambda_k\T M)
\to\mathcal{D}^\prime(M, \, \Lambda_{k-1}\T M)$, defined as
\begin{equation} \label{boundary}
 \ip{\partial S}{\omega}_{\mathcal{D}^\prime,\mathcal{D}}
 := \ip{S}{\d\omega}_{\mathcal{D}^\prime,\mathcal{D}}
\end{equation}
for any $k$-current~$S$ and any smooth $(k-1)$-form~$\omega$. 
The boundary operator is sequentially continuous 
with respect to the weak$^*$ convergence of distributions.

Given a point~$x\in M$ and a
$k$-covector~$\omega\in\Lambda^k \T_x M^*$,
we define the comass of~$\omega$ as
\begin{equation} \label{comass}
 \norm{\omega} := \sup\left\{\ip{\omega}{v}\colon v \in \Lambda_kT_x M
 \textrm{ is a \emph{simple }} k\textrm{-vector such that} \abs{v} = 1\right\}
\end{equation}
(here, $\abs{\cdot}$ is the norm on~$\Lambda_k\T_x M$ induced
by the Riemannian metric on~$M$). The comass is a norm on~$\Lambda_k\T^* M$,
which does \emph{not} coincide with the norm~$\abs{\cdot}$
induced by the Riemannian metric. However, there exists 
a constant~$C_{n,k}$, depending on~$n$ and~$k$ only, such that
\begin{equation} \label{comass-Riemannian}
 \norm{\omega} \leq \abs{\omega} \leq C_{n,k} \norm{\omega}
\end{equation}
for any~$\omega\in \Lambda_{k}\T_x^* M$ and any~$x\in M$.
The equality~$\norm{\omega} = \abs{\omega}$ holds 
if and only if~$\omega$ is simple.
For any $k$-current~$S$, the mass of~$S$ is defined as
\begin{equation} \label{mass}
 \M(S) := \sup\left\{\ip{S}{\omega}_{\mathcal{D}^\prime, \mathcal{D}}
 \colon \omega\in C^\infty(M, \, \Lambda^k\T^*M), \
 \sup_{x\in M} \norm{\omega(x)} \leq 1 \right\}
\end{equation}
Let~$E\subset M$ be a~$k$-rectifiable set, oriented
by a measurable, simple $k$-vector field~$v_E\colon E\to\Lambda_k\T M$ 
that is tangent to~$E$ and satisfies~$\abs{v_E} = 1$
at $\H^k$-almost any point of~$E$.
Let~$\theta\colon E\to\Z$ be a measurable function.
We denote by~$\llbracket E, \, \theta\rrbracket$ the current
carried by~$E$ with multiplicity~$\theta$, defined by
\begin{equation} \label{imr_current}
 \ip{\llbracket E, \, \theta\rrbracket}{\tau}_{\mathcal{D}^\prime, \mathcal{D}}
 := \int_E \theta \ip{\tau}{\nu_E} \, \d\H^k
\end{equation}
for any smooth~$k$-form~$\tau$. In case~$\theta = 1$ identically,
We write~$\llbracket E \rrbracket := \llbracket E, \, 1\rrbracket$.
A current that can be written in the form~\eqref{imr_current}
is called an integer-multiplicity rectifiable~$k$-current. The mass
of an integer-multiplicity rectifiable current is given by
\begin{equation} \label{imr_mass}
 \M(\llbracket E, \, \theta\rrbracket) = \int_{E} \abs{\theta} \, \d\H^k
\end{equation}
We denote by~$\mathcal{R}_k(M)$ the set of integer-multiplicity
rectifiable~$k$-currents with finite mass. 

Finally, we define the flat norm of a
current~$S\in\mathcal{D}^\prime(M, \, \Lambda_k\T M)$ as
\begin{equation} \label{flat}
 \F(S) := \inf\left\{\M(P) + \M(Q)\colon
 P\in\mathcal{R}_{k+1}(M), \ Q\in\mathcal{R}_k(M), 
 \ S = \partial P + Q\right\}
\end{equation}
(with the understanding that~$\inf\emptyset = +\infty$).

\paragraph{Form-valued distributions.}

We equip the space of smooth~$k$-vector fields,
$C^\infty(M, \, \Lambda_k\T M)$, with the topology induced 
by the family of~$C^h$-seminorms, for all integers~$h\geq 1$.
The space of distributions with values in~$k$-forms is
defined as the topological dual of smooth~$k$-vectors,
$\mathcal{D}^\prime(M, \, \Lambda^k\T^*M) := 
(C^\infty(M, \, \Lambda_k\T M))^\prime$.
The differential and the codifferential extends to operators
on form-valued distributions,
% $\d\colon\mathcal{D}^\prime (M, \, \Lambda^k\T^* M)
% \to\mathcal{D}^\prime(M, \, \Lambda^{k+1}\T^* M)$
% and $\d^*\colon\mathcal{D}^\prime(M, \, \Lambda^k\T^* M)
% \to\mathcal{D}^\prime(M, \, \Lambda^{k-1}\T^* M)$
by duality:
\begin{align}
 \ip{\d\omega}{v}_{\mathcal{D}^\prime,\mathcal{D}}
  &:= \ip{\omega}{(\d^* v^\flat)^\#}_{\mathcal{D}^\prime,\mathcal{D}} \label{d} \\
 \ip{\d^*\omega}{v}_{\mathcal{D}^\prime,\mathcal{D}}
  &:= \ip{\omega}{(\d w^\flat)^\#}_{\mathcal{D}^\prime,\mathcal{D}} \label{d*}
\end{align}
for any~$\omega\in\mathcal{D}^\prime(M, \, \Lambda^k\T^* M)$,
$v\in C^\infty(M, \, \Lambda_{k+1}\T M)$ 
and~$w\in C^\infty(M, \, \Lambda_{k-1}\T M)$. 
The definitions~\eqref{d}, \eqref{d*} are 
consistent with~\eqref{eq:int-by-parts}.

\paragraph{The Hodge dual of currents and forms.}
\label{sect:Hodge}

There are natural operators
\begin{align*}
 \star\colon\mathcal{D}^\prime(M, \, \Lambda^k\T^*M)\to
    \mathcal{D}^\prime(M, \, \Lambda_{n-k}\T M) \\
 \star\colon\mathcal{D}^\prime(M, \, \Lambda_k\T M)\to
    \mathcal{D}^\prime(M, \, \Lambda^{n-k}\T^* M)
\end{align*}
defined as follows: for any~$\omega\in\mathcal{D}^\prime(M, \, \Lambda^k\T^*M)$
the current~$\star\omega$ is defined as
\begin{equation} \label{starsharp}
  \ip{\star\omega}{\tau}_{\mathcal{D}^\prime, \mathcal{D}} 
  := (-1)^{k(n-k)} \, 
  \ip{\omega}{(*\tau)^\#}_{\mathcal{D}^\prime, \mathcal{D}}
\end{equation}
for any~$\tau\in C^\infty(M, \, \Lambda^{n-k}\T^*M)$.
Similarly, given~$S\in\mathcal{D}^\prime(M, \, \Lambda_k\T M)$,
we define the form-valued distribution
$\star S = * S^\flat\in\mathcal{D}^\prime(M, \, \Lambda^{n-k}\T^*M)$ as
\begin{equation} \label{starflat}
  \ip{\star S}{v}_{\mathcal{D}^\prime,\mathcal{D}} 
  := (-1)^{k(n-k)} \, \ip{S}{* v^\flat}_{\mathcal{D}^\prime, \mathcal{D}}
\end{equation}
In case~$\omega$, $S$ are smooth (or, more generally, they are 
represented by $L^1$-vector or covector fields),
the definitions~\eqref{starsharp} and~\eqref{starflat} are
consistent with~\eqref{musicalstar}, i.e.
\[
 \star \omega = (*\omega)^\#, \qquad \star S = *(S^\flat)
\]
Indeed, for any smooth~$k$-form~$\omega$ 
and any smooth~$(n-k)$-form~$\tau$, there holds
\begin{equation} 
 \begin{split}
   \int_M (\tau, *\omega) \, \vol_g
%   \stackrel{\eqref{Hodge}}{=}
%    \int_M \tau\wedge(**\omega)
%   \stackrel{\eqref{starstar}}{=}
  = (-1)^{k(n-k)}\int_M \tau\wedge\omega
%   = \int_M \omega\wedge\tau
  = (-1)^{k(n-k)} \int_M (\omega, \, *\tau) \, \vol_g
 \end{split}
 \label{commutestar}
\end{equation}
by the very definition of~$*$.

As an immediate consequence of~\eqref{starsharp}, \eqref{starflat},
the operator~$\star$ is sequentially continuous 
with respect to the weak$^*$ convergence
in the sense of distributions.
We recall a few other properties of~$\star$.
For any~$\omega\in\mathcal{D}^\prime(M, \, \Lambda^k\T^*M)$,
we denote by~$\abs{\omega}\!(M)$ the total variation
of~$\omega$ induced by the Riemannian metric on~$M$,
that is
\begin{equation} \label{totalvar}
 \abs{\omega}\!(M) := \sup\left\{\ip{\omega}{v}_{\mathcal{D}^\prime, \mathcal{D}}
 \colon v\in C^\infty(M, \, \Lambda_k\T M), \
 \sup_{x\in M} \abs{v(x)} \leq 1 \right\}
\end{equation}

\begin{prop} \label{prop:Hodge}
 For any~$\omega\in\mathcal{D}^\prime(M, \, \Lambda^k\T^*M)$
 and~$S\in\mathcal{D}^\prime(M, \, \Lambda_k\T M)$,
 the following properties hold:
 \begin{enumerate}[label=(\roman*), ref=\roman*]
  \item $\star\star\omega = (-1)^{k(n-k)} \omega$ 
  and $\star\star S = (-1)^{k(n-k)} S$;
  
  \item $\partial(\star\omega) = (-1)^{k+1} \star\d\omega$;
  
  \item $\star(\partial S) = (-1)^{k} \d(\star S)$;
  
  \item there exists a constant~$C_{n,k}$, depending only on~$n$
  and~$k$, such that
  \[
   |\!\star S|(M) \leq \M(S) 
   \leq C_{n,k} \, |\!\star S|(M)
  \]
  
  \item if~$S$ is an integer-multiplicity rectifiable
  current, then~$\M(S) = |\!\star S|(M)$.
  
%   \item there exists a constant~$C$, depending only on~$M$
%   and~$k$, such that $\norm{\star S}_{W^{-1,1}(M)} \leq C\,\F(S)$.
 \end{enumerate}
\end{prop}

The proof of Proposition~\ref{prop:Hodge} is a rather direct
application of the definitions above
and we omit it for the sake of brevity.

\section{A few technical results}

\subsection{An interpolation lemma}
Let~$k\in\{0, \, 1, \, \ldots, n\}$ and~$p\in [1, \, +\infty]$.
We define the space of~$W^{-1,p}$-forms of dimension~$k$
as a topological dual,
\[
 W^{-1,p}(M, \, \Lambda^k\T^*M) := (W^{1,q}(M, \, \Lambda_k\T M))^\prime
\]
where~$q\in [1, \, +\infty]$ is the H\"older conjugate of~$p$,
such that~$1/p + 1/q = 1$.

\begin{lemma} \label{lemma:interpolation}
 Let~$k\in\Z$, $p\in\R$ be such that $0\leq k \leq n$,
 $1 < p < n/(n-1)$. For any bounded measure~$\omega$ 
 with values in~$k$-forms,
 there holds
 \[
  \norm{\omega}_{W^{-1,p}(M)} \lesssim
  \norm{\omega}_{W^{-1,1}(M)}^\alpha \abs{\omega}\!(M)^{1-\alpha}
 \]
 where $\alpha := 1 + n/p - n \in (0, \, 1)$.
\end{lemma}

Lemma~\ref{lemma:interpolation} is a special case of the more general result below (which is of some independent interest).
\begin{lemma}\label{lemma:interpolation-vec-meas}
 Let~$d\in\N$ and $1 < p < \frac{n}{n-1}$. For any  
 Radon vector measure~$\nu$ on $M$, with values in~$\R^d$,
 there holds
 \[
  \norm{\nu}_{W^{-1,p}(M)} \lesssim
  \norm{\nu}_{W^{-1,1}(M)}^\alpha \abs{\nu}\!(M)^{1-\alpha}
 \]
 where $\alpha := 1 + n/p - n \in (0, \, 1)$.
\end{lemma}
\begin{proof}
Let $\mathcal{A} = \{(U_i,\varphi_i)\}_{i=1}^N$ be a finite regular atlas for $M$, where the (bounded) open sets $U_i$ are contractible, and let $\{\rho_i\}_{i=1}^N$ be a partition of unity subordinate to the covering $\{U_i\}$. Write $\nu = (\nu_1, \dots, \nu_k)$ and, for each $k \in \{1,2,\dots,d\}$ and $i \in \{1,2,\dots,n\}$, define $\mu_{k,i} := (\rho_i \nu_k ) \circ \varphi_i^{-1}$. Then, for every $i \in \{1,2,\dots,n\}$ and every $k \in \{1,2,\dots,d\}$, $\mu_{k,i}$ is a Radon measure on $\Omega_i := \varphi_i(U_i) \subset \R^n$. Therefore, we can argue exactly as in \cite[Lemma~3.3]{JerrardSoner-GL} to deduce that
\[
	\norm{\mu_{k,i}}_{(C^{0,\alpha}_0(\Omega_i))^\prime} \lesssim \norm{\mu_{k,i}}_{(C^{0,1}_0(\Omega_i))^\prime}^\alpha \abs{\mu_{k,i}}(\Omega_i)^{1-\alpha}.
\]
On the other hand, since the bounded open sets $\Omega_i$ have smooth boundary, we have $C^{0,1}_0(\Omega_i) = W^{1,\infty}_0(\Omega_i)$ as Banach spaces, for every $i \in \{1,2,\dots,n\}$. Consequently, $(C^{0,1}_0(\Omega_i))^\prime = W^{-1,1}(\Omega_i)$ as Banach spaces for every $i \in \{1,2,\dots,n\}$. Since for $1 < p < \frac{n}{n-1}$ and $\alpha = 1 + n/p - n$ any functional in $(C^{0,\alpha}_0(\Omega))^\prime$ restricts to a functional in $W^{-1,p}(\Omega)$ and the restriction map is continuous (by Sobolev embedding, cf. Remark~\ref{rk:Calpha}), we infer
\[
	\norm{\mu_{k,i}}_{W^{-1,p}(\Omega_i)} \lesssim \norm{\mu_{k,i}}_{W^{-1,1}(\Omega_i)}^\alpha \abs{\mu_{k,i}}(\Omega_i)^{1-\alpha} \qquad (k = 1, 2, \dots d).
\]
Thus, as each $\mu_{k,i}$ is the push-forward of $\rho_i\nu_k$ through $\varphi_i$,
\[
	\norm{\rho_i \nu_k}_{W^{-1,p}(U_i)} \lesssim \norm{\rho_i \nu_k}_{W^{-1,1}(U_i)}^\alpha \abs{\rho_i \nu_k}(U_i)^{1-\alpha} \qquad (k = 1, 2, \dots d),
\]
and hence
\[
\begin{split}
	\norm{\nu_k}_{W^{-1,p}(M)} &\leq \sum_{i=1}^N \norm{\rho_i \nu_k}_{W^{-1,p}(M)}  = \sum_{i=1}^N \norm{\rho_i \nu_k}_{W^{-1,p}(U_i)} \\
	&\lesssim \sum_{i=1}^N \left\{ \norm{\rho_i \nu_k}_{W^{-1,1}(U_i)}^\alpha \abs{\rho_i \nu_k}(U_i)^{1-\alpha} \right\} \leq \sum_{i=1}^N \left\{\norm{\rho_i \nu_k}_{W^{-1,1}(M)}^\alpha \abs{\rho_i \nu_k}(M)^{1-\alpha}\right\} \\
	&\leq \left\{\sum_{i=1}^N \norm{\rho_i \nu_k}_{W^{-1,1}(M)}^\alpha\right\} \abs{\nu_k}(M)^{1-\alpha} \lesssim \norm{\nu_k}^{\alpha}_{W^{-1,1}(M)} \abs{\nu_k}(M)^{1-\alpha},
\end{split}
\]
for every $k = 1, 2, \dots, d$, and the claimed conclusion follows immediately.
\end{proof}

\begin{proof}[{Proof of Lemma~\ref{lemma:interpolation}}]
	Recalling the (easy but useful) observation at the end of \cite[p.~462]{BaldoOrlandi-Hodge}, by Nash theorem, $M \hookrightarrow \R^{n+r}$ isometrically for some $r \in \N$, hence $\Lambda^{k} \T^* M$ can be seen as a subbundle of $\Lambda^k \T^* \R^{n+r}$, and every bounded measure $\omega$ with values in $k$-forms can be regarded as a (Radon) vector measure on $M$, with values in the Euclidean inner product space $\Lambda^k (\R^{n+r})^\prime \cong \R^{\binom{n+r}{k}}$ (where the last isomorphism of Euclidean spaces is canonical). Calling $\mathcal{V}$ the image of the inclusion $\mathcal{I} : \Lambda^{k} \T^* M \hookrightarrow \R^{\binom{n+r}{k}}$, $\mathcal{V}$ is a finite dimensional Hilbert space space. Then, the conclusion follows from the fact that $W^{1,p}(M, \Lambda^k \T^* M) \cong W^{1,p}(M, \mathcal{V})$ for every $p \in [1,\infty]$ (and hence their duals can be identified, too). Indeed, although the underlined isomorphism depends on the isometric embedding $\mathcal J: M \hookrightarrow \R^{n+r}$, which is generally not unique, any two choices $\mathcal{J}_1$, $\mathcal{J}_2$ for the embedding give rise to equivalent Banach spaces $W^{1,p}(M, \mathcal{V}_1)$, $W^{1,p}(M, \mathcal{V}_2)$ because of the compactness of $M$. Thus, we can apply Lemma~\ref{lemma:interpolation-vec-meas} with $d = \binom{n+r}{k}$ and reach the conclusion.
\end{proof}

% Combining Lemma~\ref{lemma:interpolation} with 
% Proposition~\ref{prop:Hodge}, we immediately obtain
% 
% \begin{corollary} \label{cor:interpolation}
%  Let~$k\in\Z$, $p\in\R$ be such that $0\leq k \leq n$,
%  $1 < p < n/(n-1)$. For any current~$S\in\mathcal{D}^\prime(M, \, \Lambda_k\T M)$
%  with finite mass, there holds
%  \[
%   \|\!\star S^\flat\|_{W^{-1,p}(M)} \lesssim
%   \F(S)^\alpha \, \M(S)^{1-\alpha}
%  \]
%  where $\alpha := 1 + n/p - n \in (0, \, 1)$.
% \end{corollary}

\subsection{Elliptic regularity} 
\label{sec:elliptic-reg}

We establish here two results concerning existence, uniqueness and estimates for solutions to London (Lemma~\ref{lemma:elliptic_reg}) and Poisson (Lemma~\ref{lemma:elliptic_reg_bis}) equations for differential $k$-forms and data in $W^{-1,p}$. Although Lemma~\ref{lemma:elliptic_reg} and Lemma~\ref{lemma:elliptic_reg_bis} are certainly known to experts, we have not found explicit proofs in the literature. Since they are crucial to our arguments, we provide detailed proofs. %Lemma~\ref{lemma:classic-elliptic-reg} is instead well-known and can be found, for instance, in the classical reference \cite{GiaquintaModicaSoucek-I}, but we stated it explicitly to point out the estimate~\eqref{eq:elliptic-est-poisson-2}, of importance later in this work.

\begin{lemma} \label{lemma:elliptic_reg}
 Let~$p\in (1, \, 2)$. For any~$k$-form $f\in W^{-1,p}(M, \, \Lambda^k\T^*M)$, 
 the equation
 \begin{equation} \label{elliptic_eq}
  -\Delta v + v = f \qquad \textrm{in the sense of } \mathcal{D}^\prime(M)
 \end{equation}
 has a unique solution~$v\in W^{1,p}(M, \, \Lambda^k\T^* M)$,
 which satisfies
 \begin{equation}\label{eq:elliptic-est-london}
  \norm{v}_{W^{1,p}(M)} \leq C_p \norm{f}_{W^{-1,p}(M)}
 \end{equation}
 for some constant~$C_p$ depending only on~$M$, $k$, $p$.
\end{lemma}

The proof of Lemma~\ref{lemma:elliptic_reg} depends on 
Gaffney's inequality, see Proposition~\ref{prop:gaffney-closed}.

\begin{proof}[Proof of Lemma~\ref{lemma:elliptic_reg}]
 We split the proof into several steps.
 \setcounter{step}{0}
 \begin{step} 
  Let~$q > 2$. We claim the following: for 
  any~$f\in L^q(M, \, \Lambda^k\T^*M)$,
  the equation~\eqref{elliptic_eq} has a unique solution~$v\in W^{2,q}(M, \, \Lambda^k\T^*M)$, which satisfies
  \begin{equation} \label{elliptic1}
   \norm{v}_{W^{2,q}(M)} \leq C_q \norm{f}_{L^q(M)}
  \end{equation}
  for some constant~$C_q$ that depends only on~$M$, $k$ and~$q$.
  %{\BBB \textbf{I suppose it must be classical, but here is an argument anyway.}}
  For any~$f\in L^2(M, \, \Lambda^k\T^*M)$, existence and uniqueness
  of a solution~$v\in W^{1,2}(M, \, \Lambda^k\T^* M)$ follow from
  Lax-Milgram lemma combined with Gaffney's inequality,
  Proposition~\ref{prop:gaffney-closed}. By elliptic regularity
  (see, for instance, \cite[Theorem~10.3.11]{Nicolaescu}),
  %{\BBB\textbf{[Theorem~10.3.11 in Nicolaescu's notes, but we must
  %find a better reference]}}
  if~$f\in L^q(M, \, \Lambda^k\T^* M)$
  then~$u\in W^{2,q}(M, \, \Lambda^k\T^* M)$ and 
  \begin{equation} \label{elliptic2}
   \norm{v}_{W^{2,q}(M)} \leq C_q
    \left(\norm{f}_{L^q(M)} + \norm{v}_{L^q(M)}\right) 
  \end{equation}
  for some constant~$C_q$ that depends only on~$M$, $k$ and~$q$.
  Now, there exists a number~$s > q$ such that
  $W^{2,q}(M, \, \Lambda^k\T^*M)$ embeds continuously 
  in~$L^{s}(M, \, \Lambda^k\T^*M)$ (see, e.g., \cite[Theorem~1.3.6]{SchwarzG}).
  %{\BBB\textbf{[Find a reference --- Euclidean result + local coordinates.
  %Nota: stavolta stimo la norma~$L^{s}$ con tutta la norma~$W^{2,q}$,
  %compresi termini di ordine zero e uno,
  %quindi non dovrebbero esserci problemi.]}}
  By interpolation, there exists a 
  number~$\alpha = \alpha(q, \, s)\in (0, \, 1)$ such that
  \begin{equation*}
   \norm{v}_{L^q(M)}
   \leq \norm{v}_{L^{s}(M)}^\alpha \norm{v}^{1-\alpha}_{L^2(M)}
   \leq C_q \norm{v}_{W^{2,q}(M)}^\alpha \norm{v}^{1-\alpha}_{L^2(M)}
  \end{equation*}
  By applying Young's inequality, for any~$\delta>0$
  we find a constant~$C_\delta$ (depending on~$\alpha$, $M$, $k$ and~$q$ as well)
  such that
  \begin{equation} \label{elliptic3} 
   \norm{v}_{L^q(M)}
    \leq \delta \norm{v}_{W^{2,q}(M)} + C_\delta \norm{v}_{L^2(M)}
  \end{equation}
  By comparing~\eqref{elliptic2} and~\eqref{elliptic3}, 
  and choosing~$\delta$ small enough, we deduce
  \begin{equation} \label{elliptic4}
   \norm{v}_{W^{2,q}(M)} \leq C_q
    \left(\norm{f}_{L^q(M)} + \norm{v}_{L^2(M)}\right) 
  \end{equation}
  However, by testing the equation~\eqref{elliptic_eq} against~$v$,
  we obtain
  \[
   \norm{v}_{L^2(M)} \leq \norm{f}_{L^2(M)} \leq C_q \norm{f}_{L^q(M)}
  \]
  and~\eqref{elliptic1} follows.
 \end{step}
 
 \begin{step}
  Now, take~$p\in (1, \, 2)$. We claim that, for
  any~$f\in W^{-2,p}(M, \, \Lambda^k\T^* M)$, the
  equation~\eqref{elliptic_eq} has exactly one 
  solution~$v\in L^p(M, \, \Lambda^k\T^*M)$, which satisfies
  \begin{equation} \label{elliptic-2}
   \norm{v}_{L^p(M)} \leq C_p \norm{f}_{W^{-2,p}(M)}
  \end{equation}
  for some constant~$C_p$ that depends only on~$M$, $k$ and~$p$.
  	%We claim that for every $f \in W^{-2,p}(M, \Lambda^k \T^*M)$ there is a unique $v \in L^p(M, \Lambda^k \T^*M)$ such that
  %  \[
 % 	\forall w \in W^{2,q}(M,\,\Lambda^k \T^*M),\quad (w,-\Delta v+v) = \ip{w}{f}_{W^{2,q},W^{-2,p}} \,.
%  \]
%  By the symmetry of $\Delta$, the previous equality can be rewritten as
%  \[
 % 	\forall w \in W^{2,q}(M,\,\Lambda^k \T^*M),\quad (v,-\Delta w +w) = \ip{f}{w^\#}_{W^{-2,p}, W^{2,q}} \,,
 % \]
 % i.e., solving~\eqref{elliptic_eq_bis} .
  To prove the claim, we will first show existence and uniqueness of a duality solution $v \in L^p(M, \Lambda^k \T^*M)$ and then we will prove that every solution $v \in L^p(M, \Lambda^k \T^*M)$ in the sense of $\mathcal{D}'(M)$ is a duality solution.
   
  We notice in first place that $p \in (1,2)$ implies $q := p' > 2$.  By Step~1, given $h \in L^q(M, \Lambda^k \T^*M)$, there exists a uniquely determined $w_h \in W^{2,q}(M, \Lambda^k \T^*M)$ such that $w_h = (-\Delta + {\rm Id})^{-1}h$ (that it is to say, solving $-\Delta w_h + w_h = h$), which moreover satisfies $\norm{w_h}_{W^{2,q}} \leq C_p \norm{h}_{L^q}$, where $C_p > 0$ is constant depending only on $p$, $k$ and $M$. 
  
  The map $V\colon L^q(M, \Lambda^k \T^*M) \to \R$ defined by
  \begin{equation}\label{eq:V(h)}
  	 \forall h \in L^q(M, \Lambda^k \T^* M), \quad V(h) := \ip{f}{\star((-\Delta + {\rm Id})^{-1} h)}_{W^{-2,p},W^{2,q}} 
  \end{equation}
  is a bounded linear functional over $L^q(M, \Lambda^k \T^*M)$. Indeed, by Step~1,
  \[
  \begin{split}
  	\forall h \in L^q(M, \Lambda^k \T^* M), \quad \abs{V(h)} &\leq \norm{f}_{W^{-2,p}} \norm{\star((-\Delta + {\rm Id})^{-1} h)}_{W^{2,q}} \\
  	&= \norm{f}_{W^{-2,p}} \norm{w_h}_{W^{2,q}} \lesssim \norm{f}_{W^{-2,p}} \norm{h}_{L^q} ,
  \end{split}
  \] 
  hence 
  \begin{equation}\label{eq:est-V-aux}
  	\norm{V}_{(L^q)'} \leq C_p \|f\|_{W^{-2,p}},
  \end{equation}
  where the constant $C_p$ depends only on $p$, $k$ and $M$. Consequently, Riesz' theorem implies the existence of a unique $v \in L^p(M,\,\Lambda^k \T^*M)$ such that
  \[
  	\forall h \in L^q(M, \Lambda^k \T^* M), \quad (v,h) = \ip{f}{\star((-\Delta + {\rm Id})^{-1} h)}_{W^{-2,p},W^{2,q}},
  \]
  i.e., replacing $h$ by the corresponding $w = w_h \in W^{2,q}(M, \Lambda^k \T^* M)$, 
  \[
  	\forall w \in W^{2,q}(M, \Lambda^k \T^*M), \quad (v, -\Delta w + w) = \ip{f}{\star w}_{W^{-2,p},W^{2,q}}.
  \]
  Hence, $v$ is the unique duality solution of~\eqref{elliptic_eq} for the given $f \in W^{-2,p}(M,\Lambda^k \T^*M)$. Since Riesz' theorem also implies $\norm{v}_{L^p} = \norm{V}_{(L^q)'}$, by~\eqref{eq:est-V-aux} $v$ satisfies the estimate~\eqref{elliptic2}.
 % To prove \eqref{elliptic-2}, notice that, being $L^q(M,\,\Lambda^k \T^* M)$ reflexive, James' theorem assures the existence of $g \in L^q(M,\,\Lambda^k \T^*M)$ with $\| g \|_{L^q} = 1$ and such that $\| v \|_{L^p} = (g,v)$. On the other hand, by Step~1, there is exactly one solution $w \in W^{2,q}(M,\,\Lambda^k \T^* M)$ to \eqref{elliptic_eq} with $f$ replaced by $g$. Testing the equation $- \Delta v + v = f$ against $w$, by the symmetry of $\Delta$ we have
%$(w,\Delta v) = (\Delta w,v)$ and, in turn,
%\[
%	(-\Delta w + w,v) = \ip{w}{f}_{W^{2,q},W^{-2,p}}\,.
%\]
%Since $w$ is a strong solution to \eqref{elliptic_eq}, $-\Delta w + w = g$ a.e., hence the above equality can be rewritten as
%\[
	%(g,v) = \ip{w}{f}_{W^{2,q},W^{-2,p}}\,.
%\]
%By the choice of $g$, $(g,v) = \| v \|_{L^p}$ and, by definition, $\ip{w}{f}_{W^{2,q},W^{-2,p}} \leq \| w \|_{W^{2,q}} \| f \|_{W^{-2,p}}$. Again by Step~1, there is a constant $C_q > 0$ s.t. $\|w \|_{W^{2,q}} \leq C_q \|g\|_{L^q}$. On the other hand, $q = p'$, hence we can see $C_q$ as a constant $C_p$ depending on $p$, and this completes the proof.
 \end{step}
 
 	Conversely, every distributional solution to~\eqref{elliptic_eq} is a duality solution. Indeed, suppose $v \in L^p(M, \Lambda^k \T^*M)$ is a distributional solution to~\eqref{elliptic_eq} for $f \in W^{-2,p}(M, \Lambda^k \T^*M)$. Then, for every $w \in C^\infty(M, \Lambda^k \T^*M)$,
 	\begin{equation}\label{eq:distrib-to-duality}
 	\begin{split}
 		(v, -\Delta w +w) &= \ip{v}{\star(-\Delta w + w)}_{\mathcal{D}', \mathcal{D}} = \ip{-\Delta v + v}{\star w}_{\mathcal{D}', \mathcal{D}} \\
 		 &= \ip{f}{\star w}_{W^{-2,p}, W^{2,q}}.
 	\end{split}
 	\end{equation}
 	Letting again $h = -\Delta w + w$, then $h \in L^q(M,\,\Lambda^k \T^*M)$ and, from~\eqref{eq:V(h)} and~\eqref{eq:distrib-to-duality}, we have $(v,h) = V(h)$. By the density of $C^\infty(M,\,\Lambda^k \T^*M)$ into $W^{2,q}(M,\, \Lambda^k \T^*M)$, the equality $(v,h) = V(h)$ actually holds for every $h \in L^q(M, \Lambda^k \T^*M)$. Thus, $v$ is duality solution. Therefore, for every $f \in W^{-2,p}(M,\,\Lambda^k \T^*M)$, Eq.~\eqref{elliptic_eq} has a unique distributional solution in $L^p(M,\,\Lambda^k \T^*M)$, which is the duality solution.
 
 \begin{step}
  Let~$f\in W^{-1,p}(M, \, \Lambda^k\T^* M)$ and let~$v$ 
  be the unique solution of~\eqref{elliptic_eq}. By Step~2, we have
  \begin{equation} \label{elliptic3-1}
   \norm{v}_{L^p(M)} \leq C_p \norm{f}_{W^{-2,p}(M)} 
    \leq C_p \norm{f}_{W^{-1,p}(M)}
  \end{equation}
  By differentiating the equation~\eqref{elliptic_eq}, we obtain
  $-\Delta(\d v) + \d v = \d f$. Therefore, Step~2 implies
  \begin{equation} \label{elliptic3-2}
   \norm{\d v}_{L^p(M)} \leq C_p \norm{\d f}_{W^{-2,p}(M)} 
    \leq C_p \norm{f}_{W^{-1,p}(M)}
  \end{equation}
  Similarly, $-\Delta(\d^* v) + \d^* v = \d^* f$ and hence
  \begin{equation} \label{elliptic3-3}
   \norm{\d^* v}_{L^p(M)} \leq C_p \norm{\d^* f}_{W^{-2,p}(M)} 
    \leq C_p \norm{f}_{W^{-1,p}(M)}
  \end{equation}
  Combining~\eqref{elliptic3-1}, \eqref{elliptic3-2}, \eqref{elliptic3-3}
  with Gaffney's inequality (Proposition~\ref{prop:gaffney-closed}),
  the lemma follows. \qedhere
 \end{step}
\end{proof}

Here is a variant of Lemma~\ref{lemma:elliptic_reg},
which will also be useful. %We denote by~$\Harm^k(M)$ 
%the space of harmonic~$k$-forms on~$M$. ---Già introdotto

\begin{lemma} \label{lemma:elliptic_reg_bis}
 Let~$p\in (1, \, 2)$. Let~$f\in W^{-1,p}(M, \, \Lambda^k\T^*M)$
 be a $k$-form such that
 \[
  \ip{f}{\star \xi}_{W^{-1,p},W^{1,p'}} = 0 \qquad
  \textrm{for any } \xi\in\Harm^k(M)
 \]
 Then, there exists a unique~$v\in W^{1,p}(M, \, \Lambda^k\T^*M)$ such that
 \begin{equation} \label{elliptic_eq_bis}
  \begin{cases}
   -\Delta v  = f &\textrm{in the sense of } \mathcal{D}^\prime(M) \\
   \displaystyle\int_M \ip{v}{\xi} \vol_g = 0 
    &\textrm{for any } \xi\in\Harm^k(M)
  \end{cases}
 \end{equation}
 Moreover, $v$ satisfies
 \begin{equation}\label{eq:elliptic-est-poisson}
  \norm{v}_{W^{1,p}(M)} \leq C_p \norm{f}_{W^{-1,p}(M)}
 \end{equation}
 for some constant~$C_p$ depending only on~$M$, $k$, $p$.
\end{lemma}

\begin{proof}
	%Uniqueness is nearly immediate: indeed, were $v_1$, $v_2 \in W^{1,p'}(M,\,\Lambda^k \T^*M)$ both solutions of \eqref{elliptic_eq_bis}, the difference $v_0 := v_1 - v_2$ would belong to $W^{1,p'}(M,\,\Lambda^k \T^*M)$, solve $-\Delta v = 0$ in $\mathcal{D}^\prime(M)$ and still be orthogonal to all harmonic $k$-forms. By $L^p$-Hodge decomposition \cite[Proposition~6.5]{Scott} (or \cite[Section~6]{IwaniecScottStroffolini}) it would then follow that $v_0$ is a harmonic $k$-form. But, since $\ip{v_0}{\xi} = 0$ for every $\xi \in \Harm^k(M)$, we would have $v_0 = 0$, and thus $v_1 = v_2$.
	We proceed in three steps.
	
	\setcounter{step}{0}
	\begin{step}[Existence and uniqueness of a duality solution]	
	 Set $q := p'$.	We claim that for any $f \in W^{-2,p}(M, \Lambda^k \T^* M)$ such that $\ip{f}{\star \xi}_{W^{-2,p},W^{2,q}} = 0$ for every $\xi \in \Harm^k(M)$ there exists a unique distributional solution $v \in L^p(M,\,\Lambda^k \T^*M)$ to \eqref{elliptic_eq_bis}, which satisfies the estimate
	\begin{equation}\label{eq:est-v-aux}
		\|v\|_{L^p} \lesssim \|f\|_{W^{-2,p}}.
	\end{equation}
	Define $V\colon L^q(M, \Lambda^k \T^*M) \to \R$ setting
	\[
		\forall \alpha \in L^q(M, \Lambda^k \T^*M), \quad V(\alpha) := \ip{f}{-(\star G(\alpha))}_{W^{-2,p},W^{2,q}},
	\]
	where $G : L^q(M, \Lambda^k \T^*M) \to W^{2,q}(M, \Lambda^k \T^*M) \cap (\Harm^k(M))^\perp$ is Green's operator, which is linear and bounded (see Section~\ref{subsec:sobolev}). Then $V$ is linear and bounded, because 
	\[
		\forall \alpha \in L^q(M, \Lambda^k \T^*M), \quad \abs{V(\alpha)} \leq \norm{f}_{W^{-2,p}} \norm{G(\alpha)}_{W^{2,q}} \lesssim \norm{f}_{W^{-2,p}} \norm{\alpha}_{L^q},
	\]
	whence
	\[
		\norm{V}_{(L^q)'} \lesssim \norm{f}_{W^{-2,p}}.
	\]
	It then follows from Riesz' Representation theorem that there exists a uniquely determined $v \in L^p(M, \Lambda^k \T^* M)$ representing $V$, i.e., such that $V(\alpha) = (v,\alpha)$ for all $\alpha \in L^q(M, \Lambda^k \T^*M)$. Thus,
	\begin{equation}\label{eq:v-elliptic_eq_bis}
		\forall \alpha \in L^q(M, \Lambda^k \T^*M), \quad (v,\alpha) = \ip{f}{-(\star G(\alpha))}_{W^{-2,p},W^{2,q}},
	\end{equation}
	and moreover $\norm{v}_{L^p} = \norm{V}_{(L^q)'}$, whence $\norm{v}_{L^p} \lesssim \|f\|_{W^{-2,p}}$, i.e., \eqref{eq:est-v-aux}. In addition, \eqref{eq:v-elliptic_eq_bis} implies $(v,\xi) = 0$ for every $\xi \in \Harm^k(M)$. 
	
	We now show that $v$ solves \eqref{elliptic_eq_bis} in the sense of $\mathcal{D}'(M)$. To this purpose, we first recall from Remark~\ref{rk:Delta-G} that $\Delta G(\xi) = G(\Delta \xi)$ for every form $\xi \in W^{2,q}(M,\Lambda^k \T^* M)$. Then, we notice that, upon writing $\xi = \Delta G(\xi) + H(\xi)$, where $H(\xi)$ is the harmonic part of $\xi$, by assumption we have $\ip{f}{\star H(\xi)}_{W^{-2,p},W^{2,q}} = 0$. Hence, by~\eqref{eq:v-elliptic_eq_bis} we deduce that
	%we recall that by $L^p$-Hodge decomposition \cite[Proposition~6.5]{Scott}, for every $\xi \in W^{2,q}(M, \Lambda^k \T^*M)$ we have $\xi = \Delta G(\xi) + H(\xi)$, where $H(\xi)$ is the harmonic part of $\xi$. Thus, by (a) the previous remark; (b) the fact that $\ip{f}{H(\xi)^\#}_{W^{-2,p},W^{2,q}} = 0$ for every $\xi \in W^{2,q}(M,\Lambda^k \T^*M)$, and (c) the characterization of $v$, we deduce that the equalities
	\[
	%\begin{split}
		\ip{f}{\star \xi}_{W^{-2,p},W^{2,q}} = \ip{f}{\star(\Delta G(\xi))}_{W^{-2,p},W^{2,q}} = \ip{f}{\star(G(\Delta \xi))}_{W^{-2,p},W^{2,q}} \overset{\eqref{eq:v-elliptic_eq_bis}}{=} (v, -\Delta \xi), 
	%\end{split}
	\]
	for all $\xi \in W^{2,q}(M, \Lambda^k \T^*M)$. Thus, $v$ solves \eqref{elliptic_eq_bis} in the sense of distributions and satisfies estimate~\eqref{eq:est-v-aux}.
	\end{step}
	
	\begin{step}[Every distributional solution is a duality solution]
		We argue exactly as in~\eqref{eq:distrib-to-duality}, with the operator $-\Delta + {\rm Id}$ replaced by $\Delta G$.
	\end{step}

	\begin{step}[Estimate]
	To conclude, let $f \in W^{-1,p}(M, \Lambda^k \T^*M)$ and let $v \in L^p(M, \Lambda^k \T^*M)$ be the corresponding unique solution to \eqref{elliptic_eq_bis}. By~\eqref{eq:est-v-aux}, we have 
	\begin{equation}\label{eq:est-v_elliptic_eq_bis}
		\norm{v}_{L^p} \lesssim \norm{f}_{W^{-2,p}} \lesssim \norm{f}_{W^{-1,p}}.
	\end{equation}
	By differentiating Eq.~\eqref{elliptic_eq_bis} exactly as we did for \eqref{elliptic_eq} in Step~3 of the proof of Lemma~\ref{lemma:elliptic_reg}, we obtain
	\begin{equation}\label{eq:est-dv-d*v_elliptic_eq_bis}
		\norm{\d v}_{L^p(M)} \lesssim \norm{f}_{W^{-1,p}(M)} \quad \mbox{and} \quad \norm{\d^* v}_{L^p(M)} \lesssim \norm{f}_{W^{-1,p}(M)}.	
	\end{equation}
	Then, both $v \in W^{1,p}(M,\,\Lambda^k \T^*M)$ and estimate~\eqref{eq:elliptic-est-poisson} follow by \eqref{eq:est-v_elliptic_eq_bis}, \eqref{eq:est-dv-d*v_elliptic_eq_bis}, and Gaffney's inequality (Proposition~\ref{prop:gaffney-closed}).
	%=================================================
	% 2002/04/09 Argomento dettagliato commentato 
	% perché identico all'analogo nel lemma precedente
	%=================================================
	% To conclude, let $f \in W^{-1,p}(M, \Lambda^k \T^*M)$ and let $v \in L^p(M, \Lambda^k \T^*M)$ be the corresponding unique solution to \eqref{elliptic_eq_bis}. Notice that we have
	%\begin{equation}\label{eq:poisson-est-v}
	%	\norm{v}_{L^p} \lesssim \norm{f}_{W^{-2,p}} \lesssim \norm{f}_{W^{-1,p}}	
	%\end{equation}
	%Then, by differentiating \eqref{elliptic_eq_bis} we obtain $-\Delta(\d v) = \d f$, and hence
	%\begin{equation}\label{eq:poisson-est-dv}
	%	\norm{\d v}_{L^p} \lesssim \norm{\d f}_{W^{-2,p}} \lesssim \norm{f}_{W^{-1,p}}.
	%\end{equation}
	%Similarly, we $-\Delta(\d^* v) = \d^* f$, whence
	%\begin{equation}\label{eq:poisson-est-d*v}
	%	\norm{\d^* v}_{L^p} \lesssim \norm{\d^* f}_{W^{-2,p}} \lesssim \norm{f}_{W^{-1,p}},
	%\end{equation}
	%and the conclusion follows by combining \eqref{eq:poisson-est-v}, \eqref{eq:poisson-est-dv}, and \eqref{eq:poisson-est-d*v} with Gaffney's inequality. 
	\qedhere
	\end{step}
\end{proof}

\end{appendix}

\bibliographystyle{plain}
\bibliography{GL-bundles}

\newcommand{\noop}[1]{}
\begin{thebibliography}{10}

\bibitem{ABO1}
G.~Alberti, S.~Baldo, and G.~Orlandi.
\newblock Functions with prescribed singularities.
\newblock {\em J. Eur. Math. Soc. (JEMS)}, 5(3):275--311, 2003.

\bibitem{ABO2}
G.~Alberti, S.~Baldo, and G.~Orlandi.
\newblock Variational convergence for functionals of {G}inzburg-{L}andau type.
\newblock {\em Indiana Univ. Math. J.}, 54(5):1411--1472, 2005.

\bibitem{BairdWood}
P.~Baird and J.~Wood.
\newblock {\em Harmonic morphisms between Riemannian manifolds}.
\newblock London Mathematical Society Monographs. New Series, 29. The Clarendon
  Press, Oxford University Press, Oxford, 2003.

\bibitem{BaldoOrlandi1997}
S.~Baldo and G.~Orlandi.
\newblock Cycles of least mass in a riemannian manifold, described through the
  ``phase transition'' energy of the sections of a line bundle.
\newblock {\em Math. Z}, 225(4):639--655, 1997.

\bibitem{BaldoOrlandi-Hodge}
S.~Baldo and G.~Orlandi.
\newblock A note on the {H}odge theory for functionals with linear growth.
\newblock {\em Manuscripta Math.}, 97:453--467, 1998.

\bibitem{BaldoOrlandi1999}
S.~Baldo and G.~Orlandi.
\newblock Codimension one minimal cycles with coefficients in $\mathbb{Z}$ or
  $\mathbb{Z}_p$, and variational functionals on fibered spaces.
\newblock {\em J. Geom. Anal.}, 9(4):547--568, 1999.

\bibitem{BaldoOrlandiWeitkamp}
S.~Baldo, G.~Orlandi, and S.~Weitkamp.
\newblock Convergence of minimizers with local energy bounds for the
  {G}inzburg-{L}andau functionals.
\newblock {\em Indiana Univ. Math. J.}, 58(5):2369--2407, 2009.

\bibitem{BCS}
J.~Bardeen, L.~N. Cooper, and J.~R. Schrieffer.
\newblock Microscopic theory of superconductivity.
\newblock {\em Phys. Rev.}, 106:162--164, Apr 1957.

\bibitem{BellettiniNovagaOrlandi2010}
G.~Bellettini, M.~Novaga, and G.~Orlandi.
\newblock Time-like minimal submanifolds as singular limits of nonlinear wave
  equations.
\newblock {\em Phys. D}, 239(6):335--339, 2010.

\bibitem{BellettiniNovagaOrlandi2012}
G.~Bellettini, M.~Novaga, and G.~Orlandi.
\newblock Lorentzian varifolds and applications to relativistic strings.
\newblock {\em Indiana Univ. Math. J.}, 61(6):2251--2310, 2012.

\bibitem{BBH}
F.~Bethuel, H.~Brezis, and F.~H{\'e}lein.
\newblock {\em Ginzburg-{L}andau {V}ortices}.
\newblock Progress in Nonlinear Differential Equations and their Applications,
  13. Birkh\"auser Boston Inc., Boston, MA, 1994.

\bibitem{BethuelBrezisOrlandi}
F.~Bethuel, H.~Brezis, and G.~Orlandi.
\newblock Asymptotics for the {G}inzburg-{L}andau equation in arbitrary
  dimensions.
\newblock {\em J. Funct. Anal.}, 186(2):432--520, 2001.

\bibitem{BOSFactToulous}
F.~Bethuel, G.~Orlandi, and D.~Smets.
\newblock Motion of concentration sets in {Ginzburg-Landau} equations.
\newblock {\em Annales de la Facult\'e des sciences de Toulouse :
  Math\'ematiques}, Ser. 6, 13(1):3--43, 2004.

\bibitem{BethuelOrlandiSmets-Annals}
F.~Bethuel, G.~Orlandi, and D.~Smets.
\newblock Convergence of the parabolic ginzburg–landau equation to motion by
  mean curvature.
\newblock {\em Ann. Math.}, 163(1):37--163, 2006.

\bibitem{BethuelRiviere}
F.~Bethuel and T.~Rivi{\`e}re.
\newblock Vortices for a variational problem related to superconductivity.
\newblock {\em Ann. Inst. H. Poincar\'e Anal. Non Lin\'eaire}, 12(3):243--303,
  1995.

\bibitem{BethuelOrlandi}
{Bethuel, F.} and {Orlandi, G.}
\newblock Uniform estimates for the parabolic ginzburg-landau equation.
\newblock {\em ESAIM: COCV}, 8:219--238, 2002.

\bibitem{BottTu}
R.~Bott and L.~W. Tu.
\newblock {\em Differential forms in algebraic topology}, volume~82 of {\em
  Graduate Texts in Mathematics}.
\newblock Springer-Verlag, New York-Berlin, 1982.

\bibitem{Bradlow1990}
S.~B. Bradlow.
\newblock Vortices in holomorphic line bundles over closed {K}{\"a}hler
  manifolds.
\newblock {\em Communications in Mathematical Physics}, 135(1):1--17, 1990.

\bibitem{Bradlow1991}
S.~B. Bradlow.
\newblock {Special metrics and stability for holomorphic bundles with global
  sections}.
\newblock {\em J. Diff. Geom.}, 33(1):169--213, 1991.

\bibitem{BrockerJanich}
T.~Br\"{o}cker and K.~J\"{a}nich.
\newblock {\em Introduction to differential topology}.
\newblock Cambridge University Press, Cambridge-New York, 1982.
\newblock Translated from the German by C. B. Thomas and M. J. Thomas.

\bibitem{Budney}
R.~L. Budney.
\newblock {\em Regularity of $\mathcal{A}$-harmonic forms}.
\newblock PhD thesis, Syracuse University, 1996.

\bibitem{CanevariDipasqualeOrlandiII}
G~Canevari, F.~Dipasquale, and G.~Orlandi.
\newblock In preparation.

\bibitem{AG-Reno}
G.~Canevari and A.~Segatti.
\newblock Dynamics of {G}inzburg-{L}andau vortices for vector fields on
  surfaces.
\newblock Preprint arXiv~2108.01321, 2015.

\bibitem{ChenSternberg}
K.-S. Chen and P.~Sternberg.
\newblock Dynamics of {G}inzburg-{L}andau and {G}ross-{P}itaevskii vortices on
  manifolds.
\newblock {\em Discrete Contin. Dyn. Syst.}, 34(5):1905--1931, 2014.

\bibitem{Cheng2021}
D.~R. Cheng.
\newblock Stable solutions to the abelian {Y}ang-{M}ills-{H}iggs equations on
  ${S}^2$ and ${T}^2$.
\newblock {\em J. Geom. Anal.}, 31(10):9551--9572, 2021.

\bibitem{Colinet}
A.~Colinet.
\newblock Structural descriptions of limits of the parabolic
  {G}inzburg-{L}andau equation on closed manifolds.
\newblock Preprint arXiv~2107.13582, 2021.

\bibitem{DeMailly}
J.~P. DeMailly.
\newblock {\em Complex Analytic and Differential Geometry}.
\newblock OpenContent Book, 2012.

\bibitem{Donaldson}
S.~K. Donaldson.
\newblock Anti self-dual {Y}ang-{M}ills connections over complex algebraic
  surfaces and stable vector bundles.
\newblock {\em Proceedings of the London Mathematical Society}, s3-50(1):1--26,
  1985.

\bibitem{DuGunzburgerPeterson}
Q.~Du, M.~Gunzburger, and J.~Peterson.
\newblock Analysis and approximation of the {G}inzburg–{L}andau model of
  superconductivity.
\newblock {\em SIAM Review}, 34(1):54--81, 1992.

\bibitem{Federer}
H.~Federer.
\newblock {\em Geometric measure theory}.
\newblock Die Grundlehren der mathematischen Wissenschaften, Band 153.
  Springer-Verlag New York Inc., New York, 1969.

\bibitem{Garcia-Prada}
O.~García-Prada.
\newblock {Invariant connections and vortices}.
\newblock {\em Comm. Math. Phys.}, 156(3):527 -- 546, 1993.

\bibitem{GiaquintaModicaSoucek-I}
M.~Giaquinta, G.~Modica, and J.~Sou{\v{c}}ek.
\newblock {\em Cartesian currents in the calculus of variations. {I}},
  volume~37 of {\em Ergebnisse der Mathematik und ihrer Grenzgebiete. 3. Folge.
  A Series of Modern Surveys in Mathematics [Results in Mathematics and Related
  Areas. 3rd Series. A Series of Modern Surveys in Mathematics]}.
\newblock Springer-Verlag, Berlin, 1998.
\newblock Cartesian currents.

\bibitem{GinzburgLandau}
V.~L. Ginzburg and L.~D. Landau.
\newblock On the theory of superconductivity.
\newblock In D.~ter Haar, editor, {\em Collected Papers of L.~D.~Landau}, pages
  546--568. Pergamon, New York, 1965.

\bibitem{Guneysu}
B.~G\"{u}neysu.
\newblock {\em Covariant Schrödinger Semigroups on Riemannian Manifolds}.
\newblock Operator Theory: Advances and Applications. Birkhäuser Cham, 2017.

\bibitem{Hajlasz2000}
Piotr Haj{\l}asz.
\newblock Sobolev mappings, co-area formula and related topics.
\newblock In {\em Proceedings on {A}nalysis and {G}eometry ({R}ussian)
  ({N}ovosibirsk {A}kademgorodok, 1999)}, pages 227--254. Izdat. Ross. Akad.
  Nauk Sib. Otd. Inst. Mat., Novosibirsk, 2000.

\bibitem{HongJostStruwe}
M.-C. Hong, J.~Jost, and M.~Struwe.
\newblock Asymptotic limits of a {G}inzburg-{L}andau type functional.
\newblock In J.~Jost, editor, {\em Geometric Analysis and the Calculus of
  Variations for Stefan Hildebrandt}, pages 99--124. International Press, 1996.

\bibitem{IgnatJerrard}
R.~Ignat and R.~L. Jerrard.
\newblock Renormalized energy between vortices in some {G}inzburg-{L}andau
  models on 2-dimensional {R}iemannian manifolds.
\newblock {\em Arch. Ration. Mech. Anal.}, 239(3):1577--1666, 2021.

\bibitem{IwaniecScottStroffolini}
T.~Iwaniec, C.~Scott, and B.~Stroffolini.
\newblock Nonlinear {H}odge theory on manifolds with boundary.
\newblock {\em Mat. Ann. Pura Appl.}, 177(1):37--115, 1999.

\bibitem{JaffeTaubes}
A.~Jaffe and C.~Taubes.
\newblock {\em Vortices and monopoles}, volume~2 of {\em Progress in Physics}.
\newblock Birkh\"{a}user, Boston, Mass., 1980.
\newblock Structure of static gauge theories.

\bibitem{Jerrard2011}
R.~L. Jerrard.
\newblock Defects in semilinear wave equations and timelike minimal surfaces in
  {M}inkowski space.
\newblock {\em Anal. PDE}, 4(2):285--340, 2011.

\bibitem{JerrardSoner-GL}
R.~L. Jerrard and H.~M. Soner.
\newblock The {J}acobian and the {G}inzburg-{L}andau energy.
\newblock {\em Cal. Var. Partial Differential Equations}, 14(2):151--191, 2002.

\bibitem{JerrardSoner-Jacobians}
R.~L. Jerrard and H.~M. Soner.
\newblock Functions of bounded higher variation.
\newblock {\em Indiana Univ. Math. J.}, 51(3):645--677, 2003.

\bibitem{Jost}
J.~Jost.
\newblock {\em Riemannian geometry and geometric analysis}.
\newblock Universitext. Springer, Cham, seventh edition, 2017.

\bibitem{Lee}
J.~M. Lee.
\newblock {\em Manifolds and Differential Geometry}, volume 107 of {\em
  Graduate Studies in Mathematics}.
\newblock American Mathematical Society, Providence, Rhode Island, 2009.

\bibitem{ModicaMortola}
L.~Modica and S.~Mortola.
\newblock Un esempio di {$\Gamma \sp{-}$}-convergenza.
\newblock {\em Boll. Un. Mat. Ital. B (5)}, 14(1):285--299, 1977.

\bibitem{Morrey}
C.~B. Morrey, Jr.
\newblock {\em Multiple integrals in the calculus of variations}.
\newblock Die Grundlehren der mathematischen Wissenschaften, Band 130.
  Springer-Verlag New York, Inc., New York, 1966.

\bibitem{Munkres}
J.~R. Munkres.
\newblock {\em Elementary differential topology}.
\newblock Annals of Mathematics Studies, No. 54. Princeton University Press,
  Princeton, N.J., revised edition, 1966.
\newblock Lectures given at Massachusetts Institute of Technology, Fall, 1961.

\bibitem{Nicolaescu}
L.~Nicolaescu.
\newblock {\em Lectures on the geometry of manifolds}.
\newblock World Scientific Publishing Co. Pte. Ltd., New Jersey, third edition,
  2021.

\bibitem{Orlandi}
G.~Orlandi.
\newblock Asymptotic behavior of the {G}inzburg-{L}andau functional on complex
  line bundles over compact {R}iemann surfaces.
\newblock {\em Rev. Math. Phys.}, 08(03):457--486, 1996.

\bibitem{PacardRiviere-book}
F.~Pacard and T.~Rivi\`ere.
\newblock {\em Linear and nonlinear aspects of vortices}, volume~39 of {\em
  Progress in Nonlinear Differential Equations and their Applications}.
\newblock Birkh\"{a}user Boston, Inc., Boston, MA, 2000.
\newblock The Ginzburg-Landau model.

\bibitem{Palais}
R.~S. Palais.
\newblock {\em Foundations of Global Non-Linear Analysis}.
\newblock W.A. Benjamin, Inc., New York, 1968.

\bibitem{ParisePigatiStern}
D.~Parise, A.~Pigati, and D.~Stern.
\newblock Convergence of the self-dual ${{\rm U}}(1)$-{Y}ang-{M}ills-{H}iggs
  energies to the $(n-2)$-area functional.
\newblock Preprint arXiv~2103.14615, 2021.

\bibitem{PigatiStern}
A.~Pigati and D.~Stern.
\newblock Minimal submanifolds from the abelian {H}iggs model.
\newblock {\em Invent. Math.}, 223(3):1027--1095, 2021.

\bibitem{Qing}
J.~Qing.
\newblock Renormalized energy for {G}inzburg-{L}andau vortices on closed
  surfaces.
\newblock {\em Math. Z.}, 225(1):1--34, 1997.

\bibitem{SandierSerfaty-book}
{\'E}.~Sandier and S.~Serfaty.
\newblock {\em Vortices in the magnetic {G}inzburg-{L}andau model}.
\newblock Progress in Nonlinear Differential Equations and their Applications,
  70. Birkh\"auser Boston, Inc., Boston, MA, 2007.

\bibitem{SchwarzG}
G.~Schwarz.
\newblock {\em Hodge Decomposition - A Method for Solving Boundary Value
  Problems}.
\newblock Lecture notes in mathematics (Springer-Verlag), 1607. Springer-Verlag
  Berlin Heidelberg, 1995.

\bibitem{Scott}
C.~Scott.
\newblock ${L}^p$ theory of differential forms on manifolds.
\newblock {\em Trans. Amer. Math. Soc.}, 347(6):2075--2096, 1995.

\bibitem{Simon-GMT}
L.~Simon.
\newblock {\em {L}ectures in {G}eometric {M}easure {T}heory}.
\newblock Centre for Mathematical Analysis, Australian National University,
  Canberra, 1984.

\bibitem{Spivak2}
M.~Spivak.
\newblock {\em A Comprhensive Introduction to Differential Geometry}, volume~2.
\newblock Publish or Perish, Houston, Texas, third edition, 1999.

\bibitem{SpruckYang}
Joel Spruck and Yi~Song Yang.
\newblock Regular stationary solutions of the cylindrically symmetric
  {E}instein-matter-gauge equations.
\newblock {\em J. Math. Anal. Appl.}, 195(1):160--190, 1995.

\bibitem{Stern2021}
D.~Stern.
\newblock Existence and limiting behavior of min--max solutions of the
  ginzburg--landau equations on compact manifolds.
\newblock {\em J. Diff. Geom.}, 118(2), 2021.

\bibitem{Tinkham}
Michael Tinkham.
\newblock {\em Introduction to Superconductivity}.
\newblock Dover Publications, 2 edition, 2004.

\bibitem{UhlenbeckYau}
K.~Uhlenbeck and S.~T. Yau.
\newblock On the existence of {H}ermitian-{Y}ang-{M}ills connections in stable
  vector bundles.
\newblock {\em Comm. Pure App. Mat.}, 39(S1):S257--S293, 1986.

\bibitem{Warner}
F.W. Warner.
\newblock {\em Foundations of differentiable manifolds and {L}ie groups}.
\newblock Graduate Texts in Mathematics. Springer, New York, NY, 1983.

\bibitem{Wehrheim}
K.~Wehrheim.
\newblock {\em {U}hlenbeck compactness}.
\newblock EMS Series of Lectures in Mathematic. European Mathematical Society,
  2004.

\end{thebibliography}

\Addresses
\end{document}